  \documentclass[12pt, reqno]{amsart}
\makeatletter
\DeclareMathAlphabet{\mathpzc}{OT1}{pzc}{m}{it}
\usepackage[margin=2cm]{geometry}
\usepackage{amsmath, amssymb, amsfonts, amstext, verbatim, amsthm, mathrsfs}
\usepackage[mathcal]{eucal}
\usepackage{microtype}
\usepackage{epsfig}
\usepackage[all,arc,knot,poly]{xy}
\usepackage{graphicx}
\usepackage[usenames]{color}
\usepackage{aliascnt} 
\usepackage{tikz}
\usetikzlibrary{decorations.pathreplacing,cd}
\usetikzlibrary{matrix,arrows}
\usepackage[colorlinks=true,linkcolor=blue,citecolor=blue,urlcolor=blue,citebordercolor={0 0 1},urlbordercolor={0 0 1},linkbordercolor={0 0 1}]{hyperref}
\usepackage{enumerate}
\usepackage{xspace}

\renewcommand{\omit}[1]

\theoremstyle{plain}

\newcommand{\refnewtheoremn}[4]{
\newaliascnt{#1}{#2}
\newtheorem{#1}[#1]{#3}
\aliascntresetthe{#1}
\expandafter\providecommand\csname #1autorefname\endcsname{#4}}

\newcommand{\refnewtheorem}[3]{\refnewtheoremn{#1}{#2}{#3}{#3}}

\def\makeCal#1{
\expandafter\newcommand\csname c#1\endcsname{\mathcal{#1}}}
\def\makeBB#1{
\expandafter\newcommand\csname b#1\endcsname{\mathbb{#1}}}
\def\makeFrak#1{
\expandafter\newcommand\csname f#1\endcsname{\mathfrak{#1}}}
\count@=0
\loop
\advance\count@ 1
\edef\y{\@Alph\count@}
\expandafter\makeCal\y
\expandafter\makeBB\y
\expandafter\makeFrak\y
\ifnum\count@<26
\repeat

\newtheorem{thm}{Theorem}[section]
\newtheorem{theorem}{Theorem}[section]

\refnewtheorem{lemma}{thm}{Lemma}
\refnewtheorem{assumptions}{thm}{Assumptions}
\refnewtheorem{examples}{thm}{Examples}
\refnewtheorem{claim}{thm}{Claim}
\refnewtheorem{cor}{thm}{Corollary}
\refnewtheorem{conj}{thm}{Conjecture}
\refnewtheorem{prop}{thm}{Proposition}
\refnewtheorem{proposition}{thm}{Proposition}
\refnewtheorem{quest}{thm}{Question}
\refnewtheorem{assumption}{thm}{Assumption}
\refnewtheorem{problem}{thm}{Problem}

\theoremstyle{definition}
\refnewtheorem{remark}{thm}{Remark}
\refnewtheorem{remarks}{thm}{Remarks}
\refnewtheorem{defn}{thm}{Definition}
\refnewtheorem{definition}{thm}{Definition}
\refnewtheorem{notn}{thm}{Notation}
\refnewtheorem{const}{thm}{Construction}
\refnewtheorem{example}{thm}{Example}
\refnewtheorem{fact}{thm}{Fact}

\newcommand {\Hom}{\operatorname{Hom}}

\newcommand{\PGL}{\operatorname{PGL}}
\newcommand{\Quad}{\operatorname{Quad}}
\newcommand{\Pic}{\operatorname{Pic}}

\newcommand{\into}{\hookrightarrow}
\newcommand {\<}{\langle}
\renewcommand {\>}{\rangle}
\newcommand{\isom}{\cong}
\newcommand{\half}{\tfrac{1}{2}}
\newcommand{\tensor}{\otimes}

\newcommand{\rr}{r}

\newcommand{\mat}[4]{\begin{pmatrix}#1&#2\\#3&#4\end{pmatrix}}
\newcommand{\constant}{\operatorname{constant}}

\renewcommand{\O}{\mathcal{O}}
\renewcommand{\Im}{\operatorname{Im}}
\renewcommand{\Re}{\operatorname{Re}}
\newcommand{\CR}{\operatorname{CR}}

\newcommand{\e}{\varepsilon}
\newcommand{\MM}{M}

\begin{document}

\title{On the monodromy of the deformed cubic oscillator}
\author{Tom Bridgeland with an appendix by Davide Masoero}

\date{}

\begin{abstract}{We study  a second-order linear differential equation known as  the deformed cubic oscillator, whose isomonodromic deformations are controlled by the first Painlev{\'e} equation. We use the generalised monodromy map for this equation to give solutions to the Riemann-Hilbert problems of  \cite{RHDT}  arising from  the Donaldson-Thomas theory of  the A$_2$ quiver. These are the first known solutions to such  problems beyond the uncoupled case. The appendix by Davide Masoero contains a WKB analysis of the asymptotics of the monodromy map.}
\end{abstract}

\maketitle


\section{Introduction}

In this paper we study  the generalised monodromy map for a second-order linear differential equation known as  the deformed cubic oscillator. Our motivation  derives from a class of Riemann-Hilbert problems arising naturally in Donaldson-Thomas theory \cite{RHDT}, but we hope that our results will be of independent interest. We also suspect that they  can be substantially generalized.

\subsection{Deformed cubic oscillator}

Consider the second-order linear differential equation
\begin{equation}\label{de} y''(x)=Q(x,\hbar) \cdot y(x), \qquad Q(x,\hbar)=\hbar^{-2}\cdot {Q_0(x)}+\hbar^{-1}\cdot Q_1(x) +  Q_2(x),\end{equation}
where primes denote differentiation with respect to the complex variable $x\in \bC$, and the terms in the potential $Q(x,\hbar)$ are
\begin{equation}
\label{pot}Q_0(x)=x^3+ax+b, \qquad Q_1(x)=\frac{p}{x-q}+r, \qquad
Q_2(x)=\frac{3}{4(x-q)^2}+\frac{r}{2p(x-q)}+\frac{r^2}{4p^2}.\end{equation}
We view the equation \eqref{de}  as being specified by a point of the complex manifold \begin{equation}\MM=\bigg\{(a,b,q,p,r)\in \bC^5: p^2=q^3+aq+b\text{ and } 4a^3+27b^2 \neq 0, \ p\neq 0\bigg\},\end{equation}
together with a nonzero complex number $\hbar\in \bC^*$ which for now we will consider to be  fixed. We also  introduce the complex manifold
\begin{equation}S=\big\{(a,b)\in \bC^2: 4a^3+27b^2\neq 0\big\},\end{equation}
and the obvious projection map $\pi\colon \MM\to S$.

\begin{remark}\label{cobblers}The author's  interest in this topic stems from the study of a class of Riemann-Hilbert problems arising in Donaldson-Thomas theory \cite{RHDT,RHDT2}. These problems are specified by a stability condition on a CY$_3$ triangulated category, and involve maps from the complex plane to an algebraic torus with prescribed discontinuities along a  collection of rays. In this context the space $S$ arises as (a discrete quotient of) the space of stability conditions on the CY$_3$ triangulated category associated to the A$_2$ quiver \cite{BQS}. As we explain below,  the monodromy map for the equation \eqref{de} gives solutions to the corresponding Riemann-Hilbert problems. These are the first examples  of  such Riemann-Hilbert problems (beyond the uncoupled case)  for which a complete solution  is known.
\end{remark}

The expression $Q_2(x)$  appearing in \eqref{pot} is chosen to ensure that the  point $x=q$ is an apparent singularity of the equation \eqref{de}: analytically continuing any solution around this point changes its sign. Thus the generalised monodromy  of the equation consists only of the Stokes data at the irregular singularity $x=\infty$. As we recall below, this defines a point of the
 quotient space
\begin{equation}
\label{vvv}V=\Big\{\psi\colon \bZ/5\bZ\to \bP^1: \psi(i+1)\neq \psi(i)\text{ for all }i\in \bZ/5\bZ\Big\} \Big / \PGL_2,\end{equation}
which is easily seen to be a two-dimensional complex manifold.
We thus obtain a holomorphic monodromy map
\begin{equation}
\label{mon} F(\hbar)\colon \MM \to V.\end{equation}
More precisely, this map depends on a labelling of the Stokes sectors for the equation \eqref{de}, which in concrete terms amounts to a choice of fifth root of $\hbar^2$.

\begin{remark}
\label{rem}
Note that  the two points of the space $\MM$
\[(a,b,q,p,r), \qquad \Big(a,b+r\hbar+\frac{r^2\hbar^2}{4p^2},q,p+\frac{r\hbar}{2p},0\Big),\]
determine the same equation \eqref{de}. Thus  for many purposes we can  reduce to the situation when $r=0$.  In that  case \eqref{de} coincides, up to trivial changes of variables, with an equation which has been studied in connection with the first Painlev{\'e} equation for many years (see  \cite[Chapter 4]{KT} and \cite{myphd} for references).
 Nonetheless, it will be important in what follows to consider  the full form \eqref{pot} of the potential, so that the fibres of the map $\pi\colon \MM\to S$ are half-dimensional, and have the same dimension as the monodromy manifold $V$.
\end{remark}

Each point $s=(a,b)\in  S$ determines a meromorphic  quadratic differential on $\bP^1$
\begin{equation}
\label{diff}Q_0(x) \, dx^{\tensor 2} = (x^3+ax+b) \, dx^{\tensor 2}\end{equation}
with a single pole of order seven at $x=\infty$.
There is a corresponding branched double cover \begin{equation}
\label{double}p\colon X_s\to \bP^1,\end{equation} which is the projectivization of the non-singular plane cubic
\begin{equation}
\label{app}X_{s}^\circ=\big\{(x,y)\in \bC^2: y^2=x^3+ax+b\big\}.\end{equation}
  We also introduce the associated homology groups\begin{equation}
\label{homo}\Gamma_s=H_1(X_s,\bZ)\isom \bZ^{\oplus 2},\end{equation}
which we equip with the standard  skew-symmetric intersection form $\<-,-\>$.

\begin{remark}
\label{genival}
Given an integer $g\geq 0$, and a non-empty collection of  integers $m=\{m_1,\cdots,m_d\}$, with each $m_i\geq 2$, there is a complex orbifold $\Quad(g,m)$ parameterizing equivalence-classes of pairs $(S,\phi)$, where $S$ is a  compact Riemann surface of genus $g$, and $\phi$ is a meromorphic quadratic differential on $S$,  having simple zeroes, and poles of the given  orders $m_i$.
It is shown in \cite{BS} that  to such data $(g,m)$ there is naturally associated a CY$_3$ triangulated category $\cD(g,m)$, and that  the space $\Quad(g,m)$ arises as a discrete quotient of the space of stability conditions on $\cD(g,m)$.\footnote{In fact this is a slight over-simplification: it is necessary  to slightly enlarge the space $\Quad(g,m)$ by allowing the zeroes of $\phi$ to collide with any of the poles of order $m_i=2$: see \cite[Section 6]{BS} for details.} We expect that the story we describe here (which corresponds to the case $g=0$, $m=\{7\}$) extends to this more general situation, although we do not yet understand the full details of this.
\end{remark}


 Since the dimensions of the spaces $\MM$ and $V$ are four and two respectively,  we might expect the derivative of the monodromy map \eqref{mon}  to have a two-dimensional kernel, and indeed in Section 2 we  show that the map $F(\hbar)$ is invariant under the two flows \begin{equation}
\label{fir1}-\frac{1}{\hbar}\frac{\partial}{\partial r}+\bigg(\frac{\partial}{\partial b}+\frac{1}{2p}\frac{\partial}{\partial p} +\frac{r}{2p^2} \frac{\partial}{\partial r}\bigg),\end{equation}
\begin{equation} \label{sec1}-\frac{2p}{\hbar} \frac{\partial}{\partial q}-\frac{3q^2+a}{\hbar}\frac{\partial}{\partial p}+\bigg(\frac{\partial}{\partial a}-q\frac{\partial}{\partial b}-\frac{r}{p}\frac{\partial}{\partial q}- \frac{r(3q^2+a)}{2p^2}\frac{\partial}{\partial p}-\frac{r^2}{2p^3} (3q^2+a) \frac{\partial}{\partial r}\bigg).\end{equation}
Since the sub-bundle of the  tangent bundle  spanned by these flows is everywhere transverse to the fibres of  the map $\pi\colon \MM\to S$, it defines an Ehresmann connection on this map, which we will refer to as the isomonodromy connection.

It follows from the existence of the isomonodromy connection that the monodromy map $F(\hbar)$ restricts to give local isomorphisms
\begin{equation}\label{fib}F(\hbar)\colon \MM_s\to V,\end{equation}
between the fibres $\MM_s=\pi^{-1}(s)$ of the projection $\pi\colon \MM\to S$, and the monodromy manifold $V$.
What is  interesting for us is that, as we  will explain below,  both sides of the map \eqref{fib} can be more-or-less identified with the algebraic torus \begin{equation}\label{tori}\bT_s=H^1(X_s,\bC^*)\isom \Hom_\bZ(\Gamma_s,\bC^*)\isom (\bC^*)^2.\end{equation}
Using these identifications allows us to do two things:\smallskip

\begin{itemize}
\item[(i)] We can view the isomonodromy connection  as an Ehresmann connection on the bundle over $S$ whose fibres are the algebraic  tori $\bT_s$. We give a Hamiltonian form for this connection in  Theorem \ref{one}, and show that it gives an example of a Joyce structure in the sense of \cite{RHDT2}. This structure then induces a flat, torsion-free connection on the tangent bundle of $S$, which is described by Theorem \ref{one2}. \smallskip
\item[(ii)] For each point $s\in S$, we can view the monodromy map \eqref{fib}  as  giving a partially-defined automorphism of the algebraic torus $\bT_s$, depending  in a piecewise holomorphic way on the parameter   $\hbar\in \bC^*$. This allows us in  to solve a family of Riemann-Hilbert  problems of the type discussed in \cite{RHDT,RHDT2}. A precise summary of this  claim appears as Theorem \ref{two} below.
\end{itemize}

In the next two subsections we will explain these two points in more detail.

\subsection{Isomonodromy flows}

The homology groups \eqref{homo} form a local system of lattices over $S$, which induces the Gauss-Manin connection on the vector bundle on $S$ whose fibres are the  spaces $H_1(X_s,\bC)$.  In concrete terms, we can construct a basis of homology classes by taking inverse images under the double cover \eqref{double} of paths in $\bC$ connecting the zeroes of $Q_0(x)$. The Gauss-Manin connection is then obtained by keeping these paths locally constant as $Q_0(x)$ varies.

Let us choose a basis $(\gamma_1,\gamma_2)\subset \Gamma_s$
at some point $s\in S$, and extend it to nearby fibres using the Gauss-Manin connection. A particular case of a general result of \cite{BS} shows that  
 the expressions \begin{equation}
\label{zz}z_i=\int_{\gamma_i} \sqrt{Q_0(x)}\, dx=\int_{\gamma_i} y \, dx,\end{equation}
define a local system of co-ordinates  $(z_1,z_2)$ on the manifold $S$.

Consider  the bundle $\pi \colon \bT\to S$ whose fibres are the tori \eqref{tori}.
There are obvious  local co-ordinates $(\theta_1,\theta_2)$ on the fibres $\bT_s$  obtained by writing
\[\xi(\gamma_i)=\xi_i=\exp( \theta_i)\in \bC^* \qquad \big(\xi\colon \Gamma_s\to \bC^*\big)\in \bT_s,\]
and we therefore obtain local co-ordinates $(z_1,z_2,\theta_2,\theta_2)$ on the total space $\bT$.

 In Section \ref{abhol} we introduce a holomorphic map $\Theta\colon \MM\to \bT$, commuting with the two projections to $S$,
and given in local co-ordinates (up to multiples of $\pi i$) by 
\begin{equation}
\label{xi}\theta_i= -\int_{\gamma_i}\frac{Q_1(x)\, dx}{2\sqrt{Q_0(x)}}=-\int_{\gamma_i} \bigg(\frac{p}{x-q} + r\bigg) \frac{dx}{2y}.\end{equation}
This expression  is familiar in WKB analysis as the constant term in the expansion of the Voros symbols (see Section \ref{wkb} below).

In Section \ref{abhol} we give a more geometric description of the map $\Theta$. For each  point $s\in S$ we  show that there is a natural embedding of  the fibre $\MM_s=\pi^{-1}(s)$  into the space of pairs $(L,\nabla)$ consisting of a holomorphic line bundle $L$ on the  elliptic curve $X_s$, equipped with a holomorphic connection $\nabla$. The map $\Theta$ then sends such a pair $(L,\nabla)$ to its holonomy, viewed as an element of $\bT_s$.

We shall refer to the map $\Theta$ as the abelian holonomy map. It follows from the above description that it is an open embedding.
We can use it   to push forward the isomonodromy flows \eqref{fir1}-\eqref{sec1}. This gives an Ehresmann connection on a dense open subset of the bundle  $\pi \colon \bT\to S$. The following result shows that this connection has precisely the form considered in \cite{RHDT2}.

\begin{thm}
\label{one}
When written in the co-ordinates $(z_1,z_2,\theta_1,\theta_2)$, the push-forward of the isomonodromy flows \eqref{fir1}-\eqref{sec1} along the map $\Theta\colon \MM\to \bT$ take the Hamiltonian form
\begin{equation}
\label{floww}\frac{\partial}{\partial z_i} + \frac{1}{\hbar} \cdot \frac{\partial}{\partial \theta_i} +\frac{\partial^2 J}{\partial \theta_i\partial \theta_1}\cdot \frac{\partial}{\partial \theta_2}-\frac{\partial^2 J}{\partial \theta_i\partial \theta_2}\cdot \frac{\partial}{\partial \theta_1}, \end{equation}
where $J\colon \bT\to \bC$ is a meromorphic function with no poles on  the locus $\theta_1=\theta_2=0$. When pulled-back to $M$ using the abelian holonomy map it is given by the expression
\[\frac{1}{2\pi i}\cdot (J\circ\Theta)= -\frac{ 2ap^2+3p(3b-2aq)r+(6aq^2-9bq+4a^2)r^2 - 2apr^3}{4(4a^3+27b^2)p}.\]
\end{thm}

The pencil of flat non-linear connections \eqref{floww} defines a geometric structure on the space $S$ which is studied  in detail in \cite{RHDT2} and called there a Joyce structure. The author expects such  structures to exist on spaces of stability conditions of  CY$_3$ triangulated categories in much greater generality, and Theorem \ref{one} provides an interesting first example. 
We call the function $J$ the Joyce function; some of its basic properties are discussed in Section \ref{bp} below. 

The Joyce function $J=J(z_1,z_2,\theta_1,\theta_2)$ is easily seen to be odd in the variables $\theta_1,\theta_2$, and it follows that the flows \eqref{floww} preserve the section of the bundle $\pi\colon \bT\to S$ defined by setting  $\theta_1=\theta_2=0$. They therefore induce a linear connection on the normal bundle to this section, which can in turn be identified with the tangent bundle to $S$ via the map
\[\frac{\partial }{\partial \theta_i}\mapsto \frac{\partial}{\partial z_i}.\] In this way we obtain a linear, flat, torsion-free connection
 on the tangent bundle of $S$, given explicitly by the formula
 \[\nabla_{\frac{\partial}{\partial z_i}}\Big(\frac{\partial}{\partial z_j}\Big)=  \frac{\partial^3 J}{\partial \theta_i \, \partial \theta_j \, \partial \theta_2}\Big|_{\theta=0} \cdot \frac{\partial}{\partial z_1}-\frac{\partial^3 J}{\partial \theta_i \, \partial \theta_j \, \partial \theta_1}\Big|_{\theta=0} \cdot \frac{\partial}{\partial z_2}.\]
 We call it the linear Joyce connection. In Section \ref{wwp} we prove

\begin{thm}
\label{one2}
The functions $(a,b)$ are flat co-ordinates for the linear Joyce connection.
\end{thm} 

We will comment on the significance of this result after the statement of Theorem \ref{two} below.

\subsection{Riemann-Hilbert problem}
Consider now the right-hand side of the monodromy map \eqref{fib}. It is well known that the manifold $V$ has a system of birational co-ordinate systems
\begin{equation}
\label{coord}X_T\colon V\dashrightarrow (\bC^*)^2,\end{equation}
indexed by the triangulations $T$ of a regular pentagon.
These co-ordinate systems are usually called Fock-Goncharov co-ordinates, since they appear in a much more general context  in \cite{FG}. We recall their definition in Section \ref{soln2}. The co-ordinates corresponding to  different triangulations are related by post-composition with explicit birational automorphisms of $(\bC^*)^2$.

Let us fix a point $(a,b,q,p,r)\in \MM$. For generic $\hbar\in \bC^*$, the horizontal trajectory structure of the quadratic differential \begin{equation}\label{hq}\hbar^{-2}\cdot Q_0(x) \, dx^{\tensor 2}= \hbar^{-2}\cdot (x^3+ax+b) \, dx^{\tensor 2}\end{equation}  determines a  triangulation $T(\hbar)$ of a regular pentagon. This triangulation  is  well-defined when $\hbar\in \bC^*$  lies in the complement of the  finitely-many rays on which the quadratic differential \eqref{hq} has a finite-length horizontal trajectory.  Following \cite{GMN2} we refer to it as  the WKB triangulation.

When $T=T(\hbar)$ is a WKB triangulation,  the algebraic torus appearing on the right-hand side of \eqref{coord} is naturally identified with the torus $\bT_s$ associated to the point $s=(a,b)\in S$.
Keeping the point  $(a,b,q,p,r)\in \MM$ fixed, let us now consider the map
\begin{equation}\label{map}X \colon \bC^*\to \bT_s,\qquad X(\hbar)=X_{T(\hbar)}\big(F(\hbar)(a,b,q,p,r)\big),\end{equation}
which sends a point $\hbar\in \bC^*$ to the Fock-Goncharov co-ordinates of the monodromy of the equation \eqref{de} with respect to the WKB triangulation $T(\hbar)$. Using our  chosen basis $(\gamma_1,\gamma_2)$ of $\Gamma_s$ we can identify $\bT_s$ with $(\bC^*)^2$ and decompose $X(\hbar)$ into its components \[X(\hbar)=(x_1(\hbar),x_2(\hbar))\in (\bC^*)^2, \qquad x_i(\hbar)=X(\hbar)(\gamma_i)\in \bC^*.\]

The map \eqref{map} has three important properties, which we  explain in detail in Section \ref{soln2}:
\begin{itemize}
\item[(i)]
As $\hbar\in \bC^*$ crosses a ray where the differential \eqref{hq} has a finite-length horizontal trajectory,  the WKB triangulation $T(\hbar)$ changes, and the map $X(\hbar)$ undergoes a discontinuous jump obtained by post-composing with an explicit birational transformation of the torus $\bT_s$.\smallskip

\item[(ii)] The WKB approximation can be used to show that  as $\hbar\to 0$ along a ray in $\bC^*$
\[x_i(\hbar)\cdot  \exp\Big(\frac{z_i}{\hbar}-\theta_i\Big)\to 1\]
 where the $\theta_i$ are given by \eqref{xi}. This statement is proved in  the Appendix.\smallskip
 
\item[(iii)] A homogeneity property of the potential \eqref{pot}  allows us to conclude  that  as $\hbar\to \infty$ the functions $x_i(\hbar)$ have a well-defined limit.
\end{itemize}

These properties are exactly  the conditions required for the map $X(\hbar)$  to give a 
 solution to one of the Riemann-Hilbert problems defined in  \cite{RHDT}. 
To state this more precisely, recall first
  the definition of a finite BPS structure  $(\Gamma,Z,\Omega)$ from \cite{RHDT}. It consists of

 \begin{itemize}
\item[(a)] A finite-rank free abelian group $\Gamma\isom \bZ^{\oplus n}$, equipped with a skew-symmetric form \[\<-,-\>\colon \Gamma \times \Gamma \to \bZ,\]

\item[(b)] A homomorphism of abelian groups
$Z\colon \Gamma\to \bC$,\smallskip

\item[(c)] A map of sets
$\Omega\colon \Gamma\to \bQ$ such that $\Omega(\gamma)=0$ for all but finitely-many elements $\gamma\in \Gamma$, and satisfying the symmetry property $\Omega(-\gamma)=\Omega(\gamma)$.
\end{itemize}
The group $\Gamma$ is called the charge lattice, and the homomorphism $Z$ the central charge. The rational numbers $\Omega(\gamma)$ are called the BPS invariants.

As we explain in Section 6, each point $s\in S$ determines such a BPS structure $(\Gamma_s,Z_s,\Omega_s)$. 
The charge lattice  is the homology group $\Gamma_s=H_1(X_s,\bZ)$ equipped  with its intersection form $\<-,-\>$. The central charge $Z_s\colon \Gamma_s\to \bC$  is defined by the formula 
\[Z_s(\gamma)=\int_{\gamma} \sqrt{Q_0(x)} \, dx\in \bC.\]
Assuming that the point $s\in S$ is generic, in the sense that the image of $Z_s$ is not contained in a line, the 
BPS invariants $\Omega_s(\gamma)\in \bZ$ count the number of finite-length trajectories  of the differential \eqref{diff} whose lifts to $X_s$ define the given class $\gamma\in \Gamma$.

It is explained in \cite{RHDT}  how to associate a Riemann-Hilbert problem to a finite BPS structure. This problem involves piecewise holomorphic (or meromorphic) maps into the associated algebraic torus $\bT$, and depends on an element $\xi\in \bT$ called the constant term. Our final result is

\begin{thm}
\label{two}
Take a point $(a,b,q,p,r)\in \MM$ and let $(\Gamma_s,Z_s,\Omega_s)$ be the BPS structure determined by the corresponding point $(a,b)\in S$. 
Then the  map  \eqref{map} gives a meromorphic solution to the Riemann-Hilbert problem for this BPS structure, with constant term $\xi\in \bT_s$ defined by \eqref{xi}.
\end{thm}


Let us return to the abstract context of Remark \ref{cobblers}, where the space $S$ 
appears as a discrete quotient of the space of stability conditions on the CY$_3$ triangulated category associated to the A$_2$ quiver. The BPS structures $(\Gamma_s,Z_s,\Omega_s)$ considered above then coincide with those defined by the Donaldson-Thomas theory of these stability conditions. Thus Theorem \ref{two} gives solutions to the Riemann-Hilbert problems defined by the A$_2$ quiver. It is worth noting in this context that  the space $V$ also has a natural representation-theoretic meaning, since it coincides with the cluster Poisson variety.

When viewed  from this abstract  point of view, the only natural local co-ordinates on  the stability space $S$ are the central charge co-ordinates $(z_1,z_2)$. The point of Theorem \ref{one2}  is that it gives a way to derive the flat structure on $S$ whose co-ordinates are $(a,b)$  from purely abstract considerations: one first solves the Riemann-Hilbert problem defined by the Donaldson-Thomas invariants to obtain the pencil of non-linear connections of Theorem \ref{one}, and then differentiates to obtain the linear connection of Theorem \ref{one2}. Unfortunately there is one crucial missing link in this chain of reasoning:  we currently have no characterisation or uniqueness result for the solution of Theorem \ref{two}.

\begin{remark}The statement of Theorem \ref{two} takes direct inspiration from the  work of Gaiotto, Moore and Neitzke \cite{GMN1,GMN2}. In particular, the use of the Fock-Goncharov co-ordinates for the WKB triangulation, and the resulting discontinuities in the map \eqref{map} are exactly as described in \cite[Section 7]{GMN2}. It is important to note however that the picture described here is strictly different to that of \cite{GMN2}. Although Gaiotto, Moore and Neitzke start with the same data of a BPS structure, they consider  a somewhat different Riemann-Hilbert problem, which has non-holomorphic dependence on the central charge $Z$. Instead of our monodromy map $F$, they solve their Riemann-Hilbert problem using a $C^{\infty}$ isomorphism  between the moduli spaces of irregular Higgs bundles and the wild character variety  $V$.  In physical terms what we are considering here is the conformal limit \cite{G} of their story.\end{remark}

\begin{remark} The constructions of this paper are closely related to the  ODE/IM correspondence. The author is unfortunately not qualified to describe this link in any detail. It is explained in \cite{G} and \cite[Appendix E]{GMN1}  that the Riemann-Hilbert problems considered here  can be solved, at least formally, by an integral equation known in the integrable systems literature as the Thermodynamic Bethe Ansatz (TBA). The fact that these TBA  equations  also appear in  the  analysis of Stokes data of ordinary differential equations goes back in some form to work of Sibuya and Voros, but was made more precise in the work of Dorey, Dunning, Tateo and others.  We refer the reader to  \cite{DDT} for a review of the ODE/TBA correspondence, and to \cite{Mar,Mas1} for  more recent papers which deal specifically with the cubic oscillator.\end{remark}

\subsection*{Acknowledgements} As explained above, this paper  owes a significant debt to the work of Gaiotto, Moore and Neitzke \cite{GMN2}. I have also benefitted from  useful conversations with  Dylan Allegretti, Kohei Iwaki, Dima Korotkin, Davide Masoero, Andy Neitzke and Tom Sutherland. 


\section{The deformed cubic oscillator}

In this section we discuss the generalised monodromy data of the deformed cubic oscillator equation \eqref{de}. We explain why this consists entirely of the Stokes data at $x=\infty$ and recall how this is parameterised by collections of subdominant solutions. We then derive the isomonodromy flow in the form \eqref{fir1}-\eqref{sec1}. This section contains only very minor extensions of previously known results. Similar material can be found for example in  \cite{piwkb,myphd}.

\subsection{Apparent singularity}

The first claim is that for any $\hbar\in \bC^*$ and $(a,b,q,p,r)\in \MM$ the equation \eqref{de} has an apparent singularity at $x=q$. By this we mean that the analytic continuation of any solution around this point has the effect of multiplying it by $\pm 1$ (and in our case the sign is $-1$). This statement follows  immediately from the identity
\[\bigg(\frac{p}{\hbar}+\frac{r}{2p}\bigg)^2=\frac{q^3 + aq + b}{\hbar^2}+\frac{r}{\hbar}  +\frac{r^2}{4p^2},\]
and the following well-known Lemma.

\begin{lemma}
Fix a point $q\in \bC$ and suppose that $Q(x)$ is a meromorphic function having a pole  at $x=q$. Suppose further that the Laurent expansion of $Q(x)$ at this point takes the form
\[Q(x)=\frac{3}{4(x-q)^2} +\frac{u}{x-q}+v +O(x-q) .\]
Then the differential equation
\begin{equation}
\label{anglesey}y''(x)=Q(x)\cdot y(x)\end{equation}
has an apparent singularity at $x=q$ precisely if the relation $u^2=v$ holds.
\end{lemma}

\begin{proof}
This is a standard calculation using the Frobenius method, and we just give  a sketch. We  look for a solution to \eqref{anglesey} of the form
\begin{equation}
\label{test}y(x)=\sum_{i=0}^\infty c_i (x-q)^{\lambda+i},\qquad c_i\in \bC,\end{equation}
with $c_0\neq 0$ and  $\lambda\in \bC$. This leads to a recurrence relation \begin{equation}
\label{rec}(\lambda+i)(\lambda+i-1) c_i =\tfrac{3}{4} c_i + u c_{i-1} + v c_{i-2} + \cdots\end{equation}
which is valid for all $i\geq 0$ if we define $c_i=0$ for $i<0$. In particular, taking $i=0,1,2$ we obtain the relations
\begin{equation}
\label{p}(\lambda^2-\lambda - \tfrac{3}{4})c_0=0, \qquad (\lambda^2+\lambda - \tfrac{3}{4}) c_1=u c_0, \qquad (\lambda^2+3\lambda + \tfrac{5}{4})c_2=uc_1+vc_0.\end{equation}
The first of these gives the indicial equation, whose roots are  $\lambda=\tfrac{3}{2}$ and $\lambda=-\tfrac{1}{2}$. When $\lambda=\tfrac{3}{2}$ it is easy to see that the recursion \eqref{rec} has a unique solution for each choice of $c_0$, and standard theory then shows that \eqref{test}  defines  a double-valued solution to \eqref{anglesey} near $x=q$.

When $\lambda=-\tfrac{1}{2}$ the second equation of \eqref{p} gives $c_1=-uc_0$, and the third equation then implies the stated condition $u^2=v$. Assuming this, the recursion again has a unique solution for each choice of $c_0$, and we obtain another double-valued solution to \eqref{anglesey} near $x=q$. The form of these two solutions shows that \eqref{anglesey} has an apparent singularity.
If the relation $u^2=v$ does not hold, standard theory shows that the second solution to \eqref{anglesey} has a logarithmic term, and the solutions then exhibit non-trivial monodromy around the point $x=q$, which is therefore not an apparent singularity.
\end{proof}

\subsection{Stokes data}

The analysis of the last section shows that the monodromy data of the equation \eqref{de} consists only of the Stokes data at the irregular singularity $x=\infty$.
We now briefly recall how this is defined. A more detailed exposition of this material can be found for example in  \cite[Section 5]{AB}. 
The Stokes sectors are
the sectors in $\bC$ bounded by the asymptotic vertical directions of the quadratic differential
\[\hbar^{-2}\cdot Q_0(x) dx^{\tensor 2},\]
which are easily seen to be the rays passing through the fifth roots of $-\hbar^{2}$. General theory \cite{Sib} shows that in each Stokes sector
there is a unique subdominant solution to \eqref{de} up to scale, with the defining property that it exhibits exponential decay as $x\to \infty$ in the sector. Moreover, the subdominant solutions in neighbouring sectors are linearly independent.

 Since the space of solutions to the equation \eqref{de} is a two-dimensional complex vector space, the subdominant solutions define a collection of five  points of $\bP^1$, 
well-defined up to the diagonal action of $\PGL_2$,  with the property that each consecutive pair of points is distinct. These points are naturally indexed by the Stokes sectors of the equation, and hence by the fifth roots of $\hbar^2$. Choosing one such root we can identify this set with $\bZ/5\bZ$ and so obtain a point in the 
 quotient space
\begin{equation}
\label{vv}V=\Big\{\psi\colon \bZ/5\bZ\to \bP^1: \psi(i+1)\neq \psi(i)\text{ for all }i\in \bZ/5\bZ\Big\} \Big / \PGL_2,\end{equation}
which  is easily seen to be a two-dimensional complex manifold \cite{Go}.
We call the resulting  map
\[F(\hbar)\colon \MM\to V\]
 the monodromy map. Note however that this is a mild abuse of notation since $F(\hbar)$ really depends on a choice of fifth root of $\hbar^2$. The map $F(\hbar)$ is holomorphic because the subdominant solutions vary holomorphically with parameters \cite{HS,Sib}.
  
 \begin{remark}
 \label{euros}
 There is an obvious action of the group $\bZ/5\bZ$ on the space $V$ obtained by precomposing the map $\psi$ in \eqref{vv} with the translations $i\mapsto i+j$ of $\bZ/5\bZ$.  It is easy to check that it has exactly two  fixed points, represented by the cyclically-ordered 5-tuples of points of $\bP^1$ of the form $(0,1,\infty,x,x+1)$, with $x\in \bC$ a solution to the golden ratio equation $x^2+x-1=0$.
 One way to avoid the choice of fifth root of $\hbar^2$ when defining the monodromy map $F(\hbar)$ is to consider it as taking values  in the complex orbifold obtained by quotienting $V$ by this action.
  \end{remark}


\subsection{Isomonodromy flow}
The following result gives a pair of flows on the four-dimensional manifold $\MM$ along which the monodromy map $F(\hbar)$ is constant. 
\begin{prop}
\label{isom}
For a fixed $\hbar\in \bC^*$ the monodromy map $F(\hbar)$ is preserved by the flows 
 \begin{equation}
\label{fir}-\frac{1}{\hbar}\frac{\partial}{\partial r}+\bigg(\frac{\partial}{\partial b}+\frac{1}{2p}\frac{\partial}{\partial p} +\frac{r}{2p^2} \frac{\partial}{\partial r}\bigg),\end{equation}
\begin{equation} \label{sec}-\frac{2p}{\hbar} \frac{\partial}{\partial q}-\frac{3q^2+a}{\hbar}\frac{\partial}{\partial p}+\bigg(\frac{\partial}{\partial a}-q\frac{\partial}{\partial b}-\frac{r}{p}\frac{\partial}{\partial q}- \frac{r(3q^2+a)}{2p^2}\frac{\partial}{\partial p}-\frac{r^2}{2p^3} (3q^2+a) \frac{\partial}{\partial r}\bigg).\end{equation}
\end{prop}

\begin{proof}
A straightforward calculation which we leave to the reader shows that the first flow \eqref{fir} preserves the potential $Q(x,\hbar)$, and hence the equation \eqref{de}. We defer the proof that the second flow preserves the monodromy map to the next subsection.
\end{proof}

Note that the flows of Proposition \ref{isom} span a two-dimensional sub-bundle of the tangent bundle of $\MM$, which is everywhere transverse to the kernel of the derivative of the projection map $\pi\colon \MM\to S$. This is the condition that the sub-bundle defines an Ehresmann connection on this map. We call it the isomonodromy connection.

\begin{remark} \label{rem2}When $r=0$ the equation \eqref{de} reduces to the deformed cubic oscillator of \cite{myphd}, and the flow \eqref{sec} becomes
\begin{equation}
\label{simp}\frac{da}{dt}=1, \qquad \frac{db}{dt}=-q, \qquad \frac{dq}{dt}=-\frac{2p}{\hbar}, \qquad \frac{dp}{dt}=-\frac{3q^2+a}{\hbar}.\end{equation}
Let us briefly recall the well-known Hamiltonian description of this flow, and the link with Painlev{\'e} equations. Fix the parameter  $\hbar\in \bC^*$, and consider the space $\bC^4$ with co-ordinates $(a,b,q,p)$ equipped with  the symplectic form
\[\omega=da\wedge db +\hbar \cdot dq\wedge dp.\]
Then \eqref{simp} is the flow  defined by the Hamiltonian
 \[H(a,b,q,p)=q^3+aq+b-p^2.\]
Since $da/dt=1$ we can set $t=a$. The flow \eqref{simp}  then implies that
\[\hbar^2 \cdot \frac{d^2q}{dt^2}=-2 \hbar \cdot \frac{dp}{dt}=6q^2+2t,\]
which, after rescaling, becomes the first Painlev{\'e} equation.
\end{remark}

\subsection{Proof of the isomondromy property}

Let us complete the proof of Proposition \ref{isom}. We must just show that the second flow \eqref{sec}  preserves the Stokes data. 
\begin{proof} Let us fix $\hbar\in \bC^*$ and consider the potential $Q=Q(x)$ to be  also a function of a variable $t\in \bC$, in such a way that the derivative with respect to $t$ gives the flow \eqref{sec}.  The condition for the Stokes data to be constant \cite{Ueno} is the existence of an extended flat connection of the form
 \begin{equation}
 \label{spiro}\nabla=d-\mat{0}{1}{Q(x,t)}{0} dx - B(x,t) dt,\end{equation}
 with $B(x,t)$ a  meromorphic matrix-valued function. Let us make the ansatz
 \[B(x,t)=\mat{-\half A'}{A}{A Q-\half A''}{\half A'},\]
 for some function $A=A(x,t)$, where primes denote derivatives with respect to $x$. The flatness condition for the connection \eqref{spiro} then becomes
   \begin{equation}\label{basicA}\frac{\partial^3 A}{\partial x^3}-4Q \frac{\partial A}{\partial x}- 2 \frac{\partial Q}{\partial x} A + 2 \frac{\partial Q}{\partial t}=0,\end{equation}
 an equation which goes back at least to Fuchs. We now take
 $A=(x-q)^{-1}$.
 Writing out equation \eqref{basicA}  gives
 \[ \frac{4}{(x-q)^2} Q(x) - \frac{2}{x-q} Q'(x) + 2\dot{Q}(x)-\frac{6}{(x-q)^4} =0,\]
 where dots denote differentiation with respect to  $t$. In detail this is
 \[\frac{4}{\hbar^2(x-q)^2} (x^3+ax+b) - \frac{2}{\hbar^2 (x-q)} (3x^2+a) +\frac{2}{\hbar^2}(\dot{a}x+\dot{b})
 +\frac{4}{\hbar(x-q)^2}\Big( \frac{p}{x-q}+ r\Big)\]\[+\frac{2p}{\hbar(x-q)^3}+\frac{2}{\hbar}\Big( \frac{\dot{p}}{x-q}+ \dot{r}\Big)+ \frac{2p\dot{q}}{\hbar(x-q)^2}  +\frac{6}{(x-q)^4} +\frac{3\dot{q}}{(x-q)^3}-\frac{6}{(x-q)^4}\]\[+\frac{3r}{p(x-q)^3}+\frac{r^2}{p^2(x-q)^2} -\frac{r\dot{p}}{p^2(x-q)}+\frac{\dot{r}}{p(x-q)}+\frac{r\dot{q}}{p(x-q)^2} -\frac{r^2 \dot{p}}{p^3}+\frac{r \dot{r}}{p^2}=0.\]
 The expression on the left-hand side of this equation is a rational function of $x$, with possible poles only at $x=q$ and $x=\infty$. To show that it is zero we consider the terms in the Laurent expansion at each of these points, which are
 \[(x-q)^{-3}: \quad  \frac{4p}{\hbar}+\frac{2p}{\hbar}+3\dot{q}+\frac{3r}{p},\]
 \[(x-q)^{-2}: \quad \frac{4}{\hbar^2} (q^3+aq+b)+ \frac{4r}{\hbar} +\frac{2p \dot{q}}{\hbar}+\frac{r^2}{p^2} +\frac{r\dot{q}}{p}, \]    \[(x-q)^{-1}: \quad  \frac{4}{\hbar^2} (3q^2+a)-\frac{2}{\hbar^2} (3q^2+a)+\frac{2\dot{p}}{\hbar}-\frac{r \dot{p}}{p^2} +\frac{\dot{r}}{p},\]
       \[x^1 : \quad \frac{4}{\hbar^2} -\frac{6}{\hbar^2}+\frac{2}{\hbar^2}\dot{a},\qquad x^0: \quad \frac{8q}{\hbar^2}-\frac{6q}{\hbar^2} + \frac{2}{\hbar^2} \dot{b}+\frac{2\dot{r}}{\hbar}-\frac{r^2 \dot{p}}{p^3}+\frac{r \dot{r}}{p^2}.\]
These are all easily checked to vanish under the given flow
\[\dot{a}=1,\quad \dot{b}=-q,\quad \dot{q}=-\frac{2p}{\hbar}-\frac{r}{p}, \quad \dot{p}=-\frac{3q^2+a}{\hbar} -\frac{r(3q^2+a)}{2p^2},\quad   \dot{r}=-\frac{r^2}{2p^3} (3q^2+a),\] which completes the proof.
\end{proof}


\section{Periods and the abelian holonomy map}
\label{abhol}

In this section we first consider the period co-ordinates $(z_1,z_2)$  on the space $S$ and the relationship with the affine co-ordinates $(a,b)$. This is a standard calculation with Weierstrass elliptic functions. We then consider  the expression \eqref{xi} from the introduction and explain its conceptual meaning 
 in terms of the holonomy of abelian connections. The author learnt this interpretation from  \cite[Section 3]{LM}. 


\subsection{Weierstrass elliptic functions}
\label{wei}

In what follows we shall need  some basic and well known properties of the Weierstrass elliptic functions. These functions depend on a choice of lattice \[\Lambda=\bZ\omega_1\oplus \bZ\omega_2\subset \bC.\] We assume the generators $\omega_i$ are ordered so that $\Im(\omega_2/\omega_1)>0$. Proofs of the following claims can all be found for example in \cite[Chapter 20]{WW}, although the reader should note that the generators of $\Lambda$ are denoted there by $2\omega_i$.

The Weierstrass $\wp$-function
is a meromorphic function of $u\in \bC$ with double poles at each lattice point  $\omega\in \Lambda$. It is even and doubly-periodic
\[ \wp(-u)=\wp(u), \qquad \wp(u+\omega_i)=\wp(u),\]
and satisfies the differential equation
\[\wp'(u)^2=4\wp(u)-g_2(\Lambda) \wp(u)-g_3(\Lambda),\]
where $g_2(\Lambda), g_3(\Lambda)\in \bC$ are constants depending on the lattice $\Lambda$. 

The Weierstrass $\zeta$-function is
uniquely characterised by the properties \begin{equation}
\label{ozone}\zeta'(u) =-\wp(u), \qquad \zeta(-u)=-\zeta(u).\end{equation}
It has simple poles at the lattice points.
This function is not quite periodic but satisfies
\begin{equation} \label{oztwo}\zeta(u+\omega_i)-\zeta(u)= \eta_i,\end{equation}
where the quasi-periods $\eta_1,\eta_2\in \bC^*$ satisfy the Legendre relation
\begin{equation}
\label{leg}
\omega_2\eta_1-\omega_1 \eta_2 = 2\pi i.\end{equation}

There is an addition formula
\begin{equation}
\label{holiday}\zeta(u-v)-\zeta(u)+\zeta(v)=\frac{\wp'(u)+\wp'(v)}{2(\wp(u)-\wp(v))}.\end{equation}

Finally, the Weierstrass $\sigma$-function is uniquely characterised by the relations
\[\frac{d}{du} \log \sigma(u)=\zeta(u),\qquad \lim_{u\to 0} \bigg(\frac{\sigma(u)}{u}\bigg)=1.\]
It has the quasi-periodicity property
\begin{equation}\label{qup}\sigma(u+\omega_i)=-\exp\big(\eta_i (u+\half \omega_i)\big)\cdot \sigma(u),\end{equation}
and has simple poles at the lattice points $\omega\in \Lambda$.

\subsection{Period map}
\label{sectz}

Recall from the introduction the family of elliptic curves $X_s$ parameterised by the points $s\in S$. They are the projectivizations of the affine cubics
\[X_{s}^\circ=\big\{(x,y)\in \bC^2: y^2=x^3+ax+b\big\}.\]
As before we set $\Gamma_s=H_1(X_s,\bZ)$, and  denote by
\[\<-,-\>\colon \Gamma_s\times \Gamma_s\to \bZ\]
the skew-symmetric intersection form. 
We also consider the vector bundle $\pi \colon T\to S$  with fibres
\[T_s=H^1(X_s,\bC)= \Hom_\bZ(\Gamma_s,\bC)\isom \bC^2.\]
The Gauss-Manin connection defines a flat connection on this bundle.
There is a holomorphic section $Z\colon S\to T$ 
defined by sending a class $\gamma\in \Gamma_s$ to
\[Z(s)(\gamma)=\int_{\gamma} \sqrt{Q_0(x)}\, dx=\int_{\gamma} y \, dx \in \bC,\]
which we call the period map. We claim that the covariant derivative of  $Z$  defines an isomorphism
\[\nabla(Z)\colon \cT_S \to T,\]
between the holomorphic tangent bundle of $S$  and the bundle $T$.

Let us express all this in co-ordinates.  For this purpose, fix a base-point $s_0\in S$, and choose a basis
\[ \Gamma_{s_0}=\bZ\gamma_1\oplus \bZ\gamma_2\]
satisfying $\<\gamma_1,\gamma_2\>=1$. Extend this  basis to nearby fibres $\Gamma_s$ using the Gauss-Manin connection. We obtain a local trivialization of the bundle $\pi \colon T\to S$
\begin{equation}
\label{ham}\big(\theta\colon \Gamma_s\to \bC\big)\in T_s \mapsto (\theta_1,\theta_2)=\big(\theta(\gamma_1),\theta(\gamma_2)\big)\in \bC^2,\end{equation}
 and the section $Z$ becomes a pair of functions  on $S$
\begin{equation}
\label{zzz}z_i=  \int_{\gamma_i} \sqrt{x^3+ax+b} \cdot dx.\end{equation}
The claim is equivalent to the statement that these functions form a local system of co-ordinates on $S$. We check this by direct calculation in Lemma \ref{pension} below.

\subsection{Formula for the period map}

For each point $s\in S$,  we equip the elliptic curve $X_s$ with the global holomorphic one-form $\Omega$ which extends the form $dx/2y$ on the affine piece $X_s^\circ$. The periods of this form 
\[\omega_i=\int_{\gamma_i} \Omega=\int_{\gamma_i} \frac{dx}{2y} \in \bC^*\]
span a lattice $\Lambda_s=\bZ\omega_1\oplus \bZ\omega_2\subset \bC$. The condition $\<\gamma_1,\gamma_2\>=1$ ensures that $\Im(\omega_2/\omega_1)>0$. The corresponding Weierstrass $\wp$-function defines a  map
\[\bC\setminus \Lambda_s\to X_s^\circ, \qquad u\mapsto (x,y)=\big(\wp(u), \half \wp'(u)\big),\]
which  extends to an isomorphism of complex manifolds
\begin{equation}
\label{param}\bC/\Lambda_s\isom X_s.\end{equation}
Under this identification we have $\Omega=du$.

\begin{lemma}
\label{pension}
The functions $(z_1,z_2)$  give  local co-ordinates on $S$. There  are equalities of tangent vectors on $S$
\begin{equation}
\label{pens1}\frac{\partial}{\partial a}=-\eta_1\frac{\partial}{\partial z_1} -\eta_2\frac{\partial}{\partial z_2},\qquad \frac{\partial}{\partial b}=\omega_1\frac{\partial}{\partial z_1} +\omega_2\frac{\partial}{\partial z_2}.\end{equation}
\begin{equation}
\label{pens2}2\pi i\cdot\frac{\partial}{\partial z_1}=-\omega_2\frac{\partial}{\partial a} -\eta_2\frac{\partial}{\partial b},\qquad 2\pi i\cdot \frac{\partial}{\partial z_2}=\omega_1\frac{\partial}{\partial a} +\eta_1\frac{\partial}{\partial b},\end{equation}
where $\eta_1,\eta_2$ denote the quasi-periods of the Weierstrass $\zeta$-function associated to the lattice $\Lambda_s$.
\end{lemma}

\begin{proof}
Differentiating \eqref{zzz} gives\begin{equation}
\label{blob1}\frac{\partial z_i}{\partial a} =  \int_{\gamma_i} \frac{x\, dx}{2\sqrt{x^3+ax+b}}=\int_{\gamma_i} \frac{x\, dx}{2y}=\int_{\gamma_i} \wp(u) du = -\eta_i,\end{equation}
\begin{equation}
\label{blob2}\frac{\partial z_i}{\partial b} =  \int_{\gamma_i} \frac{dx}{2\sqrt{x^3+ax+b}}=\int_{\gamma_i} \frac{dx}{2y}=\int_{\gamma_i}  du = \omega_i,\end{equation}
and hence the relations \eqref{pens1}.   Inverting these using the Legendre relation \eqref{leg}  gives  \eqref{pens2}.
\end{proof}

\subsection{Abelian holonomy map}
\label{defab}

Consider a point $(a,b,q,p,r)\in \MM$ and set $s=(a,b)\in S$. We denote by $w=(q,p)$ the corresponding point of the elliptic curve $X_s$. 
Using the parameterization \eqref{param} of $X_s$ we can write
\begin{equation}
\label{tired}
w=(q,p)=\big(\wp(v), \half \wp'(v)\big).\end{equation}
for some point $v\in \bC+\Lambda_s$. Let us denote by $\infty\in X_s$ the point at infinity on the elliptic curve $X_s$. In terms of the parameterization \eqref{param} this   corresponds to $0+\Lambda_s$.  Let us introduce the meromorphic differential on $X_s$
\begin{equation}\label{b}\varpi(u) du=-\bigg(\frac{y+p}{x-q}+r\bigg)\frac{dx}{2y}=-\bigg(\frac{\wp'(u)+\wp'(v)}{2(\wp(u)-\wp(v))} +r\bigg)du.\end{equation}
A simple calculation  shows that $\varpi(u) du$  has simple poles at the points $\infty$ and $w$, with residues $+1$ and $-1$ respectively, and no other poles.

 Consider  the degree zero line bundle $L=\O_{X_s}(w-\infty)$ on $X_s$.  In terms of the parameterization \eqref{param}, the sections of $L$ over an open subset   are meromorphic functions $f(u)$ having zeroes at the points $u\in \Lambda_s$, and at worst  simple poles at the points $u\in v+\Lambda_s$. Note that for any such function $f(u)$, the function $f'(u)-\varpi(u)f(u)$ has the same property. It follows that  the formula
 \[\nabla=d- \varpi(u) du\]
 defines a holomorphic connection on $L$. Computing the flat sections of $\nabla$ shows that the holonomy of this connection  about a loop $\gamma$  in $X_s$ is given by multiplication by the expression \begin{equation}
\label{jk}\xi(\gamma)=\exp \bigg(\int_{\gamma} \varpi(u) du\bigg)\in \bC^*.\end{equation}

Consider now the moduli space $\cM_s$ of pairs $(L,\nabla)$ consisting of a line bundle $L$ on the curve $X_s$, equipped with a holomorphic  connection $\nabla$. Then $\cM_s$ is an affine bundle over the space of degree zero line bundles $\Pic^0(X_s)$ modelled on the vector space $\bC=H^0(X_s,\omega_{X_s})$.    The Riemann-Roch theorem shows that the line bundles $\O_{X_s}(w-\infty)$ for different points $w\in X_s$ are all distinct, and that all degree 0 line bundles on $X_s$ are of this form. Since these line bundles  have only trivial automorphisms,   the pairs $(L,\nabla)$ defined by different points $(q,p,r)$ of the fibre $\MM_s=\pi^{-1}(s)\subset \MM$ are all non-isomorphic. It follows that the  map   \begin{equation}
\label{ky}A_s\colon \MM_s\to \cM_s, \qquad A\colon (q,p,r)\mapsto (L,\nabla)=\big(\O_{X_s}(w-\infty),d-\varpi(u)du\big),\end{equation}
is an open embedding. The condition $p\neq 0$ on the points of $\MM$ translates into the statement that the associated line bundle $L=\O_{X_s}(w-\infty)$ is not a spin bundle, that is, it does not satisfy $L^2\isom \O_X$. The image of the embedding $A_s$ is therefore precisely the  set of pairs $(L,\nabla)$ for which the bundle $L$ is non-spin. 

For each point $s\in S$, the abelian Riemann-Hilbert correspondence shows that taking holonomy defines an isomorphism of complex manifolds $\operatorname{Hol}\colon \cM_s\to \bT_s$. Pre-composing with the open embedding $A_s\colon \MM_s\into \cM_s$ defines an open embedding $\Theta_s\colon \MM_s\into \bT_s$ which sends a point $(q,p,r)\in \MM_s$ to the holonomy  \eqref{jk} of the  pair $(L,\nabla)$ appearing in \eqref{ky}. 
Let us consider, as in the introduction, the bundle $\pi \colon \bT\to S$ whose fibres are the cohomology groups
\[\bT_s=H^1(\bT_s,\bC^*)=\Hom_\bZ(\Gamma_s,\bC^*)\isom (\bC^*)^2.\]
Then, taking the union of the maps $\Theta_s$   defines an open embedding $\Theta$, which  fits into the diagram
\begin{equation*}
\xymatrix@C=1.5em{
\MM \ar^{\Theta}[rr] \ar_{\pi}[dr] && \bT\ar^{\pi}[dl] \\
&S
} \end{equation*}
and induces the open embeddings $\Theta_s\colon \MM_s\into \bT_s$ on the fibres.
We call this map $\Theta$ the abelian holonomy map. 
%
%

\subsection{Explicit formula}
\label{formz}

The bundle of tori $\pi \colon \bT\to S$  is the quotient of the vector bundle $\pi\colon T\to S$ by the local system of lattices
\begin{equation}
\label{lattices}\Gamma_s^\vee=\Hom_\bZ(\Gamma_s,\bZ)\subset T_s.\end{equation}
Choosing a covariantly constant basis for the lattices $\Gamma_s$ as in Section \ref{sectz} gives a local trivialisation
\[\big(\xi\colon \Gamma_s\to \bC^*\big)\in \bT_s \mapsto (\xi_1,\xi_2)=\big(\xi(\gamma_1),\xi(\gamma_2)\big)\in (\bC^*)^2.\]
The quotient map $p\colon T\to \bT$ is expressed in co-ordinates by writing $\xi_i=\exp(\theta_i)$. Thus the pair $(\theta_1,\theta_2)$ of \eqref{ham} can also be viewed as local co-ordinates on the bundle $\bT$.

On the space $\MM$ we can take local co-ordinates $(a,b,q,r)$. We can also express the co-ordinate $q$ in terms of  $v$  using the parameterization \eqref{tired} as before. Of course the  Weierstrass function $\wp(v)$  depends implicitly on the lattice $\Lambda_s$, and hence  on the variables $(a,b)$.

\begin{lemma}
\label{rain}
In the above co-ordinates the abelian holonomy map $\Theta$ is given by 
\begin{equation}
\label{exp}\xi_i=\exp\big( \eta_i v-r \omega_i-\omega_i \zeta(v)\big),\end{equation}
where $\zeta(v)$ denotes the Weierstrass zeta-function for the lattice $\Lambda_s$.
\end{lemma}

\begin{proof}
 This is a direct computation which the author learnt from  \cite[Section 3]{LM}:
\[\theta_i=\log(\xi_i)=\int_{\gamma_i} \varpi(u) du = -\bigg[\log \frac{\sigma(u-v)}{\sigma(u)} + u\zeta(v)+ur\bigg]_{\gamma_i}=\eta_i v- \omega_i(\zeta(v)+r),\]
where we used the addition formula \eqref{holiday}, and the quasi-periodicity property \eqref{qup}. Note that by construction the differential $\varpi(u) \, du$ has simple poles with integer residues,  so the expression for $\theta_i$ is well-defined up to multiples of $2\pi i$, and the quantity $\xi_i=\exp(\theta_i)$ is therefore well-defined.
 \end{proof}

It will be convenient in what follows to introduce alternative local co-ordinates $(\theta_a,\theta_b)$ on the torus $\bT_s$ by setting
\begin{equation}
\label{hide}\theta_i=-\eta_i\theta_a+\omega_i \theta_b.\end{equation}
Using the Legendre relation \eqref{leg}, the  inverse transformation is
\begin{equation}
\label{hide2} 2\pi i \cdot \theta_a=-\omega_2\theta_1+\omega_1\theta_2, \qquad 2\pi i \cdot  \theta_b=-\eta_2\theta_1+\eta_1\theta_2.\end{equation}
In these co-ordinates \eqref{exp} takes the simple form
\begin{equation}\label{mod}\theta_a=-v=-\frac{1}{4}\int^{(q,p)}_{(q,-p)} \frac{dx}{y}, \qquad \theta_b=-\zeta(v)-r=\frac{1}{4}\int^{(q,p)}_{(q,-p)} \frac{x dx}{y}-r.\end{equation}
It is easy to see that the integrals in \eqref{mod} are well-defined providing we  take an integration path which is invariant under the covering involution of $p\colon X_s\to \bP^1$ defined by $(x,+y)\leftrightarrow (x,-y)$.




\subsection{Further remarks}

We record here a few further comments on the  abelian holonomy map which will be useful later.

\begin{remark}
\label{defined}
It follows from the discussion in Section \ref{defab} that the complement of the image of  the embedding $\Theta_s\colon \MM_s\into \bT_s$ consists precisely of the holonomy of  holomorphic connections on the four spin bundles on $X_s$. These correspond to the half-lattice points
\[\Big\{0,\half \omega_1, \half \omega_2, \half(\omega_1+\omega_2)\Big\}\in v+\Lambda_s.\]
 Direct calculations shows that the resulting points of $\bT_s$ have co-ordinates $\xi_i=\pm \exp({r\omega_i})$, for some $r\in \bC$, with the four possible choices of pairs of signs corresponding to the four spin bundles.  For the non-trivial spin bundles this follows from \eqref{exp} using the Legendre relation \eqref{leg} and the identities
 \[\zeta\big(\half\omega_1\big)=\half\eta_1, \qquad \zeta\big(\half\omega_2\big)=\half\eta_2, \qquad \zeta\Big(\half(\omega_1+\omega_2)\Big)=\half(\eta_1+\eta_2),\]
 which are easily derived from \eqref{ozone}-\eqref{oztwo}. On the other hand, a holomorphic connection on the trivial bundle $\O_{X_s}$ takes the form $d-r\, du$, where $d$ denotes the trivial connection. The holonomy around the cycles $\gamma_i\in \Gamma_s$ is then given by multiplication by $\xi_i=\exp({r\omega_i})$.
\end{remark}

\begin{remark}
In the introduction we  defined the map $\Theta$ by an expression
\begin{equation}
\label{duff}\xi(\gamma)=\exp\bigg(\int_{\gamma}\frac{-Q_1(x)\, dx}{2\sqrt{Q_0(x)}}\bigg).\end{equation}
The meromorphic differential on $X_s$ being integrated here
\begin{equation}
\label{a}-\frac{Q_1(x)\, dx}{2\sqrt{Q_0(x)}}=-\bigg(\frac{p}{x-q} + r\bigg) \frac{dx}{2y}=-\bigg(\frac{\wp'(v)}{2(\wp(u)-\wp(v))} + r\bigg) du,\end{equation}
has simple poles at the points $\pm v+\Lambda_s$ with residues $\mp \half$.    It follows that the integral of \eqref{a} against any homology class is well-defined only up to integer multiples of $\pi i $, and that the exponential \eqref{duff} is therefore only well-defined up to sign. The difference between  \eqref{a} and \eqref{b} is given by
the form \[\frac{dx}{2(x-q)}=\frac{\wp'(u)\, du}{2(\wp(u)-\wp(v))}.\]Since this differential is pulled back from $\bP^1$ via the double cover  $p\colon X_s\to \bP^1$,  its integral around any cycle (which is only well-defined up to integer multiples of $\pi i$) must in fact be an integer multiple of $\pi i$. Thus the expressions \eqref{jk} and \eqref{duff}  agree up to sign.
\end{remark}

\begin{remark}
\label{actions}
There are two group actions on the space $M$ which will be important later, and which are respected by the   abelian holonomy map.
\begin{itemize}
\item[(a)] There are  involutions of the spaces $M$ and $\bT$ defined in local co-ordinates  by \[(a,b,q,p,r)\leftrightarrow (a,b,q,-p,-r), \qquad (z_1,z_2,\theta_1,\theta_2)\leftrightarrow (z_1,z_2,-\theta_1,-\theta_2).\]
It follows from \eqref{mod} and \eqref{hide} that these are intertwined by the map $\Theta$.

\item[(b)] Consider the action of  $\bC^*$ on the space $M$ for which the co-ordinates $(a,b,q,p,r)$ are homogeneous of weights $(4,6,2,3,1)$ respectively. Rescaling also the co-ordinate $x$ on $\bC\subset \bP^1$  with weight $2$, the formula \eqref{zzz} shows that the co-ordinates $(z_1,z_2)$ have weight $5$, and formulae \eqref{b}-\eqref{jk} that  the co-ordinates $(\theta_1,\theta_2)$ have weight 0.  The formulae \eqref{blob1}-\eqref{blob2} then show that $(\omega_i,\eta_i)$ have weight $(-1,1)$ respectively, and thus by \eqref{hide2} the co-ordinates  $(\theta_a,\theta_b)$ have weights $(-1,1)$.
\end{itemize}
\end{remark}

We shall need the following formula for the derivative of the map $\Theta$.

\begin{lemma}
\label{derder}
The derivative of the abelian holonomy map with respect to the local co-ordinates $(a,b,q,r)$ on $\MM$ and $(a,b,\theta_a,\theta_b)$ on $\bT$ is given by
\begin{equation}
\label{froo}\Phi_*\bigg(\frac{\partial}{\partial q}\bigg)= -\frac{1}{2p} \frac{\partial}{\partial \theta_a}+\frac{q}{2p}\frac{\partial}{\partial \theta_b}, \qquad \Phi_*\bigg(\frac{\partial}{\partial r}\bigg)= -\frac{\partial}{\partial \theta_b},\end{equation}
\begin{equation}\Phi_*\bigg(\frac{\partial}{\partial a}\bigg)= \frac{\partial}{\partial a}+\kappa_1(v) \frac{\partial}{\partial \theta_a}-\kappa_2(v) \frac{\partial}{\partial \theta_b}, \qquad \Phi_*\bigg(\frac{\partial}{\partial b}\bigg)= \frac{\partial}{\partial b}+\kappa_0(v) \frac{\partial}{\partial \theta_a}- \kappa_1(v) \frac{\partial}{\partial \theta_b},\end{equation}
where we introduced
the functions
\begin{equation}
\label{newone}\kappa_i(v)=\frac{1}{8}\int^{(q,p)}_{(q,-p)}\frac{x^i dx}{y^3}=\int_{-v}^v \frac{\wp(u)^i\, du}{\wp'(u)^2}, \qquad i=0,1,2.\end{equation}
\end{lemma}

\begin{proof}
The abelian holonomy map is given by the formulae \eqref{mod}, which  can be viewed as integrals of multi-valued 1-forms on $\bP^1$. Differentiating these gives the result. \end{proof}

\begin{remark} As with \eqref{mod}, the integrals in \eqref{newone} are well-defined providing the path of integration is invariant under $(x,+y)\leftrightarrow (x,-y)$. This results in meromorphic functions which are uniquely defined by the properties
\[\kappa'_i(v)=\frac{2\wp(v)^i}{\wp'(v)^2}, \qquad \kappa_i(-v)=-\kappa_i(v).\]
By computing the derivative of  $(\alpha+\beta \wp(v)+\gamma\wp^2(v))/\wp'(v)$ as in \eqref{useful} below, and comparing constants, it is not hard to write $\kappa_i(v)$  explicitly in terms of  $\wp(v)$, $\zeta(v)$ and $v$. Since we will make no use of the resulting expressions,  we refrain from writing out the details.\end{remark}




\section{The isomonodromy connection}
\label{wwp}

In this section we combine the material from the previous sections to give proofs of Theorems \ref{one} and \ref{one2}. We first use the abelian holonomy map    to transfer  the pencil of isomonodromy connections  to the bundle $\pi\colon \bT\to S$. We  then write the transferred pencil of non-linear connections in the natural co-ordinate system  $(z_i,\theta_j)$. The resulting expressions show that  these connections define what is called a Joyce structure in \cite{RHDT2}. We then discuss the induced linear Joyce connection on $S$, and prove that its flat co-ordinates are $(a,b)$.

\subsection{Rewriting the isomonodromy flow}
We proved in the last section that the abelian holonomy map  $\Theta\colon \MM\into \bT$  is an open embedding, commuting with the projections to $S$.  We can therefore use it  to push-forward  the  isomonodromy connection of Proposition \ref{isom}.
The following result gives a Hamiltonian description of the resulting meromorphic Ehresmann connection. 

\begin{thm}
\label{hi}
The push-forward of the isomonodromy connection along the open embedding $\Theta\colon M\into \bT$  is spanned by vector fields of the form
\begin{equation}\label{zz1}
\frac{\partial}{\partial a} +\frac{1}{\hbar} \cdot \frac{\partial}{\partial \theta_a} +\frac{1}{2\pi i}\cdot \frac{\partial^2 K}{\partial \theta_a \partial \theta_b}\cdot \frac{\partial}{\partial \theta_a}-\frac{1}{2\pi i}\cdot\frac{\partial^2 K}{\partial \theta_a \partial \theta_a }\cdot \frac{\partial}{\partial \theta_b},\end{equation}\begin{equation}\label{zz2}
\frac{\partial}{\partial b} + \frac{1}{\hbar}  \cdot  \frac{\partial}{\partial \theta_b} +\frac{1}{2\pi i}\cdot\frac{\partial^2 K}{\partial \theta_b \partial \theta_b }\cdot \frac{\partial}{\partial \theta_a}-\frac{1}{2\pi i}\cdot\frac{\partial^2 K}{\partial \theta_a \partial \theta_b }\cdot \frac{\partial}{\partial \theta_b}, \end{equation}
with  $K$ a holomorphic  function defined on the image of $\Theta$.\end{thm}

\begin{proof}
Making a trivial linear combination of the flows of Proposition \ref{isom},  and leaving the variation of $p=p(a,b,q)$ implicit,  the isomonodromy connection is generated by the vector fields
\begin{equation}\label{da} -\frac{2p}{\hbar} \frac{\partial}{\partial q}
-\frac{q}{\hbar} \frac{\partial}{\partial r} +\bigg(\frac{\partial}{\partial a}-\frac{r}{p}\frac{\partial}{\partial q}
-\frac{r^2(3q^2+a)-qpr}{2p^3}  \frac{\partial}{\partial r}\bigg),\qquad -\frac{1}{\hbar}\frac{\partial}{\partial r}+\bigg(\frac{\partial}{\partial b}
 +\frac{r}{2p^2} \frac{\partial}{\partial r}\bigg).\end{equation}
Applying the derivative of  $\Theta$ computed in Lemma \ref{derder} these become
\[\frac{1}{\hbar} \frac{\partial}{\partial \theta_a} + \frac{\partial}{\partial a}+ \mu \frac{\partial}{\partial \theta_a}- \nu \frac{\partial}{\partial \theta_b},\qquad
 \frac{1}{\hbar} \frac{\partial}{\partial \theta_b} + \frac{\partial}{\partial b}+ \lambda \frac{\partial}{\partial \theta_a}- \mu \frac{\partial}{\partial \theta_b},\]
where the functions $\lambda,\mu,\nu$ are defined on  the image of  $\Theta$   by
\begin{equation}
\label{es}\Theta^*(\lambda)= \kappa_0(v), \qquad \Theta^*(\mu)=\kappa_1(v)+\frac{r}{2p^2}, \qquad \Theta^*(\nu)= \kappa_2(v)+\frac{qr}{p^2} -\frac{(3q^2+a)r^2}{2p^3}.\end{equation}
Note that the inverse to the derivative in Lemma \ref{derder} satisfies
\begin{equation}
\label{inv}\Theta_*^{-1}\Big(\frac{\partial}{\partial \theta_a}\Big)= -2p \frac{\partial}{\partial q}-q\frac{\partial}{\partial r}, \qquad \Theta_*^{-1}\Big(\frac{\partial}{\partial \theta_b}\Big)= -\frac{\partial}{\partial r}.\end{equation}

Define a holomorphic function $J$ on the image  of the open inclusion $\Theta\colon M\into \bT$ by
\begin{equation}
\label{j}\frac{1}{2\pi i} \cdot \Theta^*(J)=-\frac{1}{4\Delta p} \big( 2ap^2+3p(3b-2aq)r+(6aq^2-9bq+4a^2)r^2 - 2apr^3\big),\end{equation}
where we set $\Delta=4a^3+27b^2$. 
The defining relation $p^2=q^3+aq+b$ implies that
\begin{equation}
\label{useful} 2p\frac{d}{dq}\Big(\frac{\alpha q^2  +\beta q+\gamma }{p}\Big)=\alpha q-\beta+\frac{1}{p^2} \Big( (2a\alpha-3\gamma)q^2+ (2a\beta+3 b\alpha)q+(3b\beta-a\gamma)\Big).\end{equation}
A slightly painful caluclation using \eqref{inv} and the relation \eqref{useful}  repeatedly gives\[\frac{1}{2\pi i} \cdot \Theta^*\Big(\frac{\partial J}{\partial \theta_a}\Big)=\frac{1}{\Delta}\bigg( \frac{a^2}{2} +\frac{9bq}{4}- \frac{\big(9bq^2+2a^2q+6ab\big)r}{2p} +\frac{9br^2}{4}\bigg) -\frac{r^2}{4p^2},\]
\[\frac{1}{2\pi i} \cdot \Theta^*\Big(\frac{\partial^2 J}{\partial \theta_a^2}\Big)=\frac{-1}{\Delta}\bigg(\frac{-2a^2q^2+3abq+9b^2}{2p} + a^2r\bigg)+\frac{qr}{p^2} - \frac{(3q^2+a)r^2}{2p^3},\]
\[\frac{1}{2\pi i} \cdot \Theta^*\Big(\frac{\partial^3 J}{\partial \theta_a^3}\Big)=-\frac{3ab}{2\Delta}-\frac{3q^2}{2p^2}-\frac{2r}{p} + \frac{3q(3q^2+a)r}{p^3}+\frac{6q r^2}{p^2}-\frac{3(3q^2+a)^2r^2}{2p^4}.\]
On the other hand, applying the operators \eqref{inv}  to the expressions \eqref{es}, and noting the overlap with the previous calculation, we easily obtain\begin{equation}
\label{bol}\Theta^*\bigg(\frac{\partial \nu}{\partial \theta_a}\bigg)=-\frac{3q^2}{2p^2}- \frac{2r}{p}+\frac{3q(3q^2+a)r}{p^3}+\frac{6qr^2}{p^2}-\frac{3(3q^2+a)^2 r^2}{2p^4},\qquad \frac{\partial \lambda}{\partial \theta_b }=0,\end{equation}
\begin{equation}\label{hol}\Theta^*\bigg(\frac{\partial \mu}{\partial \theta_a }\bigg)=-\frac{q}{p^2}+\frac{(3q^2+a)r}{p^3}=\Theta^*\bigg(\frac{\partial \nu}{\partial \theta_b }\bigg),\qquad \Theta^*\bigg(\frac{\partial \lambda}{\partial \theta_a }\bigg)=-\frac{1}{2p^2}=\Theta^*\bigg(\frac{\partial \mu}{\partial \theta_b }\bigg).\end{equation}
Let us now define $K=J+C$, where
\begin{equation}
\label{c}\frac{1}{2\pi i} \cdot C=\frac{1}{4\Delta} \big(ab\theta_a^3-2a^2\theta_a^2\theta_b-9b\theta_a\theta_b^2 +2a\theta_b^3\big).\end{equation}
Comparing with \eqref{bol}-\eqref{hol} we see that
\begin{equation}
\label{pet}2\pi i \lambda=\frac{\partial^2 K}{\partial \theta_b^2}, \qquad 2\pi i \mu=\frac{\partial^2 K}{\partial \theta_a\partial \theta_b}, \qquad 2\pi i \nu=\frac{\partial^2 K}{\partial \theta_a^2},\end{equation}
up to the addition of functions independent of $\theta_a,\theta_b$. But these constants of integration must vanish because, by Remark \ref{actions}(a),  both sides of \eqref{pet} are odd functions  of the $\theta_a,\theta_b$ co-ordinates. 
\end{proof}


\subsection{Co-ordinate change}
\label{change}

To pass from the statement of Theorem \ref{hi} to that of Theorem \ref{one} we need to apply a change of variables. For this purpose, 
let us consider the following abstract problem. Suppose that $S$ is a  complex manifold equipped with a complex symplectic form $\omega$. Take a  local co-ordinate system  $(z_1,\cdots,z_n)$ on $S$ in which this form is constant, so that we can write
\begin{equation}
\label{funeral}\omega=\sum_{i,j} \omega_{ij} \cdot dz_i \wedge dz_j,\end{equation}
for some constant skew-symmetric matrix $\omega_{ij}$.  The induced Poisson bracket on $S$ is given in these co-ordinates by the inverse matrix
\begin{equation}
\label{mum}\{z_i,z_j\}=\epsilon_{ij}, \qquad \sum_j \epsilon_{ij}\cdot \omega_{jk}=\delta_{ik}.\end{equation}
There is a natural co-ordinate system $(z_1,\cdots,z_n,\theta_1,\cdots,\theta_n)$ on  the total space of the holomorphic tangent bundle $\cT_S$  obtained by writing a tangent vector in the form
$\sum_i \theta_i \cdot \frac{\partial}{\partial z_i}$.
 We are interested in   systems of flows on $\cT_S$ of the form
\begin{equation}
\label{flowj}
\frac{\partial}{\partial z_i}+ \frac{1}{\hbar} \frac{\partial}{\partial \theta_i} + \sum_{j,k}  \epsilon_{jk} \frac{\partial^2 J}{\partial \theta_i \partial \theta_j} \frac{\partial}{\partial \theta_k}.
\end{equation}
where $J\colon \cT_S\to \bC$ is some fixed holomorphic function.

Consider now some new co-ordinate system $(w_1,\cdots,w_n)$ on $S$ which is related to the first by a symplectomorphism, so that  $\omega$ takes the same form \eqref{funeral}. In the same way as before, we can consider the induced co-ordinates  $(w_1,\cdots, w_n,\phi_1,\cdots,\phi_n)$ on the tangent bundle $\cT_S$. Given a holomorphic function $K\colon \cT_S\to \bC$ we can then consider the system of flows
\begin{equation}
\label{flowk}
\frac{\partial}{\partial w_i}+ \frac{1}{\hbar} \frac{\partial}{\partial \phi_i} + \sum_{j,k}  \epsilon_{jk} \frac{\partial^2 K}{\partial \phi_i \partial \phi_j} \frac{\partial}{\partial \phi_k}.
\end{equation}
We would like to know when the two flows \eqref{flowj} and \eqref{flowk} on the space $\cT_S$ are equivalent, in the sense that they generate the same sub-bundle of the tangent bundle. 

Let $\nabla$ denotes the flat, torsion-free, linear connection on the tangent bundle $\cT_S$ whose flat co-ordinates are $(w_1,\cdots,w_n)$, and define
\begin{equation}
\label{coeffs}C_{pqr}(z)=\omega\bigg(\nabla_{\frac{\partial}{\partial z_p}} \Big(\frac{\partial}{\partial z_q}\Big),\frac{\partial}{\partial z_r}\bigg). \end{equation}
Note  that the expression \eqref{coeffs} is completely symmetric under permutation of the indices $p,q,r$. Indeed,  the assumption that the symplectic form $\omega$ is constant in the co-ordinates $z_i$ and $w_j$, and hence is preserved by  $\nabla$ gives
\[C_{pqr}(z)-C_{prq}(z)=\omega\bigg(\nabla_{\frac{\partial}{\partial z_p}} \Big(\frac{\partial}{\partial z_q}\Big),\frac{\partial}{\partial z_r}\bigg)+ \omega\bigg(\frac{\partial}{\partial z_q}, \nabla_{\frac{\partial}{\partial z_p}} \Big(\frac{\partial}{\partial z_r}\Big)\bigg)=\frac{\partial}{\partial z_p}\omega\Big( \frac{\partial}{\partial z_q},\frac{\partial}{\partial z_r}\Big)=0.\]
On the other hand, the fact that $\nabla$ is torsion-free gives $C_{pqr}(z)=C_{qpr}(z)$.

\begin{prop}
\label{fur}
The two flows \eqref{flowj} and \eqref{flowk}  define the same Ehresmann connection on the bundle $\pi \colon \cT_S\to S$ precisely if
\begin{equation}
\label{blud}J(z_i,\theta_j)=K(w_i,\phi_j)-\frac{1}{6}\cdot \sum_{p,q,r}C_{pqr}(z_i) \cdot \theta_p \theta_q \theta_r.\end{equation}
\end{prop}

\begin{proof}
Take a point in the total space $\cT_S$ with co-ordinates $(z_i,\theta_j)$ and $(w_i,\phi_j)$. Then
\[\sum_i \theta_i \frac{\partial}{\partial z_i}=\sum_j \phi_j \frac{\partial}{\partial w_j} \implies  \phi_j=\sum_i \theta_i \frac{\partial w_j}{\partial z_i}.\]
Changing co-ordinates on the space $\cT_S$ from $(z_i,\theta_i)$ to $(w_j,\phi_j)$ therefore gives  \begin{equation}
\label{agon}\frac{\partial}{\partial z_i} = \sum_j \frac{\partial w_j}{\partial z_i} \frac{\partial}{\partial w_j} +\sum_{j,k}\theta_k \frac{\partial^2 w_j}{\partial z_i \partial z_k}\frac{\partial}{\partial \phi_j}, \qquad \frac{\partial}{\partial \theta_i}=\sum_j  \frac{\partial w_j}{\partial z_i} \frac{\partial}{\partial \phi_j}.\end{equation} 
Consider the following linear combination of the flows \eqref{flowk}
\[\sum_p \frac{\partial w_p}{\partial z_i} \frac{\partial}{\partial w_p}+ \frac{1}{\hbar} \sum_p \frac{\partial w_p}{\partial z_i}\frac{\partial}{\partial \phi_p} + \sum_{p, j,k}  \epsilon_{jk} \frac{\partial^2 K}{\partial \phi_p \partial \phi_j} \frac{\partial w_p}{\partial z_i} \frac{\partial}{\partial \phi_k}\]
\begin{equation}
\label{dani}=\frac{\partial}{\partial z_i}-\sum_{j,k}\theta_k \frac{\partial^2 w_j}{\partial z_i \partial z_k}\frac{\partial}{\partial \phi_j}+\frac{1}{\hbar}\frac{\partial}{\partial \theta_i} +\sum_{j,k}  \epsilon_{jk} \frac{\partial^2 K}{\partial \theta_i \partial \theta_j} \frac{\partial}{\partial \theta_k},\end{equation}
where we used the assumption that the change of co-ordinates from $z_i$ to $w_j$ is symplectic, which, using the second relation of \eqref{agon}, implies that for any function $f\colon \cT_S\to \bC$ 
\[\sum_{j,k} \epsilon_{jk} \frac{\partial f}{\partial \theta_j}\frac{\partial }{\partial \theta_k}= \sum_{j,k} \epsilon_{jk} \frac{\partial f}{\partial \phi_j}\frac{\partial }{\partial \phi_k}.\]

The expressions \eqref{dani} agree with \eqref{flowj} provided that
\begin{equation}
\label{news}\sum_{j,k,q}  \epsilon_{jq}C_{ijk}(z) \theta_k\frac{\partial}{\partial \theta_q}= \sum_{j,k} \theta_k \frac{\partial^2 w_j}{\partial z_i \partial z_k} \frac{\partial }{\partial \phi_j} ,\end{equation}
for all indices $i$. 
But now we compute
\[C_{ijk}(z)=\omega\Bigg(\nabla_{\frac{\partial}{\partial z_i}} \Big(\frac{\partial}{\partial z_k}\Big), \frac{\partial}{\partial z_j}\Bigg)=\omega\Bigg(\sum_{p} \frac{\partial^2 w_p}{\partial z_i \partial z_k} \frac{\partial}{\partial w_p}, \frac{\partial}{\partial z_j}\Bigg)=\sum_{p,r} \omega_{rj} \frac{\partial^2 w_p}{\partial z_i \partial z_k}  \frac{\partial z_r}{\partial w_p},\]
and the identity \eqref{news} follows using the fact \eqref{mum} that $\epsilon$ and $\omega$ are inverse matrices.\end{proof}

\subsection{The Joyce function}

We can now use Proposition \ref{fur}  to rewrite the flows of Theorem \ref{hi} in the co-ordinates $(z_i,\theta_j)$. This leads to the following statement.

\begin{thm}
\label{again}
When written in the co-ordinates $(z_1,z_2,\theta_1,\theta_2)$, the push-forward of the isomonodromy flows \eqref{fir1}-\eqref{sec1} along the map $\Theta\colon \MM\to \bT$ take the Hamiltonian form
\begin{equation}
\label{wa}\frac{\partial}{\partial z_i} + \frac{1}{\hbar} \cdot \frac{\partial}{\partial \theta_i} + \frac{\partial^2 J}{\partial \theta_i\partial \theta_1}\cdot \frac{\partial}{\partial \theta_2}-\frac{\partial^2 J}{\partial \theta_i\partial \theta_2}\cdot \frac{\partial}{\partial \theta_1}, \end{equation}
where $J\colon \bT\to \bC$ is a meromorphic function with no poles on  the locus $\theta_1=\theta_2=0$. When pulled-back to $M$ using the abelian holonomy map it is given by the expression
\begin{equation}
\label{indexp}\frac{1}{2\pi i}\cdot J\circ\Theta= -\frac{1}{4\Delta p}\cdot  \big(2ap^2+3p(3b-2aq)r+(6aq^2-9bq+4a^2)r^2 - 2apr^3\big).\end{equation}
\end{thm}

\begin{proof}

Let us define the rescaled co-ordinates
\[w_1=\sqrt{2\pi i}\cdot b , \qquad w_2=\sqrt{2\pi i}\cdot a.\]
Using the expressions \eqref{pens1}, and comparing with \eqref{hide2},  the associated fibre co-ordinates of \eqref{agon} are
\[\phi_1=\sqrt{2\pi i}\cdot \theta_b , \qquad \phi_2=\sqrt{2\pi i}\cdot \theta_a.\]
Since $\epsilon_{12}=\<\gamma_1,\gamma_2\>=1$, the Poisson and symplectic forms on $S$ are
\[\{z_1,z_2\}=1, \qquad \omega=-dz_1\wedge dz_2.\]
The relations \eqref{pens1}  then ensure that
\[\omega=2\pi i \cdot da\wedge db= - dw_1\wedge dw_2.\] 
Making the trivial change of variables from $(a,b,\theta_a,\theta_b)$ to $(w_1,w_2,\phi_1,\phi_2)$ shows that the  flows of Theorem \ref{hi} are given by the equations \eqref{flowk}.
Changing variables as in  Proposition \ref{fur} then gives the flows in the form \eqref{wa}.

The formula \eqref{blud} shows that on the image of $\Theta\colon M\into \bT$, the required function $J$ differs from the function $K$ of Theorem \ref{hi} by an expression which is cubic in the $\theta_i$ co-ordinates. Moreover, by construction, the second derivatives of $J$ with respect to the $\theta_i$ are single-valued functions on the image of $\Theta$. These two conditions uniquely determine $J$. Since the expression \eqref{j} has both the required properties, it follows that this coincides with the required function $J$. 

We can view the space $\MM$ as an open subset of a larger space $\MM'$ obtained by dropping the condition $p\neq 0$. By its construction, the abelian holonomy map extends to an open embedding $\Theta\colon \MM'\into \bT$, and  the expression \eqref{j} clearly defines a meromorphic function on the  open subset $\Theta(\MM')\subset \bT$. As explained in Remark \ref{defined}, the complement of this open subset is precisely the locus where $\xi_i=\exp(c\omega_i)$ for some $c\in \bC$. So it remains to understand the behaviour  of $J$ at these points.

Let us work over a small open subset $U\subset S$.  Let $0\in D\subset \bC$ be a  disc such that the punctured disc $D^\times=D\setminus\{0\}$ contains no points of $\Lambda_s$ for any $s=(a,b)\in U$. Consider   the  map \[h\colon U\times D^{\times}\times \bC\to \MM, \qquad (a,b,v,c)\mapsto \Big(a,b,\wp(v),\half\wp'(v),-\zeta(v)-c\Big).\] Using the formula \eqref{exp} we see that the composition 
$g =\Theta\circ h$ sends  $(a,b,v,c)$ to the point of $\bT$ with local co-ordinates $\theta_a=v$ and $\theta_b=c$. This shows that $g$, and hence also $h$, is an open embedding. Moreover $g$ clearly extends to an open embedding
$g\colon U\times D\times\bC\to \bT$.
It will now be enough to show that the pull-back of the third derivatives \eqref{bol}--\eqref{hol} via the map $g$ extend holomorphically over the locus $v=0$. But indeed, if we fix $(a,b)$ and send  $c$ and $v$ to $0$, we have
\[q=\wp(v)=v^{-2} +\alpha v^2+ O(v^4), \qquad r= -v^{-1}-c+ \tfrac{1}{3}\alpha v^3 +O(v^5),\]
\[p=\half\wp'(v)=-v^{-3} + \alpha v +O(v^3), \qquad p^{-1}=-v^3 -\alpha v^7 +O(v^9).\]
The equation $p^2=q^3+aq+b$ implies that $a+5\alpha=0$.
 It follows that,  ignoring terms of total order at least 2 in $c$ and $v$,  the Joyce function satisfies
 \[-\frac{4\Delta}{2\pi i}(J\circ \Theta)=2a(p-3qr+3q^2 p^{-1} r^2 + 2ap^{-1} r^2-r^3)+9b(r-qr^2 p^{-1})\]\[\sim 2a\big( -v^{-3}+\alpha v -3(v^{-2}+\alpha v^2)(-v^{-1}-c+\tfrac{1}{3}\alpha v^3) +2a(-v^3)(v^{-2})\]\[+ 3(v^{-4}+2\alpha)(-v^3-\alpha v^7)(v^{-2}-2cv^{-1}+c^2-\tfrac{2}{3}\alpha v^2)-(-v^{-3}-3cv^{-2}-3c^2 v^{-1} +\alpha v)\big)\]\[+9b\big(-v^{-1}-c-(v^{-2}+\alpha v^2)(v^{-2}+2cv^{-1}-\tfrac{2}{3}\alpha v^2)(-v^{3})\big) \]\[\sim 2a(-5\alpha v-2av)+9bc=2a^2 \theta_a +9b\theta_b.\]
 In particular, $J$ is holomorphic along the locus $\theta_a=\theta_b=0$.
 \end{proof}

\subsection{Properties of the Joyce function}
\label{bp}
The general theory developed in \cite{RHDT2} predicts some properties of the Joyce function, which it is interesting to  check explicitly in the example being considered here.

\begin{prop}
\label{porpoise}
The Joyce function  $J\colon \bT\to \bC$ of Theorem \ref{again} is a meromorphic function with the following properties:
\begin{itemize}
\item[(i)] it is an odd function in the $\theta_i$ co-ordinates:
\[J(z_1,z_2,-\theta_1,-\theta_2)=-J(z_1,z_2,\theta_1,\theta_2);\]
\item[(ii)] it is homogeneous of degree $-1$ in the co-ordinates $z_i$: for all $\lambda\in \bC^*$
\[J(\lambda z_1,\lambda z_2,\theta_1,\theta_2)=\lambda^{-1} \cdot  J(z_1,z_2,\theta_1,\theta_2);\]
\item[(iii)] it satisfies the partial differential equation
\begin{equation}
\label{justj}
\frac{\partial^2 J}{\partial \theta_i \partial z_j}-\frac{\partial^2 J}{\partial \theta_j \partial z_i}=\sum_{p,q} \epsilon_{pq} \cdot \frac{\partial^2 J}{\partial \theta_i \partial \theta_p} \cdot \frac{\partial^2 J}{\partial \theta_j \partial \theta_q}.
\end{equation}
\end{itemize}
\end{prop}

\begin{proof} 
For part (i) consider the involution of Remark \ref{actions}(a).  It is immediate from \eqref{indexp} that the Joyce function $J(z_i,\theta_j)$ changes sign under this transformation.
For part (ii) we consider the $\bC^*$-action of Remark \ref{actions}(b) which rescales the variables $(a,b,q,p,r)$ with weights $(4,6,2,3,1)$ respectively.  Again, it is immediate from \eqref{indexp} that  $J$ has weight $-5$ for this action. Since the co-ordinates $z_i$ and $\theta_j$ have weights 5 and 0 respectively, this proves the claim.

For part (iii) note first that  the isomonodromy connection on $\pi \colon \MM\to S$ is by definition the pull-back of the trivial connection on the projection map $\pi \colon V\times S\to S$ via the map \[(F(\hbar),\pi)\colon \MM\to V\times S.\] In particular it is flat. Writing out the zero curvature condition
\[  \Bigg[\frac{\partial}{\partial z_1} + \frac{1}{\hbar} \frac{\partial}{\partial \theta_1} +\frac{\partial^2 J}{\partial \theta_1\partial \theta_1}\frac{\partial}{\partial \theta_2}-\frac{\partial^2 J}{\partial \theta_1\partial \theta_2}\frac{\partial}{\partial \theta_1}, \frac{\partial}{\partial z_2} + \frac{1}{\hbar} \frac{\partial}{\partial \theta_2} + \frac{\partial^2 J}{\partial \theta_1\partial \theta_2}\frac{\partial}{\partial \theta_2}-\frac{\partial^2 J}{\partial \theta_2\partial \theta_2}\frac{\partial}{\partial \theta_1}\Bigg]=0\]
for the flows \eqref{wa} shows that the partial derivative of \eqref{justj} with respect to any co-ordinate $\theta_j$ vanishes. So in other words, the difference between the two sides of relation \eqref{justj} is independent of the co-ordinates $\theta_j$.

To complete the proof it will be enough to show that the two sides of \eqref{justj} both vanish on the locus $\theta_1=\theta_2=0$. By the calculation given in the proof of Theorem \ref{again}
\[\frac{1}{2\pi i}\cdot \frac{\partial J}{\partial \theta_a}\Big|_{\theta_1=\theta_2=0}= \frac{2a^2}{4\Delta} , \qquad \frac{1}{2\pi i}\cdot \frac{\partial J}{\partial \theta_b}\Big|_{\theta_1=\theta_2=0}=\frac{9b}{4\Delta},\]
so we find that
\[\frac{\partial ^2 J}{\partial \theta_a \partial b}\Big|_{\theta_1=\theta_2=0}- \frac{\partial ^2 J}{\partial \theta_b \partial a}\Big|_{\theta_1=\theta_2=0}=\frac{2\pi i}{4\Delta^2} \Big( 9b\frac{\partial \Delta}{\partial a}-2a^2 \frac{\partial \Delta}{\partial b}\Big)=0.\]
It follows that the left-hand side of \eqref{justj} vanishes on the locus $\theta_1=\theta_2=0$. But the right-hand side also vanishes because by part (i) 
 $J$ is an odd function of the $\theta_i$.  \end{proof}

\subsection{The linear Joyce connection }

An interesting output of the general theory developed in  \cite{RHDT2} is a flat, torsion-free connection on the tangent bundle of the space $S$, which we call  the linear Joyce connection. To define it, note that by Theorem \ref{again} the function $J$ is holomorphic in a neighbourhood of the section of  the map $\pi \colon \bT\to S$ defined  by setting $\theta_1=\theta_2=0$. Proposition \ref{porpoise}(i) implies that the flows
\[\frac{\partial}{\partial z_i} + \frac{\partial^2 J}{\partial \theta_i\partial \theta_1}\frac{\partial}{\partial \theta_2}-\frac{\partial^2 J}{\partial \theta_i\partial \theta_2}\frac{\partial}{\partial \theta_1},\]
preserve this section, and it follows  that their derivatives in the fibre directions are the flat sections of a linear connection on its normal bundle. This normal bundle can in turn be identified with the tangent bundle $\cT_S$ via the map
\[\frac{\partial}{\partial \theta_1}\mapsto  \frac{\partial}{\partial z_1}, \qquad \frac{\partial}{\partial \theta_2}\mapsto  \frac{\partial}{\partial z_2}. \]
The resulting connection on  $\cT_S$ is given explicitly by the formula
\begin{equation}\label{nab}\nabla^J_{\frac{\partial}{\partial z_i}}\Big(\frac{\partial}{\partial z_j}\Big)=  -\sum_{k,l}\epsilon_{kl}\cdot\frac{\partial^3 J}{\partial \theta_i \, \partial \theta_j \, \partial \theta_k}\Big|_{\theta=0} \cdot \frac{\partial}{\partial z_l}.\end{equation}
For more details on the general definition and properties of the linear Joyce  connection the reader can consult \cite[Section 7]{RHDT2}.
The next result shows that, at least  in the particular context treated in this paper, it is  a very natural object. 

\begin{thm}
\label{obvious}
The linear Joyce connection $\nabla^{ J}$ is the unique connection on $S$ for which the co-ordinates $(a,b)$ are flat.
\end{thm}

\begin{proof}
The same argument used to derive the formula \eqref{nab} shows that in the alternative co-ordinates $(w_i,\phi_j)$ used in the proof of Theorem \ref{again}, the linear Joyce connection is given by
\begin{equation}
\label{fafa}\nabla^J_{\frac{\partial}{\partial w_i}}\Big(\frac{\partial}{\partial w_j}\Big)=  -\sum_{k,l}\epsilon_{kl}\cdot\frac{\partial^3 K}{\partial \phi_i \, \partial \phi_j \, \partial \phi_k}\Big|_{\phi=0} \cdot \frac{\partial}{\partial w_l}.\end{equation}
But applying the limiting argument used in the proof of Theorem \ref{again} to the equations \eqref{bol}-\eqref{hol} shows that the third derivatives of the function $K$ vanish along the locus $\phi_1=\phi_2=0$. Thus the right-hand side of \eqref{fafa} vanishes, and the functions $w_i$ are flat for the linear Joyce connection. \end{proof}

\subsection{The Joyce form}
Let us introduce the vector field
\[E=z_1\frac{\partial}{\partial z_1}+z_2\frac{\partial}{\partial z_2},\]
and consider the endomorphism of $\cT_S$ defined by\[V(X)=X-\nabla^J_X(E).\]
 General theory developed in \cite[Section 7]{RHDT2} shows that  the bilinear form
\[g(X,Y)=\omega\big(V(X),Y\big)\]
is symmetric, and that both this form, and the operator $V$, are  covariantly constant with respect to the linear Joyce connection. We call $g(-,-)$  the Joyce form. Note that when the Joyce form is non-degenerate, the resulting complex metric on $S$  is necessarily flat, since the associated Levi-Civita connection is the linear Joyce connection $\nabla^J$. 

\begin{prop}
The operator $V$ is given by
\[V\bigg(\frac{\partial}{\partial a}\bigg)=\frac{1}{5}\cdot \frac{\partial}{\partial a}, \qquad V\bigg(\frac{\partial}{\partial b}\bigg)=-\frac{1}{5}\cdot\frac{\partial}{\partial b},\]
and
the Joyce form is
\[g=\frac{2\pi i}{5} \cdot (da\tensor db+ db\tensor da).\]
\end{prop}

\begin{proof}
The properties of the $\bC^*$-action used in the proof of Proposition \ref{porpoise} show that
\[E=z_1\frac{\partial}{\partial z_1}+z_2\frac{\partial}{\partial z_2} =\frac{4}{5} \cdot a\frac{\partial}{\partial a} + \frac{6}{5}\cdot b \frac{\partial}{\partial b}.\]
Using Theorem \ref{obvious} the claims then follow directly  from the  definitions.
\end{proof}

\section{ BPS structures}
\label{sec_bps}

In the last section it was explained that the isomonodromy connection for the family of deformed cubic oscillators gives   an example of a Joyce  structure in the sense of \cite{RHDT2}. The remainder of the paper is devoted  to showing how the monodromy map, and hence also the isomonodromy connection,   can be derived from much simpler data called a variation of BPS structures,  by solving an infinite-dimensional Riemann-Hilbert problem.

In this section we introduce BPS structures and their variations. These axiomatise the wall-crossing properties of Donaldson-Thomas (DT) invariants under deformations of stability parameters. We then introduce the Riemann-Hilbert problem associated to a BPS structure. For more details on the contents of this section we refer the reader to \cite{RHDT}.

\subsection{BPS structures}
The notion of a BPS structure was introduced in \cite{RHDT} to axiomatise the output of unrefined DT theory. It is a special case of Kontsevich and Soibelman's notion of a stability structure in a graded Lie algebra  \cite{KS}. In this paper we will only need to consider finite BPS structures, which allows us to make some significant expositional simplifications compared to the general treatment of  \cite{RHDT}.

\begin{defn}
\label{bps}
A finite BPS structure  consists of \begin{itemize}
\item[(a)] a finite-rank free abelian group $\Gamma\isom \bZ^{\oplus n}$, equipped with a skew-symmetric form \[\<-,-\>\colon \Gamma \times \Gamma \to \bZ;\]

\item[(b)] a homomorphism of abelian groups
$Z\colon \Gamma\to \bC$;\smallskip

\item[(c)] a map of sets
$\Omega\colon \Gamma\to \bQ;$
\end{itemize}

satisfying the following properties:

\begin{itemize}
\item[(i)] $\Omega(-\gamma)=\Omega(\gamma)$ for all $\gamma\in \Gamma$, and $\Omega(0)=0$;\smallskip
\item[(ii)] there are only finitely many classes $\gamma\in \Gamma$ such that $\Omega(\gamma)\neq 0$.\footnote{For the general notion of a (possibly non-finite) BPS structure  the condition (ii) is replaced with a weaker condition called the support property: see \cite[Section 2]{RHDT} for details.} \end{itemize}
\end{defn}

A finite BPS structure $(\Gamma,Z,\Omega)$ is called
non-degenerate if the form $\<-,-\>$ is non-degenerate, 
and integral if  $\Omega(\gamma)\in \bZ\subset \bQ$ for all $\gamma\in \Gamma$.

\subsection{The twisted torus}
\label{twtw}

Let $(\Gamma,Z,\Omega)$  be a finite BPS structure as above. We  introduce
the algebraic torus
\[\bT_+=\Hom_\bZ(\Gamma,\bC^*)\isom (\bC^*)^n,\]
whose character lattice is $\Gamma$. 
We denote  its co-ordinate ring by
\[\bC[\bT_+]=\bC[\Gamma]\isom \bC[y_1^{\pm 1}, \cdots, y_n^{\pm n}],\]
and write $y_\gamma\in \bC[\bT_+]$ for the  character of $\bT_+$ corresponding to an element $\gamma\in \Gamma$.

We also consider the associated torsor
\[\bT_-= \big\{g\colon \Gamma \to \bC^*: g(\gamma_1+\gamma_2)=(-1)^{\<\gamma_1,\gamma_2\>} g(\gamma_1)\cdot g(\gamma_2)\big\},\]
called the twisted torus. The difference between  $\bT_+$ and $\bT_-$  just has the effect of introducing signs into various formulae, and can safely be ignored at first reading.

The co-ordinate ring of $\bT_-$ is spanned as a vector space by the functions
\[x_\gamma\colon \bT_-\to \bC^*, \qquad x_\gamma(g)=g(\gamma)\in \bC^*,\]
which we  refer to as twisted characters.  Thus
\begin{equation}\bC[\bT_-]=\bigoplus_{\gamma\in \Gamma} \bC\cdot x_\gamma, \qquad x_{\gamma_1}\cdot x_{\gamma_2}=(-1)^{\<\gamma_1,\gamma_2\>}\cdot x_{\gamma_1+\gamma_2}.\end{equation}

The torus $\bT_+$ acts freely and transitively on the twisted torus $\bT_-$  via
\[(f\cdot g)(\gamma)=f(\gamma)\cdot g(\gamma)\in \bC^*, \quad f\in \bT_+, \quad  g\in \bT_-.\]
Choosing a base-point $g_0\in \bT_-$ therefore gives a bijection \begin{equation}
\label{bloxy}\theta_{g_0}\colon \bT_+\to \bT_-, \qquad f\mapsto f\cdot g_0.\end{equation}
It is often convenient to choose a base-point in the finite subset
 \[\big\{g\colon \Gamma \to \{\pm 1\}: g(\gamma_1+\gamma_2)=(-1)^{\<\gamma_1,\gamma_2\>} g(\gamma_1)\cdot g(\gamma_2)\big\}\subset \bT_-,\]
 whose points  are 
 called quadratic refinements of the form $\<-,-\>$.

 A class $\gamma\in \Gamma$ is called active if the corresponding BPS invariant $\Omega(\gamma)$ is nonzero. A ray $\bR_{>0}\cdot z\subset \bC^*$ is called active if it contains a point of the form $Z(\gamma)$ with $\gamma\in \Gamma$ an active class.  Given a  finite and integral BPS structure, we define for each ray $\ell=\bR_{>0}\cdot z\subset \bC^*$  a birational automorphism of the twisted torus $\bT_-$  by the formula
\begin{equation}
\label{garage}\bS(\ell)^*(x_\beta)=x_\beta\cdot \prod_{Z(\gamma)\in \ell}   (1-x_{\gamma})^{\Omega(\gamma)\<\gamma,\beta\>}.\end{equation}
The product is over all active classes $\gamma\in \Gamma$ such that $Z(\gamma)\in \ell$. The assumption that the BPS structure is finite ensures that this is a finite set.

\subsection{Variation of BPS structures}

The  behaviour of  DT invariants under changes in stability parameters is controlled by the Kontsevich-Soibelman wall-crossing formula, which forms the main ingredient in the notion of a variation of BPS structures \cite{RHDT}. The condition that a family of BPS structures defines a variation is quite tricky to write down for  general BPS structures, and  the finiteness condition of Definition \ref{bps}(ii) will not usually be preserved under wall-crossing. Nonetheless, for the very special class of BPS structures considered in this paper, it is possible to give a straightforward formulation of the wall-crossing formula using birational automorphisms of the twisted torus $\bT_-$.

\begin{defn}
\label{co}
Let $S$ be a complex manifold. A collection of finite, integral and non-degenerate BPS structures $(\Gamma_s,Z_s,\Omega_s)$ indexed by the points  $s\in S$ forms a variation of  BPS structures if
\begin{itemize}
\item[(a)] the charge lattices  $\Gamma_s$ form a local system of abelian groups, and the intersection forms $\<-,-\>_s$ are covariantly  constant;\smallskip

\item[(b)] for any covariantly constant family of elements $\gamma_s\in \Gamma_s$, the central charges $Z_s(\gamma_s)\in \bC$ vary holomorphically;\smallskip

\item[(c)] consider an acute closed subsector $\Delta\subset \bC^*$, and for each $s\in S$ define the anti-clockwise composition over active rays in $\Delta$\begin{equation}\label{wcw} \bS_s(\Delta)=\prod_{\ell\subset \Delta} \bS_s(\ell);\end{equation}
then if $s\in S$ varies in such a way that the boundary rays of $\Delta$ are never active, the birational automorphism $\bS_s(\Delta)$ of the twisted torus $\bT_{s,-}$ is covariantly constant.
\end{itemize}
\end{defn}

For part (c) note that the flat connection on the family of lattices $\Gamma_s$ induces an Ehresmann connection on the family of associated twisted tori $\bT_{s,-}$, and we are asking that the birational automorphism $\bS_s(\Delta)$ is constant with respect to this.

\subsection{Riemann-Hilbert problem}

Let $(\Gamma, Z, \Omega)$ be a finite BPS structure with associated twisted torus  $\bT_-$. Given a ray $\rr\subset \bC^*$ we consider the corresponding half-plane
\begin{equation}
\label{halfplane}\bH_\rr=\{\hbar\in \bC^*:\hbar=z\cdot v \text{ with } z\in \rr\text{ and }\Re(v)>0\}\subset \bC^*.\end{equation}
We shall be dealing with meromorphic functions \[X_\rr\colon \bH_\rr\to  \bT_-.\] Composing with the twisted characters of $\bT_-$ we can equivalently consider functions
\[X_{\rr,\gamma}\colon \bH_\rr\to \bC^*,\qquad X_{\rr,\gamma}(t)=x_\gamma(X_\rr(t)).\]
The Riemann-Hilbert problem associated to the BPS structure $(\Gamma, Z, \Omega)$ depends on the additional  choice of element $\xi\in \bT_-$, which we refer to as the constant term. It reads as follows:

\begin{problem}
\label{dtsect}
For each non-active ray $\rr\subset \bC^*$ we  seek a meromorphic function
\[X_\rr\colon \bH_\rr\to \bT_-,\]
such that the following three conditions are satisfied:
\begin{itemize}

\item[(RH1)]  if two  non-active rays $\rr_1,\rr_2\subset \bC^*$ form the  boundary rays of a convex sector $\Delta\subset \bC^*$ taken in clockwise order then
\[X_{\rr_2}(\hbar)= \bS(\Delta)( X_{\rr_1}(\hbar)), \]
as meromorphic functions of $\hbar\in \bH_{\rr_-}\cap \bH_{\rr_+}$, where $\bS(\Delta)$ is as in \eqref{wcw};
\smallskip

\item[(RH2)]
for each non-active ray $\rr\subset \bC^*$, and each class $\gamma\in \Gamma$, we have
\[\exp(Z(\gamma)/\hbar)\cdot X_{\rr,\gamma}(\hbar) \to \xi(\gamma)\]
as $\hbar\to 0$ in the half-plane $\bH_\rr$;\smallskip

\item[(RH3)] for each non-active ray $\rr\subset \bC^*$, and each class $\gamma\in \Gamma$, there exists $k>0$ such that
\[|\hbar|^{-k} < |X_{\rr,\gamma}
(\hbar)|<|\hbar|^k,\]
for  $\hbar\in \bH_\rr$ satisfying $|\hbar|\gg 0$.
\end{itemize}
\end{problem}

Note that in constrast to the treatment in   \cite{RHDT} we have here allowed the functions $X_\rr$ to be meromorphic. The necessity of doing this was explained in \cite{Barbieri}. It has the unfortunate effect that we lose any hope to prove uniqueness of solutions. It would be interesting to find a natural characterisation of the  solutions to the Riemann-Hilbert problem constructed in this paper.


\section{Quadratic differentials}
\label{aaaa}

In this section we explain how the trajectory structure of the meromorphic quadratic differentials \eqref{diff} define a variation of BPS structures on the space $S$. This can be described completely explicitly and corresponds to the Donaldson-Thomas theory of the A$_2$ quiver. We also discuss the WKB triangulation defined by a saddle-free quadratic differential. For more details on meromorphic quadratic differentials on Riemann surfaces we refer the reader to \cite{BS}.

\subsection{Quadratic differentials}

Let us consider a meromorphic quadratic differential \[\phi(x)=\varphi(x) dx^{\tensor 2}\]
 on the Riemann surface $\bP^1$ having a single pole of order $7$ at the point $x=\infty$, and three simple zeroes.
It is easy to see  \cite[Section 12.1]{BS} that any meromorphic quadratic differential of this type can be put in the  form
\begin{equation}\label{phir}\phi(x)=(x^3 + ax +b) dx^{\tensor 2}\end{equation}
 by applying an automorphism of $\bP^1$. However it will not always be convenient to do this in what follows. Note also that care is required, since rescaling $x$ by a fifth root of unity preserves the form of \eqref{phir} but changes the pair $(a,b)$. 

Away from the zeroes and poles of $\phi(x)$ there is a distinguished local co-ordinate on $\bP^1$
\begin{equation}
\label{w}w(x)=\pm \int_*^{x} \sqrt{\varphi(u)} \, du\end{equation}
 in terms of which $\phi(x)$ takes the form $dw^{\tensor 2}$. Such a co-ordinate is uniquely determined up to transformations of the form $w\mapsto \pm w + c$.
 By definition, the horizontal foliation determined by $\phi(x)$  then consists of the arcs $\Im(w)=\constant$. This foliation has singularities at the zeroes and poles of $\phi(x)$.
 Local computations \cite{strebel} summarised in \cite[Section 3.3]{BS} show that
 \begin{itemize}
\item[(i)] there are three horizontal arcs emanating from each of the three simple zeroes;\smallskip

\item[(ii)] there are five tangent distinguished directions at the pole $x=\infty$, and  an open  neighbourhood $\infty\in U\subset \bP^1$ such that all horizontal trajectories entering $U$  approach $\infty$  along one of the distinguished directions.
\end{itemize}


Following \cite[Section 6]{BS} we take the 
 real oriented blow-up  of the surface $\bP^1$ at the point $\infty$ which is the unique pole of the quadratic differential $\phi(x)$. Topologically the  resulting surface $\bS$ is a disc. The distinguished directions  at the pole determine a subset of five  points $\bM\subset \partial \bS$ of the boundary of this disc; the pair $(\bS,\bM)$ is an example of a marked bordered surface. The horizontal foliation of $\bP^1$ lifts to a foliation on the surface $\bS$, with singularities at the points $\bM\subset \partial \bS$ and the zeroes of $\phi(x)$.


\subsection{Periods and saddle connections}
\label{persad}

Let us associate to a point $s\in S$ the  quadratic differential 
\begin{equation}\label{phi}\phi_s(x)= Q_0(x) dx^{\tensor 2} =  (x^3 + ax +b) dx^{\tensor 2}.\end{equation}
There is a canonically associated double cover
\begin{equation}
\label{peep}p\colon X_s\to \bP^1,\end{equation}
branched at the zeroes and pole of $\phi_s(x)$, on which there is a well-defined global choice of square-root of $\phi_s(x)$. This is nothing but the projectivisation of the affine elliptic curve
\[X_{s}^\circ=\big\{(x,y)\in \bC^2: y^2=x^3+ax+b\big\}.\]
considered before. The square-root is the meromorphic differental $y dx$, which has a single pole at the point at infinity. There is a well-defined group homomorphism
\begin{equation}
\label{stuffy}Z_s\colon H_1(X_s,\bZ)\to \bC, \qquad Z_s(\gamma)=\int_{\gamma} \sqrt{\phi_s(x)}\in \bC.\end{equation}
We shall call a point $s\in S$ generic if the image of $Z_s$ is not contained in a one-dimensional real subspace of $\bC$.

A horizontal trajectory of $\phi_s(x)$  is said to be of finite-length if it never approaches the pole $x=\infty$. In our situation any such trajectory necessarily connects two distinct simple zeroes of $\phi_s(x)$, and is known as a saddle connection. The inverse image of  a saddle connection under the double cover
\eqref{peep}
is a cycle $\gamma$, which can be canonically oriented by insisting  that
$Z_s(\gamma)\in \bR_{>0}$.
This gives a well-defined  homology class in $H_1(X_s,\bZ)$. See \cite[Section 3.2]{BS} for more details.\footnote{For the purposes of comparison with the general situation of \cite{BS} involving the hat-homology group $H_1(X_s^\circ,\bZ)^-$, note that the group $H_1(X_s,\bZ)$ coincides with its $-1$ eigenspace under the action of the covering involution of \eqref{peep}; indeed the $+1$ eigenspace can be identified with the first homology of the quotient $\bP^1$, which vanishes; moreover, puncturing $X_s$ at the inverse image of the pole $\infty\in \bP^1$ also leaves the first homology  unchanged.}

More generally we can consider trajectories of the differential $\phi_s(x)$ of some phase $\theta\in \bR$. By definition these are arcs which make a constant angle $\pi \theta$ with the horizontal foliation. Alternatively one can view them as horizontal trajectories for the rescaled quadratic differential $e^{-2\pi i \theta}\cdot \phi_s(x)$. Once again, these finite-length trajectories $\gamma\colon [a,b]\to \bC$ define homology classes in $ H_1(X_s,\bZ)$, with the orientation convention being that
$Z_s(\gamma)\in \bR_{>0}\cdot e^{\pi i \theta}$.

\subsection{Walls and chambers}

Given a point $s\in S$, the quadratic differential $\phi_s(x)$ is said to be saddle-free if it has no finite-length horizontal trajectories. This is an open condition on the space $S$. 
As  explained in \cite[Section 3.4]{BS},   the horizontal foliation of a saddle-free differential splits the surface $\bP^1$ into a union of domains called  horizontal strips and half-planes. In the present case we obtain five half-planes and two horizontal strips. The resulting trajectory structure  on the blown-up surface $\bS$ is  illustrated in Figure \ref{bel2}. The crosses denote zeroes of the differential, and the black dots are the points of $\bM$.

\begin{figure}[ht]
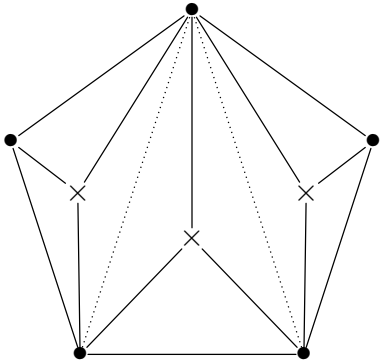

\begin{center}
  \begin{tabular}{c}
\xy /l1.2pc/:
(0,-5)*{\bullet}="1";
(4.76,-1.55)*{\bullet}="2";
(2.94,4.05)*{\bullet}="3";
(-2.94,4.05)*{\bullet}="4";
(-4.76,-1.55)*{\bullet}="5";
(0,1)*{\times}="6";
(3,-.2)*{\times}="7";
(-3,-.2)*{\times}="8";
{"1"\PATH~={**@{-}}'"2"},
{"2"\PATH~={**@{-}}'"3"},
{"3"\PATH~={**@{-}}'"4"},
{"4"\PATH~={**@{-}}'"5"},
{"5"\PATH~={**@{-}}'"1"},
{"2"\PATH~={**@{-}}'"7"},
{"3"\PATH~={**@{-}}'"6"},
{"4"\PATH~={**@{-}}'"6"},
{"5"\PATH~={**@{-}}'"8"},
{"4"\PATH~={**@{-}}'"8"},
{"1"\PATH~={**@{-}}'"8"},
{"1"\PATH~={**@{-}}'"6"},
{"1"\PATH~={**@{-}}'"7"},
{"3"\PATH~={**@{-}}'"7"},
{"1"\PATH~={**@{.}}'"3"},
{"1"\PATH~={**@{.}}'"4"},
\endxy
\end{tabular}
\end{center}
\caption{The separating trajectories of a saddle-free differential  of the form \eqref{phi}. \label{bel2}}
\end{figure}

Taking one trajectory from the interior of each horizontal strip defines a triangulation of the marked bordered surface $(\bS,\bM)$ called the WKB triangulation (see \cite[Section 10.1]{BS} for details). In our case the result is the two dashed edges  in Figure \ref{bel2}. Note that there are exactly two internal edges, and all such triangulations differ by a rotation of the pentagon.
%
As explained in \cite[Section 3.6]{BS}, each of the two horizontal strips contains a unique finite-length trajectory of some  phase in the interval $(0,1)$, and the corresponding classes in $\gamma_i\in H_1(X_s,\bZ)$ determine a basis, whose elements are therefore indexed by the edges of the WKB triangulation.
%
%
%



\subsection{Associated BPS structures}
\label{abps}

 There is a variation of BPS structures  over the space $S$ naturally associated to the family of quadratic differentials  $\phi_s(x)$ defined by \eqref{phi}. 
 
 \begin{definition}\label{hann}The BPS structure $(\Gamma_s,Z_s,\Omega_s)$ associated to a generic point  $s\in S$   is defined as follows:
  \begin{itemize}
\item[(a)]  the charge lattice is $\Gamma_s=H_1(X_s,\bZ)$ with its intersection form $\<-,-\>$;
\item[(b)] the central charge  $Z_s\colon \Gamma_s\to \bC$ is the map \eqref{stuffy};
\item[(c)]  the BPS invariants  $\Omega_s(\gamma)$ are either $0$ or $1$, with $\Omega_s(\gamma)=1$ precisely if the differential $\phi_s(x)$ has a finite-length trajectory of some phase whose associated homology class is $\gamma\in \Gamma_s$.
\end{itemize}

\end{definition}

\begin{remark}
Condition (c) needs modification in the special case that the image of $Z_s$ is contained in a line $\bR\cdot z$, and  the correct definition of the invariants $\Omega_s(\gamma)$  at such non-generic  points is quite subtle (see \cite[Section 6.2]{JS}). This will play no role in what follows however, since what appears in the Riemann-Hilbert problem are the  automorphisms $\bS_s(\Delta)$ associated to sectors by the products  \eqref{wcw}, and by the wall-crossing formula these are locally constant, and hence determined by their values at generic points. See the last paragraph of the proof of Proposition \ref{andy} below. 
\end{remark}

Suppose that $s\in S$ corresponds to a  saddle-free and generic differential $\phi_s$. As explained in the last subsection, the lattice 
$\Gamma_s$ then has a distinguished basis $(\gamma_1,\gamma_2)\subset \Gamma_s$, indexed by the edges of the WKB triangulation, which can be canonically ordered by insisting that $\<\gamma_1,\gamma_2\>=1$. 
Set $z_i=Z(\gamma_i)\in \bC$. 
The orientation conventions discussed above imply that $\Im(z_i)>0$, and the genericity assumption is the statement that $\Im(z_2/z_1)\neq 0$.  

 \begin{prop}
 \label{proop}
Take a point $s\in S$, and let $(\Gamma_s,Z_s,\Omega_s)$ be the corresponding BPS structure. Suppose that the differential $\phi_s$  is saddle-free and generic, and let $(\gamma_1,\gamma_2)\subset \Gamma_s$ be the ordered basis as above.  Define  $z_i=Z(\gamma_i)\in \bC^*$.  Then the BPS invariants are as follows:

 \begin{itemize}
 \item[(a)]  if $\Im (z_2/z_1)>0$ then $\Omega_s(\pm \gamma_1)=\Omega_s(\pm \gamma_2)=1$ with all others zero; 
 \smallskip
 \item[(b)]  if $\Im (z_2/z_1)<0$ then $\Omega_s(\pm \gamma_1)=\Omega_s(\pm(\gamma_1+\gamma_2))=\Omega_s(\pm \gamma_2)=1$ with all others zero.
 \end{itemize}
 \bigskip
 
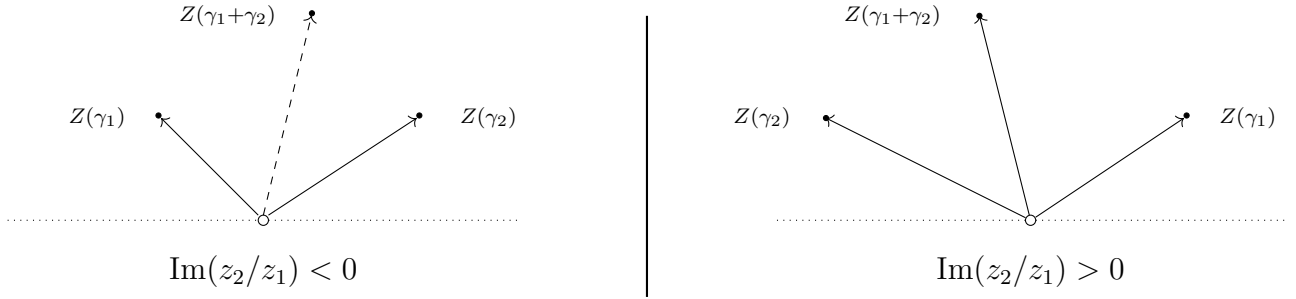
\begin{figure}[ht]
\begin{tikzpicture}[scale=.68]
\draw[dotted] (0,0)--(4.9,0);
\draw[dotted] (5.1,0)--(10,0);
\draw (5,0) circle [radius=0.1];
\draw[->] (5.1,0.1) -- (8,2);
\draw[->] (4.9,0.1) -- (3,2);
\draw[->,dashed](5.0,0.1)--(5.94,4);
\draw (9.4,2) node { $\scriptstyle Z(\gamma_2)$};
\draw (1.75,2) node {$\scriptstyle Z(\gamma_1)$};
\draw (4.3,4) node {$\scriptstyle Z(\gamma_1+\gamma_2)$};
\draw[fill] (8.05,2.05) circle [radius=0.05];
\draw[fill] (2.95,2.05) circle [radius=0.05];
\draw[fill] (5.95,4.05) circle [radius=0.05];
\draw[thick](12.5,-1.5)--(12.5,4);
\draw[dotted] (19.9,0)--(15,0);
\draw[dotted] (20.1,0)--(25,0);
\draw (5,-1) node {$\Im(z_2/z_1)<0$};
\draw (20,-1) node {$\Im(z_2/z_1)>0$};
\draw (20,0) circle [radius=0.1];
\draw[->] (20.1,0.06) -- (23,2);
\draw[->] (19.9,0.05) -- (16,2);
\draw[->] (19.975,0.1) -- (19,4);
\draw (24.25,2) node { $\scriptstyle Z(\gamma_1)$};
\draw (14.75,2) node { $\scriptstyle Z(\gamma_2)$};
\draw (17.3,4) node { $\scriptstyle Z(\gamma_1+\gamma_2)$};
\draw[fill] (23.05,2.05) circle [radius=0.05];
\draw[fill] (16,2) circle [radius=0.05];
\draw[fill] (19,4) circle [radius=0.05];
\end{tikzpicture}
\caption{The BPS structures of Proposition \ref{proop}.}\label{lateagain}
\end{figure}
\end{prop}
 
 \begin{proof}
 This could presumably be  proved by direct analysis of the trajectory structure of the differentials $\phi_s$. Alternatively, it follows  from the results of \cite{BS}, together  with the well-known representation theory of the A$_2$ quiver. In more detail, in the case of the marked bordered surface $(\bS,\bM)$ considered above, the CY$_3$ triangulated category $\cD(\bS,\bM)$ appearing in  \cite{BS} can be identified with the derived category $\cD$ of the Ginzburg algebra of the A$_2$ quiver  \cite[Section 12.1]{BS}. The main result \cite[Theorem 1.2]{BS} then shows that the differentials \eqref{phi} define stability conditions on this category, and moreover, by \cite[Theorem 1.4]{BS},  the  finite-length trajectories of  the  differential are in bijection with the stable objects of the associated stability condition. The result therefore follows from the easy and well known classification of stable representations of the A$_2$ quiver. Note that the basis $(\gamma_1,\gamma_2)\subset \Gamma_s$ correspond to the basis of the Grothendieck group $K_0(\cD)$ given by the classes of the vertex simples. The assumption $\<\gamma_1,\gamma_2\>=1$ then corresponds to a quiver with   a single arrow from vertex $2$ to vertex $1$.
 \end{proof}

In the situation of Proposition \ref{proop}  there is a quadratic refinement
 $g\in \bT_{s,-}$, defined by setting \[g(\gamma_1)=g(\gamma_2)=-1,\]
 which is unique with the property that $g(\gamma)=-1$ for every active class $\gamma\in \Gamma_s$. We use this element and the map \eqref{bloxy} to identify the twisted torus $\bT_{s,-}$ with the standard torus $\bT_{s,+}$. Under this identification the birational automorphism \eqref{garage} becomes the birational automorphism of $\bT_{s,+}$ defined by
 \[\bS(\ell)^*(y_\beta)=y_\beta\cdot \prod_{Z(\gamma)\in \ell}   (1+y_{\gamma})^{\Omega(\gamma)\<\gamma,\beta\>}.\]

Once we have Proposition \ref{proop}, the fact that the  BPS structures of Definition \ref{hann} form a variation of BPS structures comes down to the wall-crossing formula
\begin{equation}
\label{pent}C_{\gamma_1} \circ C_{\gamma_2} = C_{\gamma_2} \circ C_{\gamma_1+\gamma_2} \circ C_{\gamma_1},\end{equation}
where for each class $\alpha\in \Gamma_s$  we defined a birational automorphism $C_\alpha\colon \bT_{s,+}\dashrightarrow \bT_{s,+}$ by
\[C_\gamma^*(y_\beta)=y_\beta\cdot (1+y_\gamma)^{\<\gamma,\beta\>}.\]
This   identity is familiar in cluster theory, and can be viewed as the semi-classical limit of the pentagon identity for the quantum dilogarithm.

%
%
%
%
%


\section{The solution to the  Riemann-Hilbert problem}
\label{soln2}

In this section we  first introduce the Fock-Goncharov co-ordinates on the monodromy space $V$. These are birational maps to the torus $(\bC^*)^2$ and depend on a choice of triangulation of the pentagon. We then prove that, when composed with these maps, the monodromy map for the deformed cubic oscillator gives a solution to the Riemann-Hilbert problem  associated to the BPS structures of Section \ref{aaaa}. In particular, this gives a proof of Theorem \ref{two} from the introduction. Most of the  content of this section is due to Gaiotto, Moore and Neitzke \cite[Section 7]{GMN2}.

\subsection{Fock-Goncharov co-ordinates}
\label{fg}

Let  $(\bS,\bM)$ be a marked bordered surface of the kind appearing in Section \ref{aaaa}, namely  a disc with five marked points on the boundary.  We call two points $p,q\in \bM$ adjacent if they lie in the closure of the same  connected component of $\partial \bS\setminus \bM$. We introduce the space
\[\cV(\bS,\bM)=\big\{\psi\colon \bM \to \bP^1: \psi(p)\neq \psi(q)\text{ for all adjacent points }p,q\in \bM\big\}.\]

Let us now choose a triangulation $T$ of the surface $(\bS,\bM)$ as in Figure \ref{triz}. In particular, the vertices of $T$ are the points of $\bM$. There are precisely five possible choices for $T$, all related by rotations. We denote by $E(T)$ the set of internal edges of  $T$: this set contains exactly two elements. Define \[\cV_T(\bS,\bM)\subset \cV(\bS,\bM)\] to be the open subset consisting of those points for which the  elements $\psi(p)\in \bP^1$ associated to the two  ends of any edge of $T$ are distinct.

%

\begin{figure}
\begin{center}
  \begin{tabular}{c}
\xy /l1pc/:
(0,-5)*{\bullet}="1";
(4.76,-1.55)*{\bullet}="2";
(2.94,4.05)*{\bullet}="3";
(-2.94,4.05)*{\bullet}="4";
(-4.76,-1.55)*{\bullet}="5";
{"1"\PATH~={**@{-}}'"2"},
{"2"\PATH~={**@{-}}'"3"},
{"3"\PATH~={**@{-}}'"4"},
{"4"\PATH~={**@{-}}'"5"},
{"5"\PATH~={**@{-}}'"1"},
{"1"\PATH~={**@{-}}'"3"},
{"1"\PATH~={**@{-}}'"4"},
\endxy
\end{tabular}
\end{center}
\caption{A triangulation of the marked bordered surface $(\bS,\bM)$.\label{triz}}
\end{figure}

For each internal edge $e\in E(T)$ there is
a holomorphic map
\begin{equation}
\label{cr}X_e\colon \cV_T(\bS,\bM)\to \bC^*\end{equation}
obtained by taking the cross-ratio\[X_e=\CR(a_1,a_2,a_3,a_4)=\frac{(a_1-a_2)(a_3-a_4)}{(a_1-a_4)(a_2-a_3)},\]
 of the four points $a_i=\psi(i)\in \bP^1$ corresponding to the vertices of the two triangles adjoining the edge $e$. More precisely,  the points $\psi(i)$ should be taken in anti-clockwise order starting with one of the two ends of $e$: there are two possible such orderings, but the two choices give the same value for the cross-ratio.

Combining the maps $X_e$ associated to the two internal edges of $T$ gives a holomorphic map
\[X_T\colon \cV_T(\bS,\bM)\to (\bC^*)^{E(T)}\isom (\bC^*)^2.\]
The invariance property of the cross-ratio shows that this descends to the quotient space
\[V_T(\bS,\bM)=\cV_T(\bS,\bM)/\PGL_2\subset V(\bS,\bM)=\cV(\bS,\bM)/PGL_2,\]
and it is easy to see that the resulting map
\begin{equation}\label{xbox}X_T\colon V_T(\bS,\bM) \to (\bC^*)^{E(T)}\isom (\bC^*)^2\end{equation}
is an isomorphism.
The components of this map  are called the Fock-Goncharov co-ordinates for the triangulation $T$.

\subsection{Solution to the Riemann-Hilbert problem}

Take a  point $s\in S$ and consider the corresponding quadratic differential \eqref{phi}. We would like to solve Problem \ref{dtsect} for the associated BPS structure $(\Gamma_s,Z_s,\Omega)$ of Definition \ref{hann}. As explained in Section \ref{abps}, there  is a distinguished quadratic refinement of the form $\<-,-\>_s$, and we can use the associated map \eqref{bloxy} to identify the twisted torus $\bT_{s,-}$ with the standard torus $\bT_s=\bT_{s,+}$. The Riemann-Hilbert problem then depends on a choice of a constant term   $\xi\in \bT_{s}$, and involves constructing meromorphic maps
\begin{equation}
\label{yrr}X_\rr\colon \bH_\rr \to \bT_{s}\end{equation}
for all non-active rays $\rr\subset \bC^*$, where $\bH_r$ is the half-plane defined in \eqref{halfplane}.

Let us assume first that the chosen point  $\xi \in \bT_{s}$ lies in the image of the abelian holonomy map $\Theta_s\colon \MM_s\to \bT_s$, so that we can write $\xi_s=\Theta_s(m)$ for some point $m\in \MM_s$.   We will construct a suitable map \eqref{yrr} by sending $\hbar\in \bH_r$ to the Fock-Goncharov co-ordinates of the monodromy   of the deformed cubic oscillator \eqref{de}-\eqref{pot} defined by the point $m\in \MM$. More precisely, we will take the  Fock-Goncharov co-ordinates  defined by the WKB triangulation of the quadratic differential 
\begin{equation}
\label{rescaled}\lambda^{-2} \cdot Q_0(x) dx^{\tensor 2}= \lambda^{-2}\cdot (x^3+ax+b)\cdot dx^{\tensor 2},\end{equation}
where $\lambda\in r$ is an arbitrary point of the given ray.  Note that the assumption  that the  ray $r\subset \bC^*$ is non-active is equivalent to the statement that the differential \eqref{rescaled}  is saddle-free for $\lambda\in r$.


One confusing point requires a little care. For each $\lambda\in \bC^*$ let us denote by $(\bS(\lambda),\bM(\lambda))$ the marked bordered surface determined by the rescaled differential \eqref{rescaled}. We can always take the underlying surface $\bS(\lambda)$ to be the unit disc in $\bC$, and the marked points  $\bM(\lambda)$  are then positive real multiples of the fifth roots of $\lambda^2$ (see for example \cite[Section 3.2]{AB}). Given a ray $r\subset \bC^*$, we will also use the notation $(\bS(r),\bM(r))$ for the marked bordered surface corresponding to an arbitrary point $\lambda\in r$.
It is important to note that if two rays $\rr_1,\rr_2\subset \bC^*$ lie in the same half-plane  then there is a canonical identification between the two surfaces $(\bS(\rr_i),\bM(\rr_i))$. In concrete terms, this is because the fifth root function is single-valued on any given half-plane. 

Returning to our non-active ray $\rr\subset \bC^*$, we can consider the associated  WKB triangulation $T(\rr)$ of the marked bordered surface $(\bS(\rr),\bM(\rr))$ . Since the internal edges of $T(\rr)$ are labelled by basis elements of the group $\Gamma_s$, the map \eqref{xbox} can be interpreted as a birational isomorphism
\begin{equation}\label{both}X_{T(\rr)}\colon V(\bS(\rr),\bM(\rr)) \dashrightarrow \bT_+.\end{equation}
On the other hand,  as in \cite[Section 5.3]{AB}, the Stokes sectors of the equation \eqref{diff} are in natural bijection with the points of $\bM(\hbar)$. As discussed above, since $\hbar\in \bH_\rr$, there is a canonical identification  between the surfaces $(\bS(r),\bM(r))$ and $(\bS(\hbar),\bM(\hbar))$. We can then compose the monodromy map
\[F\colon \bH_r\to V(\bS(r),\bM(r)), \qquad \hbar\mapsto F(\hbar)(m),\]
with the map \eqref{both} to obtain the required map $X_r\colon \bH_r\to \bT_s$. We now proceed to check the conditions (RH1) -- (RH3) of Problem \ref{dtsect}.

\subsection{Jumping}
 Let us  start with the jumping condition (RH1). Take a point $s\in S$ and let $\ell\subset \bC^*$ be an active ray for the corresponding BPS structure $(\Gamma_s,Z_s,\Omega_s)$. Consider non-active rays  $r_-$ and $r_+$ which are small anti-clockwise and clockwise deformations of the ray $\ell$.  We can identify the marked bordered surfaces $(\bS(r_{\pm}), \bM(r_\pm))$ associated to the rays $r_\pm$ with the surface $(\bS(\ell),\bM(\ell))$ as above, and hence also identify the spaces $V(r_\pm)$ with the fixed space $V(\ell)$. Let $T_\pm=T(r_\pm)$ be the WKB triangulations of the surface $(\bS(\ell),\bM(\ell))$ defined by the non-active rays $r_\pm$, and let $X_{T_\pm}\colon V(\ell)\dashrightarrow \bT$ be the associated  Fock-Goncharov co-ordinates. 

\begin{prop}
\label{andy}
The two systems of co-ordinates are related by
\[X_{T_+}=\bS(\ell)\circ X_{T_-} .\]
\end{prop}

\begin{proof}
Suppose first that $s\in S$ is generic. For $\lambda\in \ell$ the differential \eqref{rescaled} has a unique saddle connection, and  the WKB triangulations $T_\pm$ for the saddle-free differentials \eqref{rescaled} corresponding to $\lambda_\pm \in r_\pm$ differ by a flip in a single edge.  This situation is discussed in detail in \cite[Section 10.3]{BS}. 

Without loss of generality we can assume that the triangulation $T_+$ is as in Figure \ref{triz}. There are  two cases, depending on which edge of the triangulation is being flipped. These are illustrated in Figures \ref{ohgod} and \ref{ohgod2}.
We label the vertices of the pentagon in clockwise cyclic order as shown. In each case, the left-hand picture illustrates $T_-$, and the right-hand picture is $T_+$. The two 
edges $e^+_1,e^+_2$ of the triangulation $T_+$ are labelled by classes $\gamma_1,\gamma_2\in \Gamma$. Since $e_1^+,e_2^+$ appear as adjacent edges in clockwise  order in the unique triangle of $T_+$ which contains them both, the sign correction to \cite[Lemma 10.3]{BS}  mentioned in the proof of Proposition  \ref{proop}  shows that $\<\gamma_1,\gamma_2\>=1$. Let us now consider the two cases in turn.

\begin{figure}
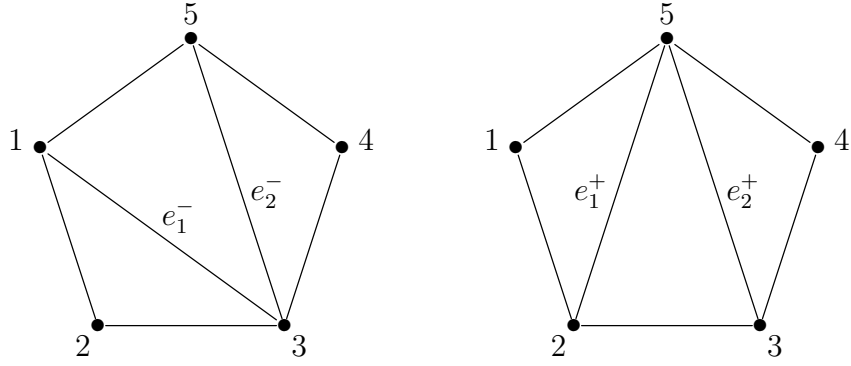

\begin{center}
  \begin{tabular}{c}
\xy /l1pc/:
(0,-5)*{\bullet}="1";
(4.76,-1.55)*{\bullet}="2";
(2.94,4.05)*{\bullet}="3";
(-2.94,4.05)*{\bullet}="4";
(-4.76,-1.55)*{\bullet}="5";
{"1"\PATH~={**@{-}}'"2"},
{"2"\PATH~={**@{-}}'"3"},
{"3"\PATH~={**@{-}}'"4"},
{"4"\PATH~={**@{-}}'"5"},
{"5"\PATH~={**@{-}}'"1"},
{"2"\PATH~={**@{-}}'"4"},
{"1"\PATH~={**@{-}}'"4"},
(0,-5.8)*{5};
(5.52,-1.8)*{1};
(3.41,4.7)*{2};
(-3.41,4.7)*{3};
(-5.52,-1.8)*{4};
(-2.4,-.2)*{e^-_2};
(.4,.6)*{e^-_1};
(-15,-5)*{\bullet}="6";
(-10.24,-1.55)*{\bullet}="7";
(-12.06,4.05)*{\bullet}="8";
(-17.94,4.05)*{\bullet}="9";
(-19.76,-1.55)*{\bullet}="10";
{"6"\PATH~={**@{-}}'"7"},
{"7"\PATH~={**@{-}}'"8"},
{"8"\PATH~={**@{-}}'"9"},
{"9"\PATH~={**@{-}}'"10"},
{"10"\PATH~={**@{-}}'"6"},
{"6"\PATH~={**@{-}}'"8"},
{"6"\PATH~={**@{-}}'"9"},
(-15,-5.8)*{5};
(-9.48,-1.8)*{1};
(-11.59,4.7)*{2};
(-18.41,4.7)*{3};
(-20.52,-1.8)*{4};
(-12.6,-.2)*{e^+_1};
(-17.4,-.2)*{e^+_2};
\endxy
\end{tabular}
\end{center}
\caption{Flipping the  triangulation: first case.\label{ohgod}}
\end{figure}

In the first case,  illustrated in Figure \ref{ohgod}, the edge $e_1^+$ is being flipped. According to \cite[Proposition  10.4]{BS}, the edges $e^-_1,e^-_2$ are labelled by the classes $-\gamma_1,\gamma_1+\gamma_2$. 
The Fock-Goncharov co-ordinates  are 
\[X^+_1=X_{T_+}^*(y_{\gamma_1})=\CR(a_5,a_1,a_2,a_3), \qquad X^+_2=X_{T_+}^*(y_{\gamma_2})=\CR(a_5,a_2,a_3,a_4),  \]
on the right, whereas on the left we have
\[X^-_1=X_{T_-}^*(y_{-\gamma_1})=\CR(a_1,a_2,a_3,a_5)=(X^+_1)^{-1},\]\[ X^-_2=X_{T_-}^*(y_{\gamma_1+\gamma_2})=\CR(a_5,a_1,a_3,a_4)=X^+_2\cdot \big(1+(X^+_1)^{-1}\big)^{-1},\]
where  we used the easily-checked identity
\[\CR(a_5,a_1,a_3,a_4)=\CR(a_5,a_2,a_3,a_4)\cdot \big(1+\CR(a_5,a_1,a_2,a_3)^{-1}\big)^{-1}.\]
Thus we have
\[X_{T_+}^*(y_{\gamma_1})=X_{T_-}^*(y_{\gamma_1}), \qquad X_{T_+}^*(y_{\gamma_2})=X_{T_-}^*\big(y_{\gamma_2}(1+y_{\gamma_1})\big).\]

Consider the central charges $Z_\pm =\lambda_\pm\cdot Z_s$ with $\lambda_\pm \in r_\pm$. By definition of the classes $\gamma_i\in \Gamma$ associated to the triangulation $T_+$ we have $\Im Z_+(\gamma_1)>0$. Since the rotation from $\lambda\in r_-$ to $\lambda\in r_+$ is clockwise, the  central charges $\lambda^{-1}\cdot Z(\gamma)$ rotate anti-clockwise, and it follows that for $\lambda\in \ell$ the central charge $\lambda^{-1}\cdot Z(\gamma_1)$ lies on the positive real axis. Thus the corresponding wall-crossing automorphism $\bS(\ell)=C_{\gamma_1}$ satisfies
\[\bS(\ell)^*(y_{\gamma_1})=y_{\gamma_1}, \qquad \bS(\ell)^*(y_{\gamma_2})=y_{\gamma_2}\cdot (1+y_{\gamma_1}),\]
and we therefore conclude that $X_{T_+}^*=X_{T_-}^*\circ \bS(\ell)^*$ as required.

In the second case,  illustrated in Figure \ref{ohgod2}, the edge $e_2^+$ is being flipped. This time \cite[Proposition  10.4]{BS} shows that the edges $e^-_1,e^-_2$ are labelled by the classes $\gamma_1, -\gamma_2$. 
The Fock-Goncharov co-ordinates  on the right are as before. On the left they are
\[ X^-_1=X_{T_-}^*(y_{\gamma_1})=\CR(a_5,a_1,a_2,a_4)=X_1^+\cdot (1+(X_2^+)),\]
\[X^-_2=X_{T_-}^*(y_{-\gamma_2})=\CR(a_4,a_5,a_2,a_3)=(X_2^+)^{-1},\]where we used
\[\CR(a_5,a_1,a_2,a_4)=\CR(a_5,a_1,a_2,a_3)\cdot \big(1+\CR(a_5,a_2,a_3,a_4)\big).\]
Thus we have
\[X_{T_+}^*(y_{\gamma_1})=X_{T_-}^*\big( y_{\gamma_1}(1+y_{\gamma_2})^{-1}\big),\qquad X_{T_+}^*(y_{\gamma_2})=X_{T_-}^*(y_{\gamma_2}).\]
This time the wall-crossing automorphism $\bS(\ell)=C_{\gamma_2}$ is given by
\[\bS(\ell)^*(y_{\gamma_1})=y_{\gamma_1}\cdot (1+y_{\gamma_2})^{-1},\qquad \bS(\ell)^*(y_{\gamma_2})=y_{\gamma_2},\]
so we again find that $X_{T_+}^*=X_{T_-}^*\circ \bS(\ell)^*$.

\begin{figure}
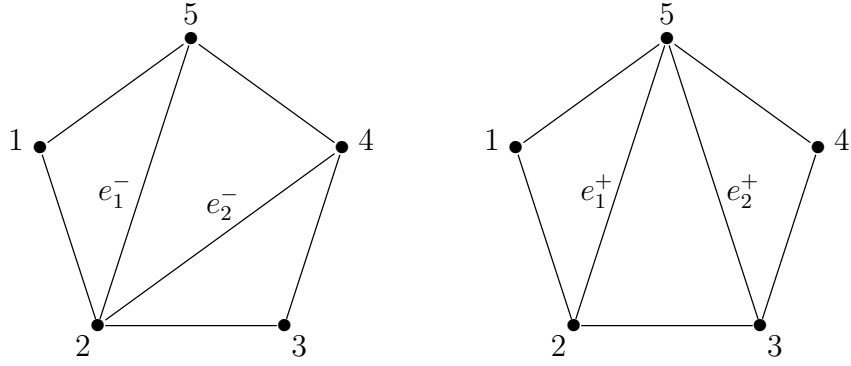

\begin{center}
  \begin{tabular}{c}
\xy /l1pc/:
(0,-5)*{\bullet}="1";
(4.76,-1.55)*{\bullet}="2";
(2.94,4.05)*{\bullet}="3";
(-2.94,4.05)*{\bullet}="4";
(-4.76,-1.55)*{\bullet}="5";
{"1"\PATH~={**@{-}}'"2"},
{"2"\PATH~={**@{-}}'"3"},
{"3"\PATH~={**@{-}}'"4"},
{"4"\PATH~={**@{-}}'"5"},
{"5"\PATH~={**@{-}}'"1"},
{"1"\PATH~={**@{-}}'"3"},
{"3"\PATH~={**@{-}}'"5"},
(0,-5.8)*{5};
(5.52,-1.8)*{1};
(3.41,4.7)*{2};
(-3.41,4.7)*{3};
(-5.52,-1.8)*{4};
(-1,0.2)*{e^-_2};
(2.4,-.2)*{e^-_1};
(-15,-5)*{\bullet}="6";
(-10.24,-1.55)*{\bullet}="7";
(-12.06,4.05)*{\bullet}="8";
(-17.94,4.05)*{\bullet}="9";
(-19.76,-1.55)*{\bullet}="10";
{"6"\PATH~={**@{-}}'"7"},
{"7"\PATH~={**@{-}}'"8"},
{"8"\PATH~={**@{-}}'"9"},
{"9"\PATH~={**@{-}}'"10"},
{"10"\PATH~={**@{-}}'"6"},
{"8"\PATH~={**@{-}}'"6"},
{"6"\PATH~={**@{-}}'"9"},
(-15,-5.8)*{5};
(-9.48,-1.8)*{1};
(-11.59,4.7)*{2};
(-18.41,4.7)*{3};
(-20.52,-1.8)*{4};
(-12.8,-.2)*{e^+_1};
(-17.4,-.2)*{e^+_2};
\endxy
\end{tabular}
\end{center}
\caption{Flipping the triangulation: second case.\label{ohgod2}}
\end{figure}

Consider now the case when $s\in S$ is not generic. The corresponding BPS structure has exactly two  active rays $\pm \ell$. Let us deform the  point $s\in S$  to a nearby generic point $t\in S$. Under this deformation the ray $\ell$ splits into two or three rays $\ell_i$ as in Figure \ref{lateagain}, but for $t$ close enough to $s$ these rays $\ell_i$ will  be contained in the sector bounded by the non-active rays $r_\pm$. The triangulations associated to the rays $r_\pm$ do not change under the deformation, and the wall-crossing formula \eqref{wcw} shows that the automorphism $\bS_s(\ell)$ is the clockwise composition of the automorphisms $\bS_t(\ell_i)$. The result for the non-generic point $s\in S$ now follows by applying the same result for the generic point $t\in S$ to each of the rays $\ell_i$. 
\end{proof}

\subsection{Behaviour as $\hbar\to 0$}
\label{wkb}

To verify condition (RH2) of Problem \ref{dtsect} we must show that the map
\[X_r\colon \bH_r\to \bT_s, \qquad X_r(\hbar)=X_{T(r)}(F(\hbar)(m)),\]
has the correct asymptotics as $\hbar\to 0$. As explained above, given an edge $e$ of the WKB triangulation for the differential \eqref{rescaled}, there is a corresponding class $\gamma_e\in \Gamma_s$ defined by the saddle connection crossing the associated horizontal strip. The statement we want is that
\[X_r(\gamma_e)(\hbar)\sim \exp(-Z(\gamma_e)/\hbar) \cdot \xi(\gamma_e),\]
as $\hbar\to 0$ in the half-plane $\bH_r$. To simplify matters a little,  we can, by applying the $\bC^*$ action on $\MM$ used in the proof of Proposition \ref{porpoise}, assume that the ray $r=\bR_{>0}$ is the positive real axis, and hence that the differential $\phi_s$ is saddle-free. 

Let us then state the required result as concretely as possible. Consider a  deformed cubic oscillator of the form \eqref{de}-\eqref{pot}, and assume that the corresponding quadratic differential $Q_0(x) dx^{\tensor 2}$  on $\bC$ is saddle-free. The horizontal trajectory structure of this differential then defines a WKB triangulation  of the regular pentagon with vertices at the fifth roots of unity. Moreover, each of the two edges $e_i$ of this triangulation $T$ is naturally labelled by a  class $\gamma_{i}$ in the homology group $H_1(X_s,\bZ)$  of  \eqref{homo}. We set
\[z_i=\int_{\gamma_i} \sqrt{Q_0(x)} \, dx\in \bC, \qquad \xi_i=\exp\bigg(\int_{\gamma_i} \frac{-Q_1(x)dx}{2\sqrt{Q_0(x)}}\bigg)\in \bC^*.\]
By definition of the orientation of the classes $\gamma_i$ we have $\Im(z_i)>0$.

When $\hbar\in \bR_{>0}$ the Stokes sectors for our equation \eqref{de} are centered on the rays spanned by the fifth roots of unity. Thus for all $\Re(\hbar)>0$ we can continuously  identify the Stokes sectors  with the vertices of the triangulation $T$. Using this identification, we let $X_i(\hbar)$ denote the Fock-Goncharov co-ordinate corresponding to the edge $e_i$ of the triangulation $T$, for the point of the monodromy manifold defined by the subdominant solutions  of the equation \eqref{de} .

\begin{thm}
\label{dav}
The Fock-Goncharov co-ordinates $X_i(\hbar)$ satisfy
\[\exp(z_i/\hbar)\cdot X_i(\hbar) \to \xi_i,\]
 as $\hbar\to 0$ in any closed subsector of the  half-plane $\Re(\hbar)>0$.
\end{thm}

We defer the  proof of this result to the Appendix (written by Davide Masoero).

\subsection{Behaviour as $h\to \infty$}
\label{final4}
The final step is to check the condition (RH3). In fact we will prove more, namely that, for a fixed point $m\in \MM$,  the point $F(\hbar)(m)$ of the monodromy manifold tends to a well-defined limit point. To see this, we will 
use the homogeneity of the potential \eqref{pot} under the $\bC^*$ action of Remark \ref{actions}(b). 

\begin{prop}
For any point $m\in \MM$  the monodromy $F(\hbar)(m)\in V$ has a well-defined limit as $\hbar\to \infty$ in a fixed half-plane. This limit  is independent of $m\in \MM$ and is one of the two fixed points of the $\bZ/5\bZ$ action of Remark \ref{euros}.\end{prop}

\begin{proof}Let us consider the partial compactification  \[\bar{\MM}=\big\{(a,b,q,p,r)\in \bC^5: p^2=q^3+aq+b\big\}\]
of the space $\MM$, obtained by dropping the vanishing discriminant condition. We denote by $0\in \bar{\MM}$ the point where all co-ordinates vanish. For a given $\hbar\in \bC^*$ the monodromy map $F(\hbar)$ extends to a holomorphic map
\[\bar{F}(\hbar)\colon \bar{\MM}\to V,\]
subject to the usual warning that this depends on a choice of fifth root of $\hbar^2$.

Consider  the action of $\bC^*$  on $\bar{\MM}$  of Remark \ref{actions}(b) which scales the co-ordinates $(a,b,q,p,r)$ with weights $(4,6,2,3,1)$ respectively.
 Note that if we also rescale $\hbar$ with weight $5$, and $x$ with weight 2, then the equation \eqref{de}  is unchanged.  Thus for all points $(a,b,q,p,r)\in \bar{\MM}$
 \[\bar{F}(\hbar)(a,b,q,p,r)=\bar{F}(\lambda^{5}\hbar)(\lambda^{4}a,\lambda^{6}b,\lambda^{2}q,\lambda^{3}p,\lambda^1 r).\] Taking $\lambda^{5}\cdot \hbar=1$, and sending  $\hbar\to \infty$ in a fixed half-plane, it follows that the monodromy of $F(\hbar)(m)$ tends to the finite limit $\bar{F}(1)(0)$, which is the monodromy of the equation
\[y''(x)=\bigg({x^3}+ \frac{3}{4x^2}\bigg) y(x).\]

For the final claim, 
note that the above $\bC^*$ action  induces an action of the fifth roots of unity $\mu_5\subset \bC^*$, which leaves $\hbar$ invariant. Since this action rescales $x$ by an element of $\mu_5$,  the monodromy map $F(\hbar)$ intertwines this action with the $\bZ/5\bZ$ action on $V$ of Remark \ref{euros}. But the special point $0\in \bar{\MM}$ is clearly fixed by the $\mu_5$ action, so its image is also a fixed point. \end{proof}

%

%


\begin{appendix}

\section{Asymptotic analysis of the functions $X_i(\hbar)$  by Davide Masoero}

The Appendix is dedicated to the computation of the full asymptotic expansion of the function $e^{\frac{z_i}{\hbar}} X_i(\hbar), i=1,2$.
As a particular case, we prove Theorem 7.1 of the main text.

The Appendix is organised as follows. In Section \ref{sec:wkb} we study the asymptotic expansion of solutions
of the deformed cubic oscillator according to the Complex WKB method.
In Section \ref{sec:Xs} we lift the formal WKB solutions to the elliptic curve $X_s$
punctured at the branch points. Finally, in Section \ref{sec:proof} we prove Theorem \ref{thm:appendix}.
We made this Appendix self-contained. The proofs are lengthy but complete, and the reader experienced in the complex
WKB method may want to skip all proofs until the last section.

\subsection*{Acknowledgements} The author has benefitted from useful conversations with
Anna Barbieri, Tom Bridgeland, Akane Nakamura, and J\"org Teschner. He is partially supported by the FCT Project PTDC/MAT-PUR/30234/2017, ``Irregular connections on algebraic curves and Quantum Field Theory" and by the FCT Investigator grant IF/00069/2015, ``A mathematical framework for the ODE/IM correspondence".

\subsection{Statement of the result}

In order to state the our main result, we begin by fixing some notation.
Recall that we deal with the small $\hbar$ limit of the deformed cubic oscillator,
$y''(x)=Q(x)$, $Q(x)=\hbar^{-2}Q_0(x)+\hbar^{-1}Q_1(x)+Q_0(x)$, where
\begin{eqnarray} \label{eq:potApp}  Q_0=x^3+ax+b, \; Q_1(x)=\frac{p}{x-q}+r, \;
Q_2(x)=\frac{3}{4(x-q)^2}+\frac{r}{2p(x-q)}+\frac{r^2}{4p^2} .
\end{eqnarray}
Here $r$ is an arbitrary complex number, while the parameters $(q,p)$ are assumed to belong to
the affine elliptic curve $X_s^\circ=\lbrace p^2=Q_0(q) \rbrace$, punctured at the three branch points $p=0$. We call $B$ the set of the three branch points.
The projectivization of the affine elliptic curve is called $X_s$, and it is endowed with the canonical double cover
$p: X_s \to \mathbb{P}^1$, which is branched at $B$ and at infinity. 

The asymptotic expansion of the Fock-Goncharov co-ordinates is naturally written in terms of complete elliptic integrals over $X_s$.
The following meromorphic abelian differentials $\alpha(x)\,dx$ on $X_s$ are relevant to our analysis
\begin{eqnarray} \nonumber
 \alpha_0(x) &=&\sqrt{Q_0(x)},\quad
 \alpha_1(x)= -\frac{Q_0'(x)}{4 Q_0(x)} + \widetilde{\alpha_1}, \mbox{ with } \widetilde{\alpha_1}=\frac{ Q_1(x)}{2\sqrt{Q_0(x)}}\\ \nonumber
 \alpha_2(x) &=&\frac{1}{ 2 \sqrt{Q_0(x)}} \big(Q_2(x)-\alpha_1'(x)-\alpha_1^2(x) \big) \\ \label{eq:alphaforms}
 \alpha_k(x) &=&-\frac{1}{ 2 \sqrt{Q_0(x)}} \big(\alpha_{k-1}'(x)+ \sum_{j=1}^{m-1} \alpha_j(x) \alpha_{k-j}(x) \big) , \; k\geq 3 .
\end{eqnarray}

The cycles along which the above differentials are evaluated are the cycles $\gamma_{i}, i=1,2$ defined in the main text.\footnote{There is a subtle difference with respect to the main text relative to
the cycles $\gamma_i,i=1,2$. In the main text $\gamma_{1,2}$ are elements (a basis of) $H_1(X_s,\mathbb{Z})$. In our setting, they
are elements of the hat-homology $H_1(X_s^\circ \setminus B, \mathbb{Z})^-$; this is the $-1$ eigenspace of $H_1(X_s^\circ \setminus B, \mathbb{Z})$
under the action of the elliptic involution $p \to -p$, as defined in \cite{BS}.
The embedding $\iota: X_s^\circ \setminus B \to X_s$
induces an isomorphism $\iota^* :H_1(X_s^\circ \setminus B, \mathbb{Z})^- \to  H_1(X_s,\mathbb{Z})$; 
this follows from the same reasoning used in the main text,
in the footnote in Section \ref{persad}, to show that $H_1(X_s^\circ , \mathbb{Z})^-$ and $H_1(X_s,\mathbb{Z})$ are isomorphic.
Under the isomorphism $\iota^*$, the cycles $\gamma_{1,2}$
of this Appendix coincide with the ones defined in the main text.}

Finally, we use the following formalism in dealing with asymptotic expansions in sectors of the complex
$\hbar$ plane.
For every $\theta \in [0,\frac{\pi}{2}[$ and $\hbar_\theta >0$, we let $S_{\theta,\hbar_\theta}$ denote
the sector $ \lbrace |\arg \hbar |\leq \theta,0< |\hbar|\leq \hbar_\theta \rbrace$.
For any formal power series $A =\sum_{k\geq0} a_k \hbar^k \in \mathbb{C}[[\hbar]]$, we denote
by $A_m:=\sum_{k=0}^m a_k \hbar^k$ its $m-th$ truncation.
\begin{definition}
Let $f$ be a function on the sector $\Re \hbar>0 $, and $A$ a formal power series.
We say that $f$ is asymptotic to $A$, and we write
$f \approx A$ on $\Re \hbar >0$, if for every $\theta \in [0,\frac{\pi}{2}[$
there exists a sequence of positive constants $\hbar_{\theta}$, $C_{\theta,m}, m\geq0$ such that
$|f(\hbar)-A_m(\hbar)|\leq C_{\theta,m} |\hbar|^{m+1}$ for all $\hbar \in S_{\theta,\hbar_\theta}$. 
\end{definition}
\begin{theorem}\label{thm:appendix}
 Assume that $Q_0$ is saddle-free.
 The functions $ e^{\frac{z_i}{\hbar}} X_i(\hbar), i=1,2$ have the following asymptotic expansion
 \begin{equation}\label{eq:thesis}
  e^{\frac{z_i}{\hbar}} X_i(\hbar) \approx \xi_i \exp{\left( \sum_{k=1}^{\infty} \hbar^{k} C_{k,i} \right)} \mbox{ on } \Re \hbar>0,
 \end{equation}
where
\begin{equation}\label{eq:terms}
z_i=\int_{\gamma_i}\alpha_0(x) dx,
\quad \xi_i= e^{-\int_{\gamma_i}\big(\widetilde{\alpha}_1(x)+\frac{1}{2(x-q)} \big)dx},
\quad C_{k,i}= -  \int_{\gamma_i} \alpha_{k+1}(x) \, dx , \, k\geq1.
\end{equation}
In particular, $$\lim_{\hbar \to 0} e^{  \frac{z_i}{\hbar}} X_i(\hbar) = \xi_i \mbox{ in any closed subsector of } \Re \hbar >0 .$$
\end{theorem}
\begin{remark}
 Before we tackle the proof of the Theorem, we check that all terms in the asymptotic expansions \eqref{eq:terms}
attain the same value for every path $\gamma$ in the homology class $\gamma_i \in  H_1(X_s^\circ \setminus B, \mathbb{Z})^- ,i=1,2$, even though
 the forms $\alpha_k$ are possibly singular at the points $(q,\pm p)$.

 This is indeed the case. In fact,
 \begin{itemize}
  \item[($k=0$)] The only singular point of $\alpha_0(x) dx$ is $\infty$. Furthermore its residue is zero.
 \item[($k=1$)] $e^{-\int_{\gamma_i}\big(\widetilde{\alpha}_1(x)+\frac{1}{2(x-q)} \big)dx}$ is well-defined, since 
 all residues of the form $\big(\widetilde{\alpha}_1(x)+\frac{1}{2(x-q)} \big)dx$ are integer-valued, as it is shown
 in the main text after formula (42).
 \item[($k\geq2$)]For the forms $\alpha_k(x)dx, k \geq 2$, the
 residue at $(q,\pm p)$ vanishes, see Corollary \ref{cor:residueatq} below.
 \end{itemize}
\end{remark}

 \subsection{WKB analysis of the deformed cubic oscillator}\label{sec:wkb}
Our approach is based
on transforming a linear ODE of the second order into an integral equation of Volterra type, following \cite{erdelyi10}:
We consider a second order scalar linear ODE of the form
\begin{equation}\label{eq:schrodingergen}
 y''(x)=Q(x) y(x), \quad x \in \mathbb{C}
\end{equation}
where $Q(x)$ may depend on additional parameters, and a putative approximate solution $Y(x)$, which we suppose to be of such a form that
\begin{equation}\label{eq:t}
 u(x)=\frac{y(x)}{Y(x)}
\end{equation}
is well-defined and approximately $1$ in a certain domain of $\mathbb{C}$ to be later specified.

Defining the forcing term
\begin{equation}\label{eq:forcing}
 F(x)=Q(x)-\frac{Y''(x)}{Y(x)}
\end{equation}
the equation \eqref{eq:schrodingergen} for $y(x)$, when rewritten in terms of the function $u(x)$ defined by \eqref{eq:t}, becomes
\begin{equation}\label{eq:usch}
 \frac{d}{dx}\big( Y^2(x) u'(x) \big)- Y^2(x) F(x) u(x)=0
\end{equation}
We fix a point $x'$ in the Riemann sphere, the boundary conditions $u'(x')=0,u(x')=1$, and a piece-wise smooth integration path
$\gamma$ connecting $x'$ to another point $x \in \mathbb{C}$.
Integrating twice equation \eqref{eq:usch}, $u(x)$ is proven to solve the following integral equation
\begin{equation}\label{eq:integral}
 u(x)=1-\int_{x',\gamma}^x K(x,s) F(s) u(s) d s \,, \qquad
 K(x,s)=\int_{s,\gamma}^x \frac{Y^2(s)}{Y^2(r)} d r \; ,
\end{equation}
provided the above integral converges absolutely.

Conversely, given any continuous solution $u(x)$ of the latter integral equation, the function
$y(x):=u(x) Y(x)$ solves \eqref{eq:schrodingergen} and satisfies the (possibly singular) Cauchy problem
\begin{equation*}
 \lim_{x \to x', x \in \gamma} \frac{y(x)}{Y(x)}=1 , \quad \lim_{x \to x', x \in \gamma} \frac{y'(x)}{Y'(x)}=1
\end{equation*}

\begin{remark}
Given a solution $u$ of \eqref{eq:integral}, the corresponding solution $y$ of \eqref{eq:schrodingergen}
is a priori only defined on the trajectory of the curve $\gamma$. It can however be analytically extended to any open
simply connected domain of analyticity of $Q$ which intersects the trajectory of $\gamma$. To be more precise: let $D \subset \mathbb{C}$ be
an open simply connected domain such that $Q_{|D}$ is analytic, and assume that, for some $t_1<t_2$, $\gamma(t) \in D$ for all $t\in]t_1,t_2[$;
there exists a unique solution $\widehat{y}:D \to \mathbb{C}$ such that $\widehat{y}(\gamma(t))=u(\gamma(t))$ for all $t \in ]t_1,t_2[$.
\end{remark}

\subsection{Formal WKB solutions}
 
We are interested in studying the small $\hbar$ limit of the equation
$y''(x)=Q(x)$ where $Q(x)=\hbar^{-2}Q_0(x)+\hbar^{-1}Q_1(x)+Q_0(x)$, as per
\eqref{eq:potApp}.

The $m-th$ WKB approximation, with $m\geq0$, is provided by the function
\begin{equation}\label{eq:Ym}
 Y_m(x;x')=
 \exp\left\lbrace \hbar^{-1} \sum_{k=0}^{m+1} \int_{x'}^x \hbar^k \alpha_k(s) ds \right\rbrace ,
\end{equation}
where the forms $\alpha_k(x) dx$ are recursively determined by the following requirement on the forcing term
\begin{equation}
F_m:=Q(x)- \frac{Y_m''(x)}{Y_m(x)}=O(\hbar^m).
\end{equation}
A simple computation shows that the forms $\alpha_k$ are given by equation \eqref{eq:alphaforms}, and that
\begin{equation}\label{eq:Fm}
F_0(x)=\alpha_1^2(x)+\alpha_1'(x)-Q_2 , \quad
F_m(x)= \hbar^m \left(\alpha'_{m+1}(x)+  \sum_{k=m}^{2m} \hbar^{k-m } \sum_{l=0}^{2m-k} \alpha_{m+1-l}\alpha_{k+1-m+l} \right).
\end{equation}

The following Lemma will be useful in the proof of the main result.
\begin{lemma}\label{lem:forms}
The forms $\widetilde{\alpha_1}dx$, $\alpha_k(x) dx, k\geq 2$ and
$\widehat{F}_m:=\frac{F_{m}(x)}{\sqrt{Q(x)}}dx, m\geq 0$
are holomorphic
on $X_s\setminus V$ where $V=B\cup \lbrace(q,p),(q,-p) \rbrace$.

\begin{proof}
The forms under consideration are well-defined and
meromorphic on $X_s$ since they are represented by the formula $\beta(x) dx$ with $\beta(x)=R(x,\sqrt{Q_0(x)})$ for some rational function $R$.
Because of formula \eqref{eq:alphaforms}, they are manifestly holomorphic on $X_s$ punctured at $V$ and at $\infty$.
Hence the thesis is proven if they are shown to be holomorphic at $\infty$.

Recall that a meromorphic form on $X_s$, written as $\beta(x) dx$, is regular at $\infty$ if the degree of $\beta(x)$
at $\infty$ is less or equal than $-\frac{3}{2}$ (in fact a good local parameter at $\infty$ is $\tau=x^{-\frac12}$ so that
$dx=-2 x^{\frac32} d \tau$).
Let $d_k$ denote the degree of $\alpha_k(x)$ at $\infty$. After formula \eqref{eq:alphaforms} we have
that $d_1=-1$, $\deg \widetilde{\alpha}_1=-\frac32$, $d_2=-\frac32$; moreover,
we recursively obtain $d_{2k+1}=-1-3(2k-1)$, and $d_{2k+2}=d_{2k+1}-\frac12$.
Let $\widehat{d}_m$ denote the degree of $\frac{F_{m}(x)}{\sqrt{Q(x)}}$ at $\infty$. After formula \eqref{eq:Fm}, we have that
$\widehat{d}_0=-\frac32$ and, recursively, $\hat{d}_k=d_{k+2}$. 
The thesis is proven.
\end{proof}
\end{lemma}

\subsection{WKB estimates}
Our method of analysis of the deformed cubic oscillator is based on the study of the integral equation \eqref{eq:integral},
in the case the approximate solution is the formal WKB solution $Y_m(x;x_0)$ and the forcing term is $F_m(x)$,
as defined by formulas \eqref{eq:Ym} and \eqref{eq:Fm}.

In order to cosntruct a solution of \eqref{eq:schrodingergen} using the integral equation \eqref{eq:integral},
we first need to choose an integration path $\gamma$ 
in such a way that the integral
equation admits a solution $u$ which converges uniformly to $1$, as $\hbar \to 0$,
in all sectors of the form $\hbar \in S_{\theta,\hbar_{\theta}}$ with $\theta \in [0,\frac\pi2[$.

The complex WKB method provides such a solution whenever the integration path $\gamma:[0,1] \to \mathbb{P}^1$ satisfies the following
inequality in the sector $\Re \hbar>0$
\begin{equation}\label{ineq:path}
 \left|\arg \hbar^{-1}\int_{t'}^{t} \sqrt{Q_0(\gamma(s))}\dot{\gamma}(s)ds \right|<\frac{\pi}{2}, 
\quad \forall t,t'\in]0,1[ \mbox{ such that } 0<t'<t<1 ,
\end{equation}
for one of the two choices of $\sqrt{Q_0}$.

It is straightforward to see that the only paths which satisfy inequality (\ref{ineq:path}) are the 
horizontal trajectory of $Q_0 dx^{\otimes 2}$, since these are
the steepest descent paths for the function $\Re w$, where $w(x):=\int^x \sqrt{Q_0(s)} ds$ (equivalently the arcs along which
$\Im w$ is constant). More precisely, the horizontal trajectories that serve our purposes are those horizontal trajectories that
can be prolonged indefinitely without crossing any zero of $Q_0$. These admit
a maximal extension to a simple closed Jordan curve $\gamma:=\gamma_{k,k'}:[0,1] \to \mathbb{P}^1$, that satisfies
the following 3 Properties
\begin{itemize}
 \item[(P1)] As $t \to 0$, $\gamma(t)$ is asymptotic to the ray of argument $\frac{2 \pi k}{5}$, for some $k \in \mathbb{Z}$.
 \item[(P2)] As $t \to 1$, $\gamma(t)$ 
is asymptotic to the ray of argument $\frac{2 \pi k'}{5}$, for some $k' \in \mathbb{Z}$, $k \not \equiv k' \mod 5$.
\item[(P3)]  $\frac{|\dot{\gamma}_{k,k'}(t)|}{|\gamma_{k,k'}(t)|}\to 1$
as $t\to 0$ or $t \to 1$, which implies that
$\int_0^1 | f(\gamma_{k,k'}(t))\dot{\gamma}_{k,k'}(t)| dt $ converges for every function $f$
continuous on the support of $\gamma_{k,k'}$, which decays at $\infty$ as $x^{-1-\e} $ for some $\e>0$. 
\end{itemize}
Properties (1,2) above were recalled in Section 6.1 of the main text. Property (3) follows rather directly from the
expansion $w = \frac{2}{5} x^{\frac{5}{2}}+O(x^{\frac12})$; see
 \cite[\S 7.3]{strebel} for a proper proof.
 
 We notice that the set of horizontal trajectories is naturally partitioned into subsets of trajectories with the same end points.
 Hence the following definition is quite natural.
\begin{definition}
 For every $k,k' \in \mathbb{Z}$, we denote by $\Gamma_{k,k'}$ the set of oriented horizontal trajectories leaving $\infty$
 parallel to the ray of argument $\frac{2\pi k}{5}$ and arriving at $\infty$ parallel to the ray of argument $\frac{2\pi k'}{5}$.
 Every element of $\Gamma_{k,k'}$ is endowed with a parametrisation $\gamma:[0,1] \to \mathbb{P}^1$, satisfying the properties P(1,2,3) listed
 above. We denote by the same symbol $\gamma$ an oriented trajectory and its parametrisation.
\end{definition}
\begin{remarks}\label{rem:paths}
\begin{itemize}
 \item[(i)] For every $k,k'$, the set $\Gamma_{k,k'}\cong\Gamma_{k',k}$ is either empty or diffeomorphic to the real line.
 
 The set $\mathbb{C} \setminus \cup_{k\neq k'}\lbrace x \in \gamma, \gamma \in \Gamma_{k,k'} \rbrace$ is known as the
 anti-Stokes complex. It is the union of the roots of $Q_0$ with the horizontal trajectories emanating from them.
 
 \item[(ii)]Given a point $q \in \mathbb{C}\setminus \lbrace Q_0(x)=0 \rbrace $, it belongs to at most one curve $\gamma \in \Gamma_{k,k'}$.
 
 If $q$ belongs to a curve $\gamma_q \in \Gamma_{k,k'}$, then this curve separates $\Gamma_{k,k'}\setminus \gamma_q$
 into two non-empty disjoint subsets.
 A curve belonging to one subset is homotopic in $\mathbb{P}^1 \setminus \lbrace q \rbrace$ to any other curve in the same subset,
 and not-homotopic to any curve belonging to the complementary subset.

 \item[(iii)]
 If $k'=k \pm1$, the set $\Gamma_{k,k'}$ is not empty and moreover
 $$
 \inf_{\gamma \in \Gamma_{k,k\pm1}}  \int_0^1 | f(\gamma(t))\dot{\gamma}(t)| dt =0
 $$
 for every function $f$ that is defined on a neighbourhood of $\infty$, which decays as $x^{-1-\e} $ for some $\e>0$.
 \item[(iv)]
 The condition $Q_0(x) dx^{\otimes 2}$ is saddle free can be rephrased as follows:
 there exists a $k$ such that both $\Gamma_{k,k+2}$ and $\Gamma_{k,k-2}$ are not empty.
\end{itemize}
\end{remarks}
In order to prove our main result, we need to relax inequality (\ref{ineq:path}) to allow for slightly more general integration curves.
\begin{definition}
 For every $k,k' \in \mathbb{Z}$ and any $\theta \in[0,\frac\pi2[$, we denote by $\Gamma^{\theta}_{k,k'}$
 the set of curves $\gamma:[0,1] \to \mathbb{P}^1$, satisfying the properties (P1,P2,P3) of the horizontal trajectories, and moreover
 such that there exists an $\e_{\gamma}>0$ such that
 \begin{equation}\label{eq:argument}
 \left|\arg \int_{t'}^{t} \sqrt{Q_0(\gamma(s))}\dot{\gamma}(s) ds \right|\leq\frac{\pi}{2}-\theta-\e_{\gamma}, \qquad \forall t'<t
\end{equation}
 for one of the two choices of $\sqrt{Q_0(x)}$.
\end{definition}
The great advantage of relaxing \eqref{ineq:path} to \eqref{eq:argument} is that we are able to deform the integration paths.
More precisely, we have the following Lemma.
\begin{lemma}\label{lem:deformedpath}
Suppose that $\Gamma_{k,k'}$ is not-empty. For any $\gamma \in \Gamma_{k,k'}$ such that $q \notin \gamma$, and
any $\theta \in [0,\frac\pi2[$, there exists a $\gamma_{\theta} \in \Gamma_{k,k'}^{\theta}$ satisfying the following properties:
\begin{itemize}
 \item $\gamma_{\theta}$ is homotopic to $\gamma$ in $\mathbb{P}^1\setminus \lbrace \lbrace Q_0(x)=0 \rbrace \cup \lbrace q \rbrace\rbrace$;
 \item there exist $0<t_1<t_2<1$ such that $\arg \gamma_{\theta}(t)=\frac{2\pi k}{5}$ for all $t\in]0,t_1]$
and $\arg \gamma_{\theta}(t)=\frac{2\pi k'}{5}$ for all $t\in[t_2,1[$.
\end{itemize}
\begin{proof}
 The easy proof is left to the reader.
\end{proof}
\end{lemma}
We have introduced the integration curves which we will use to define the integral equation \eqref{eq:integral} and to prove Theorem 
\ref{thm:appendix}. Before dealing
with the analysis of \eqref{eq:integral}, we need a last preparatory lemma.
\begin{lemma}\label{lem:bounded}
 Let $\beta(x)dx$ be one of the forms considered in Lemma
 \ref{lem:forms}: namely $\beta$ is either $\widetilde{\alpha}_1$, or
 $\alpha_k$ $k\geq 2$, or $\widehat{F}_k$, $k\geq 0$. Then
 \begin{equation*}
  \int_0^1 |\beta(\gamma(t)) \dot{\beta}(t)| dt < \infty, \quad \forall \gamma \in \Gamma^{\theta}_{k,k'} \mbox{ such that } q \notin \gamma.
 \end{equation*}
\begin{proof}
 It follows from Lemma \ref{lem:forms} and Property (3) of the paths $\Gamma^{\theta}_{k,k'}$. 
\end{proof}
\end{lemma}

We now prove the fundamental estimate underlying the complex WKB method.
\begin{prop}\label{prop:basicwkb}
Fix a $\theta \in [0,\frac{\pi}{2}[$, a $\gamma \in \Gamma^{\theta}_{k,k'}$ such that $q \notin \gamma$, a $t_0 \in ]0,1[$
and the branch of $\sqrt{Q_0(x)}$ in such a way that $\lim_{t\to 0^⁺}\Re \int_{t}^{t_0} \sqrt{Q_0(\gamma(t))}\dot{\gamma}(t) dt =\infty$.

 For any $\hbar_{\theta}>0$, there is a sequence of positive constants
 $C_{m}, m\geq 0$ - depending on $\theta, \gamma$ -
 and a unique sequence of solutions $y_{k,m}(x)$ of the deformed cubic oscillator satisfying the following inequality
 \begin{equation}\label{eq:WKBestimate}
 \sup_{t \in[0,1]} \left| \frac{y_{k,m}(\gamma(t))}{Y_m(\gamma(t),\gamma(t_0))}-1 \right| \leq C_{m} |\hbar|^{m+1} ,
 \quad \forall \hbar\in S_{\theta,\hbar_\theta}
 \end{equation}
 where $Y_m(x;\gamma(t_0))$ is the formal WKB solution defined by formula (\ref{eq:Ym}) with $x=\gamma(t),x'=\gamma(t_0)$.

 Moreover, the solutions $y_{k,m}$ satisfy the following properties
 \begin{enumerate}
  \item $\lim_{|x| \to \infty}y_{k,m}(e^{i\frac{2\pi k}{5}}|x|)=0$. Equivalently, $y_{k,m}(x)$
  is subdominant in the k-th Stokes sector.
  \item $y_{k,m}(x)=D_m \, y_{k,0}(x)$, with $D_m=\exp \big(- \sum_{k=2}^{m+1}\hbar^k \int_{0}^{t_0}
 \alpha_k(\gamma(t)) \dot{\gamma(t)} dt \big)$.
  \item  $\lim_{|x| \to \infty}|y_{k,m}(e^{i\frac{2\pi k'}{5}}|x|)|=\infty$.
  Equivalently, $y_{k,m}(x)$ is dominant in the k'-th Stokes Sector.
 \end{enumerate}
 
 \begin{proof}
 We introduce an order relation on $\gamma$: $v\leq x$ if $v=\gamma(s), x=\gamma(t)$ and $s\leq t$.
 We use the following  convention: $\int_{\gamma(s),\gamma}^{\gamma(t)} f(v) dv:= \int_{s}^t f(\gamma(t')) \dot{\gamma}(t') dt' $, and
 $\int_{\gamma(s),\gamma}^{\gamma(t)}| f(v) dv|:= \int_{s}^t |f(\gamma(t')) \dot{\gamma}(t')| dt' .$
 
 According to the general theory we provide the solution $y_{k,m}$ (hence we prove its existence)
  by analysing the integral equation \eqref{eq:integral} for the ratio $u(x):=\frac{y_{k,m}(x)}{Y_m(x)}$.
   For convenience we rewrite the integral equation in the following form
  \begin{equation}\label{eq:integralproof}
   u(x)=1- \hbar \int_{\infty,\gamma}^x\widehat{K}_m(x,v)  \widehat{F}_m(v) u(v)  dv ,
  \end{equation}
  where $\widehat{K}_m(x,v)=\hbar^{-1}\sqrt{Q_0(v)} \int_{v,\gamma}^x \frac{ Y_m^2(v)}{Y_m^{2}(r)} dr$, and
  $\widehat{F}_m(v)=\frac{F_m(v)}{\sqrt{Q_0(v)}}$
  with $F_m(x)$ as defined in \eqref{eq:Fm}.
  
  We divide the analysis of the integral equation \eqref{eq:integralproof} in two steps
  \begin{enumerate}
   \item  We show the estimate: For any given $\theta<\frac\pi2$, if $|\hbar|$ is smaller than an arbitrary, but fixed,
   constant $\hbar_{\theta} >0$, there exists a $C_m>0$ such that
  $|\widehat{K}_m(x,v)|\leq C_m $ for all $x,v \in \gamma, x\leq v$.
  \item We use the above estimate on $|\widehat{K}_m(x,v)|$ to study the integral equation and prove the thesis.
  \end{enumerate}

  Step 1. 
  In order to estimate  $\widehat{K}_m(x,v)$ we need to control the integral $\int_{v,\gamma}^x \frac{ Y_m^2(v)}{Y_m^{2}(r)} dr$,
  where $Y_m$ is the formal WKB solution. If $m\geq 1$, the integral $\int_{v,\gamma}^x \frac{ Y_m^2(v)}{Y_m^{2}(r)} dr$
  cannot be computed in close form. To overcome this difficulty we factorise 
  $Y_m(x)$ as $Y(x)T_m(x)$, where $Y$ is an unbounded function such that $\int Y^{-2}(r) dr$ can be computed in closed form,
  and $T_m$ is a bounded function (with bounded derivatives).
  
  Explicitly, we make the following choice
  $$Y(x)= \exp{\big( \int_{x',\gamma}^x  \hbar^{-1} \sqrt{Q_0(w)} -\frac{Q_0'(w)}{4Q_0(w)}dw \big)} , \quad 
  T_m(x)=\exp{\big( \int_{x',\gamma}^x  \widetilde{\alpha}_1(w) + \sum_{l=1}^{m+1}\hbar^l \alpha_l(w) dw \big)},$$
  where $x'=\gamma(t_0)$, and the forms $\alpha$ are as in \eqref{eq:alphaforms}.
  
  We notice that $Y^{-2}(r)= \frac{\hbar}{2}\frac{d}{dr} e^{-\frac{2}{\hbar}\int_{x'}^r \sqrt{Q_0(w)} dw}$ and we integrate by parts to obtain
  $$
  \widehat{K}_m(x,v)= \frac12 \left( e^{-\frac{2}{\hbar}\int_v^x \sqrt{Q_0(w)} dw}\frac{T^{2}_m(v)}{T^2_m(x)}-1 \right) -
  \frac12 \int_v^x e^{-\frac{2}{\hbar}\int_v^r Q(w) dw} \frac{d}{d r} \frac{T_m^2(v)}{T_m^{2}(r)} dr
  $$
  Due to Lemma \ref{lem:forms}, the functions $T_m(x), T_m^{-1}(x)$ as well as all their derivatives are uniformly bounded on $\gamma$,
  provided $|\hbar|$ is bounded.
  It follows that
  \begin{equation}\label{eq:middleestimate}
   |\widehat{K}_m(x,v)|\leq C_{m,1} \left(1+  |e^{-\frac{2}{\hbar}\int_v^x \sqrt{Q_0(r)} dr}| +
   \int_{v,\gamma}^x |e^{-\frac{2}{\hbar}\int_{v,\gamma}^r \sqrt{ Q_0(r')} dr'} dr| \right),
  \end{equation}
   where $C_{m,1}$ is a sufficiently high positive constant (in the third term we have used the H\"older inequality).
   
   By definition of $\Gamma_{k,k'}^{\theta}$ there exists an $\e_{\gamma}>0$ such that
\begin{equation}\label{eq:argproof}
 |\arg \int_{v}^{x} \sqrt{Q_0(\gamma(s)})ds |\leq\frac{\pi}{2}-\theta-\e_{\gamma}, \qquad \forall v,x \in \gamma(]0,1[) \mbox{ such that }
 v\leq x .
\end{equation}
From the above inequality it follows directly that $|e^{-\frac{2}{\hbar}\int_v^x \sqrt{Q(r)} dr}|\leq 1$ for all $v,x \in \gamma(]0,1[)$
such that $v\leq x$.
   
   We complete Step 1 by showing that the inequality \eqref{eq:argproof} implies that also the third term in
   \eqref{eq:middleestimate} is uniformly bounded by a constant $C_2$.
   To be more precise we show that there exists a $C_{2}<\infty$ such that
   \begin{equation}\label{eq:Exv}
    E(x,v):= \int_{v}^x |e^{-\frac{2}{\hbar}\int_{v}^{r} \sqrt{Q(r')} dr'}  dr| \leq C_{2}  \mbox{ if }
   v\leq x.
   \end{equation}
   We notice that, due to \eqref{eq:argproof}, under the function $w(x)=\int_{x'}^x \sqrt{Q_0(r)}dr$, the curve
   $\gamma$ is mapped onto a curve which is diffeomorphic
   to its projection on to the real axis. We call $g$ such a curve and we parametrise it by its real part; explicitly,
with $x=\Re w(\gamma(t))$, $\Re g(x)=x$, $\Im g(x)=\Im w(\gamma(t))$. Using $x=\Re (w(\gamma(t))$ as the new variable of integration
in \eqref{eq:Exv}, we transform the problem of bounding $E(x,v)$ into the equivalent problem: prove that there exists a $C_2'>0$ such that
\begin{equation}\label{eq:Exx'}
\widetilde{E}(x,y):=
\int_{y}^x \big|e^{-\frac{2}{\hbar}\big( g(x')-g(y)\big)}
\frac{\frac{d g(x')}{d x'}}{\sqrt{Q_0\big(\Phi(x')\big)}}\big|  dx' \leq C_2' , \qquad \forall y\leq x \in \mathbb{R} ,
\end{equation}
where $\Phi$ is the inverse of $w$ composed with $g$.
The functions $\big|Q_0^{-\frac12}\big(\Phi(x)\big)\big|$ and $|\frac{d g(x}{d x}|$ are bounded.
Indeed,
$Q_0^{-\frac12}(\Phi(\Re(x)))$ decays as $|x| \to \infty$ (one can show as $O(|x|^{-\frac{2}{5}})$); moreover
$|\frac{d g(x)}{dx}|$ converges
to $1$ as $|x| \to \infty$ by definition of $\Gamma_{k,k'}^{\theta}$ (Property (P3) of the steepest descent paths).
Hence if we show that $\int_{y}^x \big|e^{-\frac{2}{\hbar}\big( g(x')-g(y)\big)}\big|dx'$
is smaller than a constant $C_2''$ for all $y\leq x$, \eqref{eq:Exx'} follows by the H\"older inequality. To this aim,
we notice that \eqref{eq:argproof} implies that
$\big|e^{-\frac{2}{\hbar}\big( g(x')-g(y)\big)}\big|\leq e^{-\frac{2 \cos{(\frac{\pi}{2}-\e_{\gamma})}}{|\hbar|}(x'-y)}$ for all $y\leq x'$; integrating
the right hand side we obtain that
$$\int_y^x \big|e^{-\frac{2}{\hbar}\big( g(x')-g(y)\big)}\big| dx \leq \frac{|\hbar_{\theta}|}{2 \sin(\e_{\gamma}) }, \qquad \forall y\leq x ,$$
which completes the proof of Step 1.
%
%
  
 Step 2. We denote by $\mathcal{C}_{\gamma}$ the space of continuous functions supported on $\gamma$
 endowed with the supremum norm $\|f\|_{\infty}$.
 In this space we define the linear operator $K_m$ by the formula
 $$K_m[f](x)=-\hbar \int_{\infty,\gamma}^x \widehat{K}_m(x,v) \widehat{F}_m(v) f(v) dv$$
 which allows us to
write the integral equation \eqref{eq:integralproof} in the compact form $u=1+ K_m[u]$.
As we will show, this integral equation admits a (unique) continuous solution, which is of the form
$u=\sum_{N=0}^{\infty}K_m^N [1] $, where $1$ is the constant function $1$ on $\gamma$ and $K_m^N$ is
the $N$-th iterate of $K$.

After Step 1., for every $0<\theta<\frac\pi2$ there exist a sequence of positive constants $\hbar_{\theta},\widehat{C}_m,m\geq 0$ such that
$|\widehat{K}_m(x,s)|\leq \widehat{C}_m<\infty$ for all $\hbar \in S_{\theta,\hbar_{\theta}}$. Furthermore,
due to Lemma \ref{lem:bounded}, there exists another sequence of positive constants $\rho_m,m\geq0$ such that
$\int_{0}^1 |\widehat{F}_m(v) d v |  \leq |\hbar|^{m} \rho_{\gamma}$ for all $\hbar  \in S_{\theta,\hbar_{\theta}}$.
Using once again the H\"older inequality together with the two estimates above, we immediately obtain that
the operator $K_m$ is bounded as indeed its operator norm
$\|K_m\|$ is less or equal than $|\hbar|^{m+1}\widehat{C_m} \rho_{m}$. Namely
$$
\|K_m [f]\|_{\infty}\leq |\hbar|^{m+1} \widehat{C_m} \rho_{m}\|f\|_{\infty} \mbox{ for every bounded function } f .
$$
It is a basic fact of integral equations of Volterra type that $\|K_m^N\|= \frac{\|K_m\|^N}{N!}$ (in fact $K_m^N [f]$ is defined as
an integral on the $N$ dimensional simplex whose volume is $\frac{1}{N!}$; for a detailed proof see e.g. \cite[\S 79]{copson}).
It follows that the series $u=\sum_{N=0}^{\infty}K_m^N[1]$ converges in $\mathcal{C}_{\gamma}$,
and that $\|u-1\|_{\infty}\leq e^{|\hbar|^{m+1}\widehat{C_m} \rho_{m}}-1 $, for every $\hbar \in S_{\theta,\hbar_{\theta}}$.
Therefore the function $y_{k,m}(x):=u(x) Y_m(x)$ is a solution of the deformed cubic which satisfies the estimate
\eqref{eq:WKBestimate}.

We complete the proof by proving properties (1,2,3) of $y_{k,m}$, as well as its uniqueness.

(1) By construction $\lim_{t \to 0}Y_m(\gamma(t);\gamma(t_0))=0$ and $\gamma$ is asymptotic to the ray of argument
$e^{i\frac{2\pi k}{5}}$ for $t\to 0$. It follows that $y_{k,m}$ is subdominant in the $k-th$ Stokes sector.

It is well-known (see e.g. \cite{piwkb}) that in any given Stokes sector, the subdominant solution is uniquely defined up to a scale.
Hence property (1) implies hat $y_{k,m}$ is the unique solution satisfying \eqref{eq:WKBestimate}.

(2)For the same reason $y_{k,m}=D_m y_{k,0}$ for some $D_m \in \mathbb{C}^*$. 
Since $$\lim_{t \to 0}\frac{Y_m(\gamma(t);\gamma(t_0))}{Y_{0}(\gamma(t);\gamma(t_0))}=e^{-\sum_{k=2}^{m+1}\hbar^k \int_0^{t_0}\alpha_k(\gamma(t))
\dot{\gamma}(t) dt},$$ the thesis follows.

(3)The thesis follows from the fact that $\lim_{t \to 1}|Y_m(\gamma(t);\gamma(t_0))|=\infty$.

 \end{proof}

\end{prop}
\begin{remark}\label{rem:lowerpoint}
One of the hypothesis of the proposition above is that the lower integration point $x'$ in the definition of
$Y_m(x;x')$ belongs to the curve $\gamma$. However this condition can be dropped. Choose any other point $x''$ in the complex plane,
which is not a root of $Q_0(x)$, and a path $\gamma'$ connecting $x''$ to $x'$. The function
$\hat{y}_{k,m}=e^{\hbar^{-1}\sum_{k=0}^{m+1}\int_{x''}^{x'} \alpha_k(x) dx} y_{k,m}(x)$ is a new solution of the deformed cubic oscillator
and $ \frac{\hat{y}_{k,m}(x)}{Y_m(\gamma(t);x'')}$
satisfies the estimate \eqref{eq:WKBestimate}, since $ \frac{\hat{y}_{k,m}(x)}{Y_m(\gamma(t);x'')}=
 \frac{y_{k,m}(x)}{Y_m(x;x')}$ by construction.
\end{remark}

We have the following Corollary.
\begin{cor}\label{cor:residueatq}
Let $(q,p), p \neq 0$ be the point of $X_s$ used to define the potentials $Q_1,Q_2$. We have that
\begin{equation}
 \mbox{res}_{(q,\pm p)} \alpha_k(x) dx =0, \; \forall k \geq 2  .
\end{equation}
\begin{proof}
Instead of considering the forms $\alpha_k dx$ as meromorphic differentials on $X_s$, we can consider them
as multi-valued meromorphic differentials on $\mathbb{C}$. The thesis is then equivalent to
$\mbox{res}_{x=q} \alpha_k(x) dx =0,\forall k \geq 2 $ for both branches
of $\sqrt{Q_0(x)}$. We prove this statement here.

We fix a branch of $\sqrt{Q_0}$. We suppose -
without loss of generality\footnote{If $q$ is not in generic position, we can consider, instead of $X_s$  punctured at $(q,\pm p)$,
the isomorphic punctured curve, $X_{\e},\e \in \mathbb{C}$ punctured at $(q_{\e}=e^{i \e} q,\pm e^{i \frac32 \e}p)$, where
$X_{\e}$ is the (projectivization) of the affine elliptic curve obtained by twisting the coefficients $a,b$ of $Q_0$ as
$a \to e^{2i\e}a,b \to e^{3i\e}b $. If $\e \neq0$ is small then
$q_{\e}$ is in generic position. The same $\mathbb{C}^*$ action is discussed in the main text in Remark \ref{actions}(b).} -
that, there exists a pair $(k,k')$ and a $\gamma_q \in \Gamma_{k,k'}$ such that $q \in  \gamma_q$, and
$\lim_{t\to 0^⁺}\Re \int_{t}^{t_0} \sqrt{Q_0(\gamma_q(t))}\dot{\gamma_q}(t) dt =\infty$.


We can then choose two paths $\gamma,\gamma'$ in $\Gamma_{k,k'}$,
which are not homotopic in $\mathbb{C}\setminus \lbrace q \rbrace$; see Remark \ref{rem:paths}(ii).

We fix a $\theta \in[0,\frac\pi2[$.
According to Lemma \ref{lem:deformedpath}, $\gamma,\gamma'$ can be deformed to two paths
$\gamma_{\theta},\gamma'_{\theta} \in \Gamma_{k,k'}^{\theta}$ such that
$\gamma_{\theta}(t)=\gamma'_{\theta}(t)$ as $t\to 0$, and as $t\to 1$.
These are by construction non-homotopic paths in $\mathbb{P}^1 \setminus \lbrace q \rbrace$. Since $\gamma_{\theta},\gamma'_{\theta}$
coincide for large vaue of $|x|$, $\gamma_{\theta}-\gamma'_{\theta}$ defines a non-trivial closed loop
in $\mathbb{C}\setminus \lbrace q \rbrace$.
Let $t_0$ small enough so that $\gamma_{\theta}(t_0)=\gamma'_{\theta}(t_0)$. After Proposition \ref{prop:basicwkb}
it follows that there are positive constants $\hbar_{\theta}$ and $C_m, m \geq 0$ such that for all $\hbar \in S_{\theta,\hbar_{\theta}}$
 \begin{eqnarray}\label{eq:estimatecor}
  \sup_{t \in[0,1]} & \left| \frac{y_{k,m}(\gamma_{\theta}(t))}{Y_m(\gamma_{\theta}(t);\gamma_{\theta}(t_0))}-1 \right| \leq C_{m} |\hbar|^{m+1} , \quad
  \sup_{t \in[0,1]} & \left| \frac{y_{k,m}(\gamma_{\theta}'(t))}{Y_m(\gamma'_{\theta}(t);\gamma'_{\theta}(t_0))}-1 \right| \leq C_{m} |\hbar|^{m+1} .
 \end{eqnarray}
Here $Y_m$ is the m-th WKB approximation defined in \eqref{eq:Ym}.
Recall, from the main text, that every non-trivial solution of the deformed cubic is two valued and
the point $x=q$ is its branch point. By construction $\gamma_{\theta},\gamma'_{\theta}$ are not homotopic in $\mathbb{C}\setminus \lbrace q \rbrace$, hence
$\frac{y_{k,m}(\gamma_{\theta}(t))}{y_{k,m}(\gamma'_{\theta}(t))}=-1$ as $t \to 1$. Moreover we have that
$$
\frac{Y_{m}(\gamma_{\theta}(t);\gamma_{\theta}(t_0))}{Y_{m}(\gamma'_{\theta}(t);\gamma'_{\theta}(t_0))}=
- \exp{\left(2\pi i \sum_{k=2}^{m+1} \hbar^{k-1} \mbox{res}_q\alpha_{k}(x) dx\right)} , \; \mbox{ as }
t \to 1,
$$
since $e^{ 2\pi i\oint_{\gamma_{\theta}-\gamma'_{\theta}} \hbar \alpha_0(x)+ \alpha_{1}(x) dx}=-1$, by explicit computation.
Because of the above identities, the inequalities \eqref{eq:estimatecor} imply that for all $m\geq0$ there exists a $\widetilde{C}_m$ such that
$$
\left| \exp\big(2 \pi i\sum_{k=2}^{m+1} \hbar^{k-1} \mbox{res}_q\alpha_{k}(x) dx\big)-1 \right|\leq \widetilde{C}_m  |\hbar|^{m+1} , \mbox{ as } \hbar \to 0. 
$$
It immediately follows that $\mbox{res}_{x=q} \alpha_k(x) dx =0,  \forall k \geq 2 $, for the chosen branch of
$\sqrt{Q_0}$.
\end{proof}
\end{cor}

 \subsection{Lifting WKB solutions to $X_s$}\label{sec:Xs}
 By hypothesis the potential $Q_0(x)$ is saddle free, from which it follows that there exists a $k$ such that $\Gamma_{k,k\pm2}$ is not empty,
 see Remark \ref{rem:paths} (iv).
Without losing in generality, we suppose that $\Gamma_{0,\pm2} \neq \emptyset $ (the other cases are obtained by a rotation).
Hence we are in the situation depicted in Figure \ref{fig:case0}, and we can fix the roots $x_0,x_1,x_{-1}$
of $Q_0(x)$, as depicted in the same Figure.
\begin{figure}[ht]
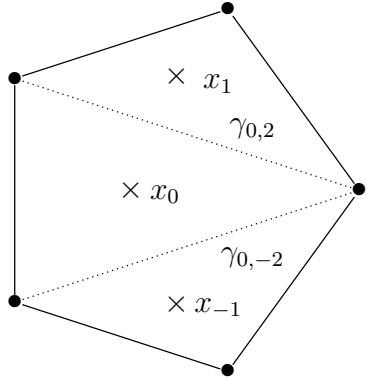

\begin{center}
  \begin{tabular}{c}
\xy /l1.2pc/:
(-5,0)*{\bullet}="1";
(-1.55,-4.76)*{\bullet}="2";
(4.05,-2.94)*{\bullet}="3";
(4.05,2.94)*{\bullet}="4";
(-1.55,4.76)*{\bullet}="5";
(1,0)*{\times}="6";
(-.2,-3)*{\times}="7";
(-.2,3)*{\times}="8";
{"1"\PATH~={**@{-}}'"2"},
{"2"\PATH~={**@{-}}'"3"},
{"3"\PATH~={**@{-}}'"4"},
{"4"\PATH~={**@{-}}'"5"},
{"5"\PATH~={**@{-}}'"1"},
(0.1,.1)*{x_0};
(-1.3,-2.8)*{x_{1}};
(-2.2,-1.6)*{\gamma_{0,2}};
(-1.3,3.1)*{x_{-1}};
(-2.2,1.8)*{\gamma_{0,-2}};
{"1"\PATH~={**@{.}}'"3"},
{"1"\PATH~={**@{.}}'"4"},
\endxy
\end{tabular}

\end{center}

\caption{Schematic representation of the trajectories $\gamma_{0,\pm2}$ and of the roots of $Q_0(x)$ \label{fig:case0}}

\end{figure}

We choose $3$ branch-cuts of the function $\sqrt{Q_0(x)}$: the $j-th$ cut, $j=-1,0,1$ connects
the roots $x_j$ with the point at $\infty$ and
and it asymptotic to the ray $\pi + j \frac{2\pi }{5}$, see Figure \ref{fig:doublesheeted} below.
The elliptic curve $X_s$ is thus realised the Riemann surface of the function $\sqrt{Q_0(x)}$, and we name 
the lower sheet the one fixed by the requirement $\lim_{x \to +\infty }\Re \sqrt{Q_0(x)}=+\infty$.
To represent a curve in $X_s$ as a curve in the two-sheeted covering, we draw a solid line when the curve belong to the upper sheet, and a
dashed line otherwise. 
\begin{figure}[htbp]
\begin{center}
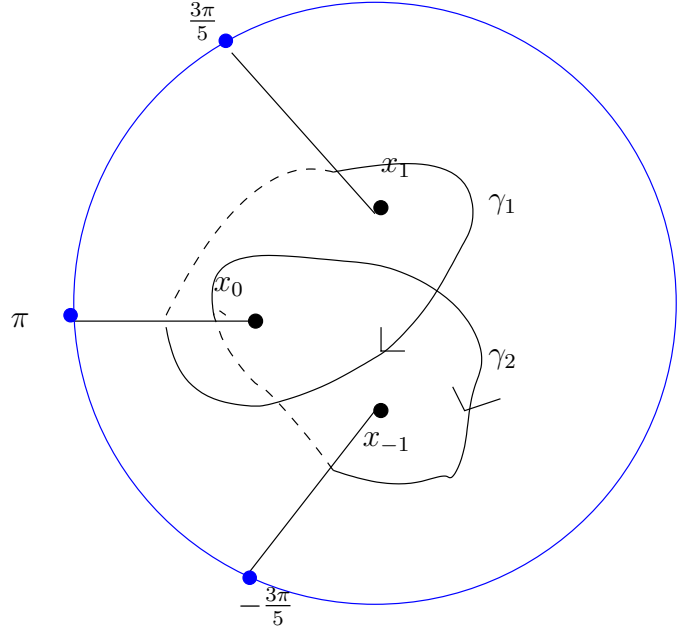
\end{center}
\caption{The elliptic curve $X_s$ as a double-sheeted cover, with the cycles $\gamma_{1,2}$ defined by the WKB triangulation}
\label{fig:doublesheeted}
\end{figure}

After Proposition \ref{prop:basicwkb}, the subdominant solutions $y_k, k \in \mathbb{Z}/5\mathbb{Z}$ are
well-approximated on paths $\gamma \in \Gamma^{\theta}_{k,k'}$ by the m-th WKB approximation
\eqref{eq:Ym}, namely
$$
Y_{m}(x;x')=
 \exp\left\lbrace \hbar^{-1} \sum_{k=0}^{m+1} \hbar^k \int_{x',\gamma}^x  \alpha_k(s) ds \right\rbrace .
$$
provided the branch of $\sqrt{Q_0}$ is chosen in such a way that
\begin{equation}\label{eq:sheetcondition}
 \lim_{t\to 0^⁺}\Re \int_{t}^{t_0} \sqrt{Q_0(\gamma(t))}\dot{\gamma}(t) dt =\infty \; .
\end{equation}
Since the formal WKB solutions are written in terms of abelian differential on $X_s$, they are naturally defined
on the lift of $\gamma$ to $X_s$, which we call $\widehat{\gamma}$.
This is not only natural, but also very convenient since there is a unique way of lifting $\gamma$ that enforces
condition \eqref{eq:sheetcondition}. In fact,
taking into consideration our choice of the branch-cuts of $\sqrt{Q_0}$, the lift
of any path $\gamma$ belonging to $\Gamma^{\theta}_{k,k'}$ is defined as follows:
\begin{itemize}
 \item If $k=0$, $\widehat{\gamma}$ lies on the upper sheet
 for $t$ small. In fact, by definition, $\lim_{x \to +\infty} \Re \sqrt{Q_0(x)}=-\infty$, if $x$ belongs to the upper sheet.
 \item If $k\neq0$, $\widehat{\gamma}$ lies on the lower sheet
 for $t$ small.
\end{itemize}

We finish this Section by analysing the cycles $\gamma_{1,2} \in H_1(X_s^\circ \setminus B,\mathbb{Z})^-$ defined in 
Figure \ref{fig:doublesheeted}. Their image in $H_1(X_s,\mathbb{Z})$ coincide with the cycles $\gamma_{1,2}$
provided by the WKB triangulation, as defined in Section 6 of the Main Text. This indeed is equivalent to the
point (i) of the following Lemma.
\begin{lemma}\label{lem:normalization}
The paths $\gamma_1,\gamma_2 \in H_1(X_s^\circ \setminus B,\mathbb{Z})^-$ defined in Figure \ref{fig:doublesheeted} satisfy the following
normalisation
 \begin{itemize}
 \item[(i)] $\int_{\gamma_1} \sqrt{Q_0(x)} dx  =2\int_{x_{1}}^{x_0} \sqrt{Q_0(x)} dx$ and
 $\int_{\gamma_2} \sqrt{Q_0(x)} dx  =2\int_{x_{0}}^{x_{-1}} \sqrt{Q_0(x)} dx$ where the right hand side is 
 computed in the upper sheet. 
 \item[(ii)] $\Im \int_{\gamma_i}\sqrt{Q_0(x)}dx>0,\, i=1,2$.
 \item[(iii)] $[\gamma_1,\gamma_2]=1$
\end{itemize}
\begin{proof} 
(i) and (iii) are self-evident.
(ii) We prove $\Im \int_{\gamma_1}\sqrt{Q_0(x)}dx>0$, and leave the other case to the reader.
Recall the following facts from
 Section 6.2 of the Main Text:
 \begin{itemize}
  \item $x_1$ and $x_0$ belong to the closure of the simply connecetd domain -- we denote by $H$ -- which is
  foliated by the horizontal trajectories belonging to $\Gamma_{0,2}$. 
  \item The map $x \mapsto \int_{x_0}^x\sqrt{Q_0(u)}du$ is a conformal map of $H$ into a a horizontal strip.
  \item There is a path $l$ connecting $x_1$ with $x_0$ such that the angle between any $\gamma \in \Gamma_{0,2}$ and $l$ is a fixed, positive
  number $\pi\theta, \theta \in ]0,1[$.
 \end{itemize}
 If $x$ belongs to the upper-sheet then $\Re \int_{x_0}^{x}\sqrt{Q_0(u)}du$ increases along $\gamma$ for any $\gamma \in 
 \Gamma_{0,2}$.
 Since  $ \int_{x_0}^x\sqrt{Q_0(u)}du$ is conformal, it follows that $\Im\int \sqrt{Q_0(x)}dx $ increases along the line $l$
 connecting $x_1$ with $x_0$. The thesis follows.
\end{proof}

\end{lemma}

\subsection{Proof of the Theorem \ref{thm:appendix}} \label{sec:proof}
The proof of the Theorem is based on the Proposition \ref{prop:basicwkb} and on the computation of the WKB approximation
of cross-ratios of asymptotic values, that the author developed in \cite{piwkb,myphd}.

As it was explained in Section \ref{sec:wkb} above, the hypothesis that the potential $Q_0(x)$ is saddle free
is equivalent to the property that there exists a $k$ such that $\Gamma_{k,k\pm2}$ is not empty.
Moreover, we can always reduce to the case that
$\Gamma_{0,\pm2} \neq \emptyset $, hence we are in the situation depicted in Figure
\ref{fig:case0} above.

For an arbitrary basis $\lbrace y, \widetilde{y} \rbrace$
of solutions to the deformed cubic oscillators, one defines the single-valued meromorphic function
$f(x)=\frac{y(x)}{\widetilde{y}(x)}$. The function $f$ has $5$ asymptotic values, $a_k, k \in \mathbb{Z}/5\mathbb{Z}$,
defined by the formula
\begin{equation}\label{eq:ak}
 a_k(\hbar)=\lim_{x\to + \infty }f(|x|e^{i \frac{2 \pi k}5}) \in \mathbb{P}^1 ,
\end{equation}
which is independent on the curves along which the limit is taken.

According to the main text, see equation \eqref{cr}, the Fock-Goncharov co-ordinates are defined as
cross-ratio of the asymptotic values
\begin{eqnarray}\label{eq:Tomformula}
 X_1(\hbar):=\CR(a_0,a_1,a_2,a_{-2})
 \quad X_2(\hbar):=\CR(a_0,a_2,a_{-2}, a_{-1})
 .
\end{eqnarray}
Here $\CR(a,b,c,d)=\frac{(a-b)(c-d)}{(a-d)(b-c)}$, is the cross-ratio.

In what follows we prove the thesis, namely equation \eqref{eq:thesis}, for the co-ordinate $X_2$.
The proof for the co-ordinate $X_1$ can obtained by repeating the very same steps, and it is therefore omitted.

%

\subsubsection{The apparent singularity}
Here we prove a generalization of formula (\ref{eq:ak}), which is useful in the presence of one apparent singularity.

 We fix a point $x' \in \mathbb{C}$ such that $x' \neq q $ and two local linearly independent solutions Moreover we fix two solutions
$y,\widetilde{y}$. Suppose that we have two paths,
$\gamma,\widetilde{\gamma}$, that connects $x'$ to $e^{i\frac{2\pi k}5}\infty$, that do not cross $x=q$, and that
coincide for large $x$, so that $\gamma-\widetilde{\gamma}$ can be thought as a Jordan curve (i.e. a simple closed curve) on $\mathbb{C}\setminus \lbrace q \rbrace$.
Denoting by $y_{\gamma}(x),\widetilde{y}_{\widetilde{\gamma}}(x)$ the analytic continuation of $y,\widetilde{y}$ along these paths,
we obtain the following expression for the asymptotic value $a_k$, which we will need below
\begin{equation}\label{eq:akratio}
 a_k=(-1)^{s(\gamma-\widetilde{\gamma})}\lim_{x\to + \infty} \frac{y_{\gamma}(xe^{i \frac{2 k \pi}5})}
 {\widetilde{y}_{\widetilde{\gamma}}(xe^{i \frac{2 \pi k}5})} .
\end{equation}
Here $s$ is the winding number of $\gamma-\widetilde{\gamma}$ around $q$.

The above formula is a consequence of \eqref{eq:ak} and the following fact:
for every non trivial solution, the point $x=q$ is a
branch point and the monodromy about $q$ is $-1$.

\subsubsection{The Fock-Goncharov co-ordinates in the small $\hbar$ limit}
In order to compute the asymptotic expansion of $X_2(\hbar)$ we need to choose a basis of solutions
with a known asymptotic expansion, and then compute the corresponding asymptotic values $a_k$ .
Our choice (the only possible) is
$\lbrace y_0,y_{-2}\rbrace$ where $y_0$ is the solution subdominant at $+\infty$ and $y_{-2}$ is the solution subdominant
at $e^{-\frac{4\pi}{5}i} \infty$.

Notice that $\lbrace y_0,y_{-2}\rbrace$ may in general fail to form a basis of solutions. They do however form a basis, whenever
$\hbar$ is small enough. Indeed, let us
fix a $\theta \in [0,\frac\pi2[$. By hypothesis $\Gamma_{0,-2}$ is not empty. Therefore, according
to Proposition \ref{prop:basicwkb}(3), there exists a $\hbar_{\theta}>0$
such that $\lim_{x \to + \infty }|y_0(x^{-\frac{4\pi}{5}}i)|=\infty$, for all $\hbar \in
S_{\theta,\hbar_\theta}$. This implies that
the solution $y_0$ and the solution $y_{-2}$ are linearly independent.
Hence $a_0=0$, $a_{-2}=\infty$, and formula \eqref{eq:Tomformula} reduces to
\begin{eqnarray} \label{eq:fgformula}
 X_2(\hbar)=-\frac{a_{2}(\hbar)}{a_{-1}(\hbar)} , \quad \forall \hbar \in S_{\theta,\hbar_\theta} .
\end{eqnarray}
%

\subsubsection{Integration paths used in the proof}
By hypothesis on the potential $Q_0$, the sets $\Gamma_{0,-2}$,$\Gamma_{0,2}$ are not empty, and the sets
$\Gamma_{-2,-1}$,$\Gamma_{-2,2}$, and $\Gamma_{0,-1}$ are not empty for every potential $Q_0$, see
Remark \ref{rem:paths}.

According to Lemma \ref{lem:deformedpath},
for every $\theta \in [0,\frac\pi2[$, we can choose paths
$\gamma_{0,\pm2} \in \Gamma^{\theta}_{0,\pm2},\gamma_{0,-1} \in \Gamma^{\theta}_{0,1},
\gamma_{-2,0} \in \Gamma^{\theta}_{-2,0},\gamma_{-2,2}\in \Gamma^{\theta}_{-2,2} $
satisfying the following properties
\begin{enumerate}
 \item $\gamma_{0,2}(t)=\gamma_{0,-2}(t)=\gamma_{0,-1}(t)$ for $t \in[0,t_0]$, with $t_0>0$. We denote by $x'=\gamma_{0,2}(t_0)$
 the (last) intersection point;
 \item $\gamma_{-2,0}(t)=\gamma_{0,-2}(1-t)$;
 \item $\gamma_{-2,2}(t)=\gamma_{-2,0}(t)=\gamma_{-2,-1}(t)$ for $t  \in[0,t_1]$, with $t_1$ small enough.
 We denote by $x''=\gamma_{-2,0}(t_1)$
 the (last) intersection point
 \item $\gamma_{-2,1}(t)=\gamma_{0,-1}(t)$ as $t \to 1$;
\end{enumerate}

After Proposition \ref{prop:basicwkb}, a subdominant solution in the $0$-th Sector, $y_0(x)$, is well-approximated by the
m-th WKB function
\begin{equation}\label{eq:Ymproof}
 Y^{(0)}_{m}(x;x')=
 \exp\left\lbrace \hbar^{-1} \sum_{k=0}^{m+1} \int_{x',\gamma}^x \hbar^k \alpha_k(s) ds \right\rbrace , \quad \forall x \in
 \gamma_{0,2} \cup \gamma_{0,-2} \cup \gamma_{0,-1},
\end{equation}
Here the integration path $\gamma$ is -depending on $x$ -  $\gamma_{0,2}$ or $\gamma_{0,-2}$ or $\gamma_{0,-1}$,
and the suffix $(0)$ stands to remind that the branch of $\sqrt{Q_0}$
is chosen in such a way that  $\lim_{t\to 0^⁺}\Re \int_{t}^{t_0} \sqrt{Q_0(\gamma_{0,2}(t))}\dot{\gamma}_{0,2}(t) dt =\infty$.
More precisely: there is a $\hbar_{\theta}>0$ and
a sequence of positive constants $C_{m,\theta}$ such that
\begin{equation}\label{eq:y0proof}
 \left|\frac{y_0(x)}{Y^{(0)}_{m}(x;x')}-1\right|\leq C_{m,\theta} |\hbar|^{m+1}, \quad x \in \gamma_{0,2} \cup \gamma_{0,-2} \cup \gamma_{0,-1},
 \hbar \in S_{\theta,\hbar_{\theta}} .
\end{equation}
The same hold for the subdominant solutions in the Sector $-2$. There are
are positive constants $\bar{C}_{m,\theta}$ and a subdominant solution $y_{-2}(x)$ such that
\begin{equation}\label{eq:y-2proof}
 \left|\frac{y_{-2}(x)}{Y^{(-2)}_{m}(x;x')}-1\right|\leq \bar{C}_{m,\theta} |\hbar|^{m+1},
 \quad x \in \gamma_{-2,0} \cup \gamma_{-2,-1} \cup \gamma_{-2,2},
 \hbar \in S_{\theta,\hbar_{\theta}} .
\end{equation}
where
\begin{equation}\label{eq:Ym-2proof}
 Y^{(-2)}_{m}(x;x')=\exp{\left\lbrace \hbar^{-1} \sum_{k=0}^{m+1} -\int_{x'',\gamma_{-2,0}}^{x'} \hbar^k \alpha_k(s) ds 
+ \int_{x'',\gamma}^x \hbar^k \alpha_k(s) ds \right\rbrace} , \, \forall x \in
 \gamma_{0,2} \cup \gamma_{0,-2} \cup \gamma_{0,-1}.
\end{equation}
In the above formula the integration path $\gamma$ is -depending on $x$ - $\gamma_{-2,0}$ or $\gamma_{-2,-1}$ or $\gamma_{-2,2}$,
and the suffix $(2)$ stands to remind that the branch of $\sqrt{Q_0}$
is chosen in such a way that  $\lim_{t\to 0}\Re \int_{t}^{t_1} \sqrt{Q_0(\gamma_{-2,0}(t))}\dot{\gamma}_{-2,0}(t) dt =\infty$.

\begin{figure}[htbp]
\begin{center}
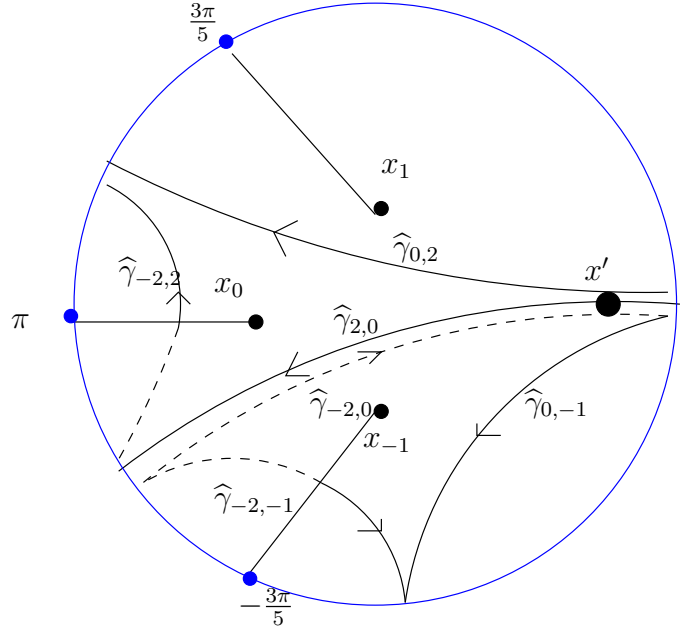
\end{center}
\caption{The \textit{Bacalhau (cod) diagram}. Integration paths for the approximate functions $Y_0,Y_{\pm2}$.}
\label{fig:coddiagram}
\end{figure}
As it was explained in the Section \ref{sec:Xs}, the choice of the branch of $\sqrt{Q_0}$ can be enforced by lifting the integration paths
to $X_s$, which we described as a two-sheeted covering of the Riemann sphere.\footnote{See Figure \ref{fig:doublesheeted} above.  Recall: the lower sheet is the one such that $\lim_{x \to+\infty}\Re \sqrt{Q_0(x)}=+\infty$}
The lift is defined as follows: the lift of $\gamma_{0,2}, \gamma_{0,-2} ,\gamma_{0,-1}$
belongs to the upper (solid) sheet for $x \to +\infty$,
the lift of $\gamma_{-2,0},\gamma_{-2,-1},\gamma_{-2,2}$ belongs to the lower (dashed) sheet as $x \to e^{-\frac{4\pi}{5}i} \infty $.
Denoting by $\widehat{\gamma}_{k,k'}$ the lift of $\gamma_{k,k'}$, for any of the paths introduced, the situation is as
illustrated in the \textit{Bacalhau Diagram}, Figure \ref{fig:coddiagram}.

\subsubsection{Computation of the Fock-Goncharov co-ordinates in WKB approximation}
We can compute $a_{-1}(\hbar)$ in the WKB approximation using formulas (\ref{eq:akratio},\ref{eq:Ymproof},\ref{eq:y0proof},\ref{eq:y-2proof},
\ref{eq:Ym-2proof}).
After formula (\ref{eq:akratio}), we have
\begin{eqnarray}\nonumber
a_{-1}(\hbar)=(-1)^{s_{-}}
\lim_{t\to 1}\frac{y_0(\gamma_{0,-1}(t))}{y_{-2}(\gamma_{-2,-1}(t))} , \quad 
s_{-}=s(\gamma_{0,-1}-\gamma_{-2,-1}+\gamma_{-2,0}).
\end{eqnarray}

Defining
\begin{equation}\label{eq:epsilon-}
\epsilon_{-}(\hbar):=\lim_{t \to 1}\frac{y_0(\gamma_{0,-1}(t))}{Y^{(0)}_m(\widehat{\gamma}_{0,-1}(t);x')}
\frac{Y^{(-2)}_m(\widehat{\gamma}_{-2,-1}(t);x')}{y_{-2}(\gamma_{-2,-1}(t))}-1 ,
\end{equation}
we obtain
\begin{eqnarray*}
\lim_{t\to 1}\frac{y_0(\gamma_{0,-1}(t))}{y_{-2}(\gamma_{-2,-1}(t))}= \left(
\lim_{t\to 1}\frac{Y^{(0)}_m(\widehat{\gamma}_{0,-1}(t);x')}{Y^{(-2)}_m(\widehat{\gamma}_{-2,-1}(t);x')}\right)
\big(1+\epsilon_{-}(\hbar)\big) 
\end{eqnarray*}
After formulae (\ref{eq:Ymproof},\ref{eq:Ym-2proof}), we have that
\begin{equation}\label{eq:a-1}
 \lim_{t\to 1}\frac{Y_0(\gamma_{0,-1}(t);x')}{Y_{-2}(\gamma_{-2,-1}(t);x')}=
\exp\left(\hbar^{-1}\sum_{k=0}^{m+1} \hbar^{k} \int_{\gamma_2^-}\alpha_k(x) dx\right), 
\end{equation}
where $\gamma_2^-$ is the lift -which is not closed- of the closed path $\gamma_{0,-1}-\gamma_{-2,-1}+\gamma_{-2,0}$,
as depicted in Figure \ref{fig:gammasplit}. Finally,
after (\ref{eq:y0proof},\ref{eq:y-2proof}), we have that there exists a sequence of positive constants $C_{m,\theta}$ such that
\begin{equation*}
 \left|\epsilon_{-}(\hbar)\right|\leq C^-_{m,\theta} |\hbar|^{m+1} , \quad  \forall \hbar \in S_{\theta,\hbar_{\theta}}  ,
\end{equation*}
where $\epsilon_{-}(\hbar)$ is the constant defined in \eqref{eq:epsilon-}.

\begin{figure}[htbp]
\begin{center}
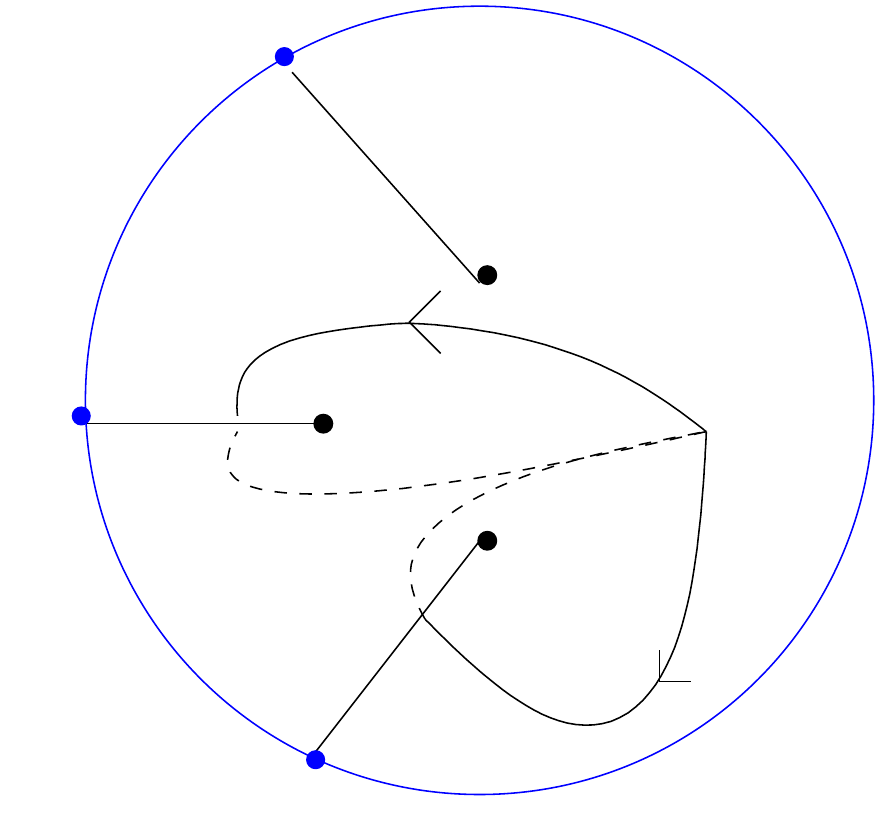
\end{center}
\caption{The paths used in formulas (\ref{eq:a-1},\ref{eq:a-2})}
\label{fig:gammasplit}
\end{figure}

We can use the same strategy to compute $a_2(\hbar)$ to obtain the following statement:
There exists a sequence of constants $\hbar_{\theta},C^+_{m,\theta}$ such that
\begin{equation}\label{eq:a-2}
 (-1)^{s_+} a_2(\hbar) \exp\left(-\hbar^{-1}\sum_{k=0}^{m+1} \hbar^{k} \int_{\gamma_2^+}\alpha_k(x) dx\right) =\big(1+\epsilon_{+}(\hbar)\big),
 \mbox{ with } \left| \epsilon_{+}\right| \leq C_{m,\theta}^+ |\hbar|^{m+1}, \; \forall \hbar \in S_{\theta,\hbar_{\theta}}.
\end{equation}
Here $s_+=s(\gamma_{0,2}-\gamma_{-2,2}+\gamma_{-2,0})$ and
$\gamma_2^+$ is the lift of $\gamma_{0,2}-\gamma_{-2,2}+\gamma_{-2,0}$, as depicted in Figure \ref{fig:gammasplit}.

We notice that
\begin{equation}\label{eq:homologcom}
 \gamma_2^+-\gamma_2^-=-\gamma_2 \mbox{ in } H_1(X_s^{\circ}\setminus B,\mathbb{Z})^-
\end{equation}
where $\gamma_2^+,\gamma_2^-$ are the curves defined in Figure \ref{fig:gammasplit} and $\gamma_2$ is the basis element
of $H_1(X_s^{\circ}\setminus B,\mathbb{Z})^-$ defined in Figure \ref{fig:doublesheeted}.
Combining \eqref{eq:a-1}, \eqref{eq:a-2}, and \eqref{eq:homologcom}, we obtain the following result:
For every $\theta \in[0,\frac\pi2$, there exist $\hbar_\theta>0$ and a sequence of positive constants $C_{m,\theta}>0,m\geq 0$ such that
\begin{equation}\label{eq:almost1}
 X_2(\hbar) e^{\hbar^{-1}\int_{\gamma_2}\sqrt{Q_0(x)}dx}= - (-1)^{(s_++s_-)} \big(1+ \epsilon_2(\hbar)\big)
 \exp\left(-\hbar^{-1}\sum_{k=1}^{m+1} \hbar^{k} \int_{\gamma_2}\alpha_k(x) dx\right),
\end{equation}
where $|\epsilon_{2}(\hbar)|\leq C_m |\hbar|^{m+1}$, for all $\hbar \in S_{\theta,\hbar_{\theta}}$.

We are left to show that 
equation \eqref{eq:almost1} is equivalent to equation \eqref{eq:thesis} (for the index $i=2$).
Comparing the two equations, we see that they are equivalent if and only if
\begin{equation}
 -(-1)^{(s_++s_-)}\exp\left( -\int_{\gamma_2}\alpha_1(x) dx\right)=\exp\left(- \int_{\gamma_2}\widetilde{\alpha}_1(x)+\frac{1}{2(x-q)} dx\right)
\end{equation}
where $\widetilde{\alpha}_1(x)=\alpha_1(x)+\frac{Q'_0(x)}{4Q_0(x)}$ as per \eqref{eq:alphaforms}.
This is indeed the case. In fact, by the residue theorem we have that
$\int_{\gamma_2}\frac{dx}{2(x-q)}=i \pi \sigma$ where $\sigma$ is the winding number of the projection of $\gamma_2$ around $q$,
and
$\int_{\gamma_2}\frac{Q'(x)}{4 Q(x)} dx=-i \pi$. 

\begin{remark}
The co-ordinates $X_1,X_2$ are strictly related to the Stokes multipliers of the cubic oscillator. These are defined as follows:
For every $k \in \mathbb{Z}/5\mathbb{Z}$ one chooses a normalisation of
the subdominant solutions $y_k,y_{k\pm1}$ of equations \eqref{eq:potApp}, see \cite{Mas1} for the precise definition,
in such a way that
\begin{equation*}
 y_{k+1}(x)=y_{k-1}(x)+\sigma_k \, y_k(x)
\end{equation*}
for some uniquely defined $\sigma_k \in \mathbb{C}$, which are the Stokes multipliers, .

It was proven in \cite[\S 2]{Mas1} that each Stokes multiplier can be expressed as the cross-ratio of $4$ asymptotic values, namely
\begin{equation}\label{eq:Stokes}
 \sigma_k=i \CR (a_{k-1}a_{k+1}, a_{k+2},a_{k-2}) .
\end{equation}

Now assume that the potential $Q_0$ is saddle-free. It follows that there is
a unique $l \in \mathbb{Z}/ 5\mathbb{Z}$ such that the sets of horizontal trajectories $\Gamma_{l,l\pm2}$ are not empty.
Comparing \eqref{eq:Stokes} with \eqref{cr} we obtain
\begin{equation}\label{eq:multipliersvscoordinates}
 X_1=  \big( -i \sigma_{l-1}\big)^{-1} \, , \quad  X_2=  -i \sigma_{l+1} .
\end{equation}

\end{remark}

\end{appendix}


\bibliographystyle{amsplain}

\begin{thebibliography}{100}

\bibitem{A} D. Allegretti, Stability conditions and cluster varieties from quivers of type A,  Advances in  Math. 337,  260--293 (2018).

\bibitem{AB} D. Allegretti and T. Bridgeland, The monodromy of meromorphic projective structures, preprint arXiv:1802.02505.

\bibitem{Barbieri} A. Barbieri, A Riemann-Hilbert problem for uncoupled BPS structures, Manuscripta Math. 162, 1--21 (2020).


\bibitem{RHDT} T. Bridgeland, Riemann-Hilbert problems from Donaldson-Thomas theory, Invent. Math. 216 (2019), no. 1, 69--124.

\bibitem{RHDT2} T. Bridgeland, Geometry from Donaldson-Thomas invariants, preprint arXiv:1912.06504.


\bibitem{BQS} T. Bridgeland, Yu Qiu and Tom Sutherland, Stability conditions and the A$_2$ quiver, preprint arXiv:1406.2566.

\bibitem{BS} T. Bridgeland and I. Smith, Quadratic differentials as stability conditions, Publ. Math. Inst. Hautes {\'E}tudes Sci. 121 (2015), 155--278. 

\bibitem{copson} T. Copson, Metric spaces, Cambridge Tracts in Math., no. 57, 152 pp, Cambridge University Press (1968).

\bibitem{DDT} P. Dorey, C. Dunning and R. Tateo, 
The ODE/IM correspondence, J. of Physics A (40) 99 pp, (2007).

\bibitem{D1} B. Dubrovin, Geometry of 2D topological field theories, Integrable systems and quantum groups (Montecatini Terme, 1993), 120--348, Lecture Notes in Math., 1620, Springer (1996). 

\bibitem{D2} B. Dubrovin, Painlev{\'e} transcendents in two-dimensional topological field theory, The Painlev{\'e} property, 287--412, CRM Ser. Math. Phys., Springer (1999).

\bibitem{erdelyi10} A. Erdelyi, Asymptotic expansions, Dover Publications, 108 pp.  (1956).

\bibitem{FG} V. Fock and A. Goncharov, Moduli spaces of local systems and higher Teichm{\"u}ller theory, Publ. Math. Inst. Hautes {\'E}tudes Sci. 103, 1--211 (2006). 

\bibitem{G} D. Gaiotto, Opers and TBA, arxiv 1403.6137.


\bibitem{GMN1} D. Gaiotto, G. Moore and A. Neitzke, Four-dimensional wall-crossing via three-dimensional field theory. Comm. Math. Phys. 299,  163--224 (2013).

\bibitem{GMN2} D. Gaiotto, G. Moore and A. Neitzke, Wall-crossing, Hitchin systems, and the WKB approximation. Adv. Math. 234, 239--403 (2013).

\bibitem{Go} A. Goncharov, Pentagon relation for the quantum dilogarithm and quantized $\cM^{cyc}_{0,5}$, Geometry and dynamics of groups and spaces, 
Progr. Math., 265, 415--428, Birkh{\"a}user (2008).



\bibitem{HS} P. Hsieh and Y. Sibuya, On the asymptotic integration of second order linear ordinary differential equations with polynomial coefficients, J. of math. analysis and appl.,
  16, 84--103 (1966). 


\bibitem{JS} D. Joyce and Y. Song, A theory of generalized Donaldson-Thomas invariants, Mem. Amer. Math. Soc. 217,  1020, 199 pp (2012)

\bibitem{Kap} A. A. Kapaev. Quasi-linear Stokes phenomenon for the Painlev{\'e} first equation. Journal of Physics A, 37:11149, 2004.

\bibitem{KT} T. Kawai and Y. Takei, Algebraic analysis of singular perturbation theory. Transl. from the 1998 Japanese original by Goro Kato. Transl. of Math. Monographs, 227, Amer. Math. Soc., 129 pp. (2005).

\bibitem{KS} M. Kontsevich and Y. Soibelman, Stability structures, motivic Donaldson-Thomas invariants and cluster transformations, preprint arXiv:0811.2435.

\bibitem{LM} F. Loray and D. Mar{\'i}n, Projective structures and projective bundles over compact Riemann surfaces, Ast{\'e}risque 323, 223--252 (2009). 

\bibitem{Mar} K. Ito, M. Mari{\~n}o and H. Shu, TBA equations and resurgent quantum mechanics, arxiv 1811.04812.

\bibitem{Mas1} D. Masoero, Y-system and deformed thermodynamic Bethe ansatz, Lett. Math. Phys. 94 2, 151--164 (2010).

\bibitem{piwkb} D. Masoero, Poles of int{\'e}grale tritronqu{\'e}e and anharmonic oscillators. A WKB approach, J. Phys. A 43 , no. 9, 28 pp. (2010).


\bibitem{myphd} D. Masoero, Essays on the first Painlev{\'e} equation and the cubic oscillator, PhD thesis,  Scuola Internazionale Superiore di Studi Avanzati (SISSA).   

\bibitem{Sib} Y. Sibuya, Global theory of a second order linear ordinary differential equation with a polynomial coefficient, North-Holland Math. Studies 18, 307 pp (1975). 

\bibitem{strebel} K. Strebel, Quadratic differentials,  Ergebnisse der Math. (3), Springer-Verlag,   184 pp. (1984)

\bibitem{Ueno} K. Ueno, Monodromy preserving deformations of linear differential equations with irregular singular points, Proc. Japan Acad., 56, Ser. A (1980).

\bibitem{WW}  E.T. Whittaker and G.N. Watson, A course of modern analysis. An introduction to the general theory of infinite processes and of analytic functions; with an account of the principal transcendental functions. Reprint of the fourth (1927) edition. Cambridge University Press, 1996. 608 pp.                                            
\end{thebibliography}

\end{document}